\documentclass[journal,onecolumn,draftclsnofoot,]{IEEEtran}
\usepackage{setspace}
\usepackage{enumerate}
\usepackage[dvipsnames]{xcolor}
\usepackage{url}
\usepackage{graphicx}
\usepackage{epstopdf}
\usepackage{csquotes}
\usepackage{subfig}
\usepackage{mathtools}
\usepackage{amssymb, amsmath,amsthm, amsfonts}
\usepackage{ifthen}
\usepackage{tikz}
\usepackage{cite}
\usepackage{enumerate}
\usepackage{verbatim}
\usepackage{pifont}
\usepackage{bm}  
\usepackage{stackengine}
\usepackage{bbm}

\allowdisplaybreaks
\usepackage{accents}
\patchcmd{\subsubsection}{\itshape}{\bfseries}{}{}
\newcommand\numberthis{\addtocounter{equation}{1}\tag{\theequation}}
\usepackage{xpatch}
\makeatletter
\AtBeginDocument{\xpatchcmd{\@thm}{\thm@headpunct{.}}{\thm@headpunct{}}{}{}}
\makeatother

\DeclareSymbolFont{bbold}{U}{bbold}{m}{n}
\DeclareSymbolFontAlphabet{\mathbbold}{bbold}

\newtheorem{theorem}{Theorem}
\newtheorem{cor}{Corollary}
\newtheorem{lemma}{Lemma}
\newtheorem{prop}{Proposition}

\newtheorem{remark}{Remark}

\newtheorem{definition}{Definition}
\newtheorem{example}{Example}

\newtheorem{assump}{Assumption}

\newcommand{\norm}[1]{\left\Vert#1\right\Vert}
\newcommand{\abs}[1]{\left\vert#1\right\vert}

\newcommand{\cA}{\mathcal{A}}
\newcommand{\cB}{\mathcal{B}}
\newcommand{\cC}{\mathcal{C}}
\newcommand{\cD}{\mathcal{D}}
\newcommand{\cE}{\mathcal{E}}

\newcommand{\cI}{\mathcal{I}}

\newcommand{\cN}{\mathcal{N}}

\newcommand{\cP}{\mathcal{P}}
\newcommand{\cQ}{\mathcal{Q}}

\newcommand{\cS}{\mathcal{S}}
\newcommand{\cT}{\mathcal{T}}

\newcommand{\cX}{\mathcal{X}}

\newcommand{\EE}{\mathbb{E}}

\newcommand{\NN}{\mathbb{N}}

\newcommand{\PP}{\mathbb{P}}

\newcommand{\RR}{\mathbb{R}}

\newcommand{\E}{\mathbb E}

\newcommand{\supp}{\mathrm{spt}}

\newcommand{\ind}{\mathbbm 1}

\newcommand{\kl}[2]{\mathsf{D}_\mathsf{KL}\left(#1\middle\|#2\right)}

\newcommand{\tv}[2]{\left\|#1-#2\right\|_\mathsf{TV}}
\newcommand{\chisq}[2]{\chi^2\left(#1\middle\|#2\right)}
\newcommand{\helsq}[2]{\mathsf{H}^2\left(#1,#2\right)}
\newcommand{\fdiv}[2]{\mathsf{D}_{f}\left(#1\middle\|#2\right)}

\newcommand{\trightarrow}[1]{\overset{#1}{\longrightarrow}}

\newcommand\blfootnote[1]{%
	\begingroup
	\renewcommand\thefootnote{}\footnote{#1}%
	\addtocounter{footnote}{-1}%
	\endgroup
}

\definecolor{cblue}{rgb}{0.16, 0.32, 0.75}

\def\h2{\tilde h}

\def\hm1{\hat h_{-1}}

\begin{document}
\title{Limit Distribution Theory for $f$-divergences}
\author{Sreejith Sreekumar, Ziv Goldfeld,  Kengo Kato}
\maketitle
\begin{abstract}
 $f$-divergences, which quantify discrepancy between probability distributions, are ubiquitous in information theory, machine learning, and statistics. While there are numerous methods for estimating $f$-divergences from data, a limit distribution theory, which quantifies fluctuations of the estimation error, is largely obscure. As limit theorems are pivotal for valid statistical inference, to close this gap, we develop a general methodology for deriving distributional limits for $f$-divergences based on the functional delta method and Hadamard directional differentiability. Focusing on four prominent $f$-divergences---Kullback-Leibler divergence, $\chi^2$ divergence, squared Hellinger distance, and total variation distance---we identify sufficient conditions on the population distributions for the existence of distributional limits and characterize the limiting variables. These results are used to derive one- and two-sample limit theorems for Gaussian-smoothed $f$-divergences, both under the null and the alternative. Finally, an application of the limit distribution theory to auditing  differential privacy is proposed and analyzed for significance level and power against local alternatives. 
\end{abstract}


\section{Introduction}
\blfootnote{S. Sreekumar is with the Institute for Quantum Information at RWTH Aachen University, Aachen 52074, Germany, (email: sreekumar@physik.rwth-aachen.de). He was with the School of Electrical and Computer Engineering, Cornell University, Ithaca, NY 14850, USA, at the time of this work. Z. Goldfeld and K. Kato are with the School of Electrical and Computer Engineering and Department of Statistics and Data Science, respectively, at Cornell University (e-mail: goldfeld@cornell.edu; kk976@cornell.edu). S. Sreekumar was partially supported by the TRIPODS Center for Data Science NSF grant CCF-1740822. Z. Goldfeld is supported by  NSF grants CCF-1947801,  CCF-2046018, DMS-2210368, and the 2020 IBM Academic Award. K. Kato is supported by the NSF grants DMS-1952306, DMS-2014636, and DMS-2210368. Parts of this work were presented at the International Symposium on Information Theory (ISIT) in 2020 \cite{Goldfeld-Kato-2020} and 2023 \cite{SGK-ISIT-2023}.}

Statistical inference often boils down to estimation of certain functional of underlying probability measures. Discrepancy measures between probability distributions, also known as statistical divergences, such as $f$-divergences \cite{csiszar1967information,Ali-Silvey-1966}, R\'{e}nyi divergences \cite{Renyi-60,vanEvren_Reyni_Div2014}, integral probability metrics \cite{zolotarev1983probability,muller1997integral}, Wasserstein distances \cite{villani2008optimal,santambrogio2015}, etc., form an important class of such functionals. They play a fundamental role in information theory, signal processing, and statistics, with some arising naturally as operational quantities characterizing the fundamental limits of data compression, hypothesis testing, and communication \cite{CoverThomas,Csiszar-Shields-2004}. Moreover, statistical divergences are potent tools for modeling, analysis, and design of machine learning algorithms, encompassing generative modeling \cite{kingma2013auto,nowozin2016f,arjovsky2017wasserstein,tolstikhin2018wasserstein,Goldfeld2020limit_wass},  homogeneity/goodness-of-fit/independence testing \cite{Kac-1955,QZhang-2018,Hallin-2021}, anomaly detection\cite{Afgani-2008,Tajer-2011}, to name a few. 
 
In data-driven applications, one only has samples from the population distributions, which necessitates estimating $f$-divergences. While there is an abundance of consistent estimators with known convergence rates (see the literature review in Section \ref{subsec:related_work}), a limit distribution theory for the empirical estimation error has remained partial and premature. For $\mu,\nu\in\cP(\RR^d)$ and an $f$-divergence $\mathsf{D}_f(\cdot\|\cdot)$, limit theorems seek to identify the scaling rate $r_n\to\infty$ and the limiting variable $G$, such that the following convergence in distribution holds\footnote{The two-sample problem, i.e., when the first divergence term is $\mathsf{D}_f(\mu_n\|\nu_n)$, is also of interest.}
\begin{equation}
    r_n\big(\mathsf{D}_f(\mu_n\|\nu)-\mathsf{D}_f(\mu\|\nu)\big) \trightarrow{d} G,\label{eq:limit_thm_intro}
\end{equation}
where $\mu_n$ is an estimate of $\mu$ from samples. As such, these results characterize the probability laws governing the random fluctuations of the error and serve as a central constituent for valid statistical inference. Indeed, distributional limits enable constructing confidence intervals, devising consistent resampling methods, proving guarantees for applications of hypothesis testing, and more.

To address the aforementioned gap, we develop a unified methodology for deriving limit distributions for $f$-divergences under general regularity conditions. Our approach relies on the functional delta method over normed vector spaces \cite{Shapiro-1990,Romisch-2004} and Hadamard directional differentiability of $f$-divergence over a certain class of probability distributions. The Hadamard differentiability analysis captures how the $f$-divergence functional changes due to small perturbations of the considered distributions within the said class. However, $f$-divergence functionals (e.g., KL divergence) are generally non-smooth and highly sensitive to support mismatch, which may cause them to degenerate or even blow up. This irregular behaviour also carries over to their derivatives. It is therefore pivotal to identify appropriate regularity conditions under which the Hadamard directional derivatives and the corresponding distributional limits exist and can be characterized. In particular, there is a trade-off between how strict the imposed regularity is and the class of distributions that the theory accounts for. Consequently, a key technical challenge is to discern the right normed space in which the densities of the considered distributions should reside, so as to obtain a limit distribution theory that accounts for the 
largest possible class of distributions.    

Existing approaches for deriving limit distributions for $f$-divergences are mostly limited to discrete distributions or compactly supported continuous distributions with smooth densities bounded away from zero, in which case the $f$-divergence functional as well as its derivatives become smooth.  
Our approach disposes of such restrictive assumptions and extends these results to the general case. We leverage the Taylor expansion of the considered $f$-divergences to ascertain minimal primitive regularity conditions on the population distributions that guarantee the existence of the Hadamard  derivative. In particular, we identify a certain $L^2(\eta)$ space, where the measure $\eta$ is defined in terms of the populations $(\mu,\nu)$, in which the Hadamard directional derivatives exist and can be characterized.  
Having that, the functional delta method enables lifting weak convergence of the estimates of the underlying distributions to convergence of the $f$-divergence between them, with the limiting variable identified in terms of this derivative.

The general framework is instantiated to obtain the one- and two-sample distributional limits, under both the null ($\mu=\nu$) and the alternative ($\mu\neq\nu$), of four popular $f$-divergences---Kullback-Leibler (KL) divergence, chi-squared ($\chi^2$) divergence, squared Hellinger ($\mathsf{H}^2$) distance, and total variation (TV) distance. These results hold under the high-level weak convergence assumptions on the empirical estimates of $\mu,\nu$ with a given scaling law $r_n$. To obtain limit theorems under basic conditions on the population distributions with explicit rates, we consider Gaussian-smoothed $f$-divergences, i.e., $\mathsf{D}_f(\mu*\gamma_\sigma\|\nu*\gamma_\sigma)$ where $\gamma_\sigma=\mathcal{N}(0,\sigma^2I_d)$, and estimate $\mu,\nu$ by the empirical measures $\hat{\mu}_n=n^{-1}\sum_{i=1}^n\delta_{X_i}$ and $\hat{\nu}_n=n^{-1}\sum_{i=1}^n\delta_{Y_i}$, respectively. Under this setup, we derive primitive conditions\footnote{The primitive conditions shown here are  sharp in the one-sample null case, see e.g., Proposition \ref{Prop:GS-KL-limdist}$(i)$.} on $\mu,\nu$ that guarantees weak convergence of the smooth empirical measures $\hat{\mu}_n*\gamma_\sigma,\hat{\nu}_n*\gamma_\sigma$, utilizing the central limit theorem (CLT) in $L^2$ spaces \cite[Proposition 2.1.11]{AVDV-book}. For  KL divergence, $\chi^2$ divergence, and $H^2$ distance under the null, we identify the scaling law as $r_n=n$ and the limiting variable as a weighted sum of independent and identically distributed (i.i.d.) $\chi^2$ random variables. Under the alternative, we show that $r_n=\sqrt{n}$ and the limit is a centered Gaussian. The TV distance behaves slightly differently, with $r_n=\sqrt{n}$ in both cases and the limiting variables having a certain integral form. By virtue of our Hadamard differentiability analysis, we automatically obtain consistency of the bootstrap, which yields a computationally tractable resampling method for estimating the  distributional limits.

As an application of our limit distribution theory, we consider auditing  $\epsilon$-differential privacy (DP). An audit of a black-box privacy mechanism seeks to certify whether it satisfies a promised DP guarantee. While existing auditing methods are heuristic \cite{ding2018detecting,jagielski2020auditing} or lack in formal guarantees \cite{domingo2022auditing}, we propose a principled hypothesis testing pipeline for DP auditing with a full (asymptotic) analysis of  significance level and power against local alternatives. The key idea is to relax the $\epsilon$-DP constraint\footnote{$\epsilon$-DP corresponds to an $\epsilon$ bound on the infinite order R\'{e}nyi divergence between the output distribution of the mechanism applied to two neighboring databases.} to a KL divergence bound, which is further relaxed to the Gaussian-smoothed KL divergence via the data-processing inequality\cite[Theorem 2.15]{polyanskiy2022codingbook}.  We then test for the  smooth KL divergence value 
 and leverage our limit theorems for the significance and power analysis. 
We also establish a stability lemma that bounds the gap due to smoothing, namely $\big|\kl{\mu*\gamma_\sigma}{\nu*\gamma_\sigma}-\kl{\mu}{\nu}\big|$. This enables lifting the audit to test for the  KL divergence value itself, 
for which we show that any non-zero significance level along with power 1 can be achieved asymptotically.

\subsection{Related Work}\label{subsec:related_work}
Statistical analysis of divergence estimators has been an active area of research in recent years. Convergence rates for various estimators, which subsumes entropy and mutual information as special cases, have been studied in \cite{Wang-2005,Perez-2008,kandasamy2015nonparametric, Singh-Poczos-2016,Noshad-2017,belghazi2018,Moon-2018,Nguyen-2010,berrett2019efficient,berrett2019-efficientfunctional,Yanjun-2020,SS-2021-aistats,sreekumar2021neural} (see also references therein). Literature on limit distributions for $f$-divergences mainly focused on analyzing specific estimators on a case-by-case basis. In \cite{SALICRU-1994}, limit distributions for $f$-divergences between maximum likelihood estimates of probability distributions over a certain parametric class is established, with the limit variable shown to be either normal or $\chi^2$. The authors of \cite{kandasamy2015nonparametric} study plugin methods of kernel density estimator and show asymptotic normality subject to high H\"{o}lder smoothness and compact support of the densities. The case when the density estimates are constructed using $k$-nearest neighbour technique is treated in \cite{Moon-2014}. One-sample null distributional limits of Gaussian-smoothed TV distance and $\chi^2$ divergence have been derived in \cite{Goldfeld-Kato-2020} by invoking the CLT in $L^1(\RR^d)$ and $L^2(\RR^d)$, respectively. Limit distributions for plug-in estimators of entropy and mutual information in the discrete setting have been  considered in \cite{Antos-Ioannis-2001,Ioannis-Skoularidou-2016}. Differing from these, the unified methodology developed herein enables obtaining distributional limits of $f$-divergences based on plugin of arbitrary estimators under general (oftentimes sufficient and necessary) conditions. In particular, our results subsume those of \cite{Goldfeld-Kato-2020} and considerably generalize the scope of this preliminary work via markedly different proof techniques.

Apart from the $f$-divergence class, limit distribution theory for other divergences such as Wasserstein distances has been extensively studied \cite{Barrio-1999,DBGU-2005,Sommerfeld2016Inference,del-barrio-loubes-2019,Tameling-2019,DBSL-21,delBarrio-2022,MBWW-2022,HKSM-22}. There is also a surge of interest in regularized optimal transport distances driven by computational and statistical gains, encompassing techniques like smoothing, slicing, and entropic penalization. Limit theorems under these frameworks for the optimal transport cost, plan, map, and dual potentials can be found in \cite{Bigot-2019,Mena-2019,Klatt-2020,HLP-20,Goldfeld2020limit_wass,GX-21,delBarrio-2022,Gonzalez-2022,Sadhu-2022,GKNR-smooth-p-2022,GKRS-2022-entropicOT,GKRS-22}.

\subsection{Organization}The rest of the paper is organized as follows. Section \ref{Sec:background}  collects  background material on Hadamard differentiability and the functional delta method. Section \ref{Sec:Techframework} formalizes the Hadamard differentiability framework for $f$-divergences. Section \ref{Sec:limittheorems} obtains limit distributions for $f$-divergences between random probability measures under general regularity conditions, with Section \ref{Sec:GS limit theorems} focusing specifically on Gaussian-smoothed divergences. Section \ref{Sec:appkldiffprivacy} studies the application of the limit distribution theory to auditing DP. Proofs are provided in Section \ref{Sec:proofs}, while Section \ref{Sec:conclusion} provides concluding remarks and discusses future research directions.

\section{Background and Preliminaries}\label{Sec:background}
\subsection{Notation}
Let $(\Omega, \cA,\PP)$ be a sufficiently rich probability space on which all random variables are defined. Let $(\mathfrak{S},\cS)$ be a separable measurable space equipped with a $\sigma$-finite measure~$\rho$. When $\mathfrak{S}$ is a topological space, we use $\cB(\mathfrak{S})$ to denote the Borel $\sigma$-field on $\mathfrak{S}$. In the sequel, we adapt $\rho$ on a case-by-case basis, but given $\rho$, all considered measures are assumed to be absolutely continuous w.r.t. it. For $\eta \ll \rho$, we write $p_{\eta}=d \eta/d \rho$ for the Radon-Nikodym derivative of $\eta$ w.r.t. $\rho$.  $\eta^{\otimes n}$ stands for the $n$-fold product measure, and $\delta_x$ represents the Dirac measure at $x$. $\ind_{\cE}$ denotes the indicator of an event $\cE$.

We use $\cP(\mathfrak{S})$ to denote the space of probability measures on $(\mathfrak{S},\cS)$, leaving the $\sigma$-field implicit. When $\mathfrak{S}=\RR^d$, we always take $\cS=\cB(\RR^d)$ and $\cP(\RR^d)$ as the set of Borel probability measures.   For $\mu,\nu \in \cP(\RR^d)$,  $\mu*\nu$ denotes the convolution of $\mu$ and $\nu$; likewise, $f*g$ represents convolution of two measurable functions $f,g: \RR^d \to \RR$. We write $\gamma_{\sigma} = N(0,\sigma^2 I_d)$ for the centered Gaussian distribution on $\RR^d$ with covariance matrix $\sigma^2 I_d$, and use $\varphi_{\sigma}(x) = (2\pi\sigma^2)^{-d/2}e^{-\norm{x}^2/(2\sigma^2)}$ ($x \in \RR^d$) for the corresponding density.  We say that $\mu \in \cP\big(\RR^d\big)$ is $\beta$-sub-Gaussian for $\beta \geq 0$, if $X \sim \mu$ satisfies $\EE \left[e^{\alpha \cdot(X-\EE[X])}\right] \leq \exp\big(\beta^2\norm{\alpha}^2/2\big)$, for all $\alpha \in \RR^d$. Let $\trightarrow{w}$ and $\trightarrow{d}$ denote weak convergence\footnote{A sequence of Borel measurable maps $X_n$ converges weakly to a Borel measurable map $X$ if $\EE[f(X_n)] \rightarrow \EE[f(X)]$ for all bounded continuous functions $f$. This is denoted by $X_n \trightarrow{w} X$. } of Borel measurable maps (or their laws) and convergence in distribution of random variables,  respectively.

For $1 \leq r \leq \infty$, let $L^r(\rho)=L^r(\mathfrak{S},\cS,\rho)$ be the space of all real-valued measurable functions $f$ on $\mathfrak{S}$ such that $\norm{f}_{r,\rho}:=\big(\int_{\mathfrak{S}} \abs{f}^r d \rho\big)^{1/r} <\infty$, with the usual identification of functions that are equal $\rho$-almost everywhere (a.e.). For $1 \leq r < \infty$, the space $(L^r,\norm{\cdot}_{r,\rho})$ is a separable Banach (and hence, Polish) space.  When $\rho$ is the Lebesgue measure
$\lambda$ on $\RR^d$, we use  $\| \cdot \|_{r}$ to denote the corresponding $L^r$ norm. $\|\cdot\|$ designates Euclidean norm. $L_+^r(\rho)$ denotes the subset of positive functions in $L^r(\rho)$.  The support of a measurable function $f: \mathfrak{S} \rightarrow \bar{\RR}$, i.e., $\{s \in \mathfrak{S}: f(s) \neq 0\}$ is denoted by $\supp(f)$. 
For a multi-index $\alpha=(\alpha_1,\dots,\alpha_d) \in \mathbb{N}_0^d$ with $|\alpha| = \sum_{j=1}^ d \alpha_j$ ($\NN_0 = \NN \cup \{ 0 \}$), we use  $D^\alpha$ for the differential operator  $D^\alpha = \frac{\partial^{|\alpha|}}{\partial x_1^{\alpha_1} \cdots \partial x_{d}^{\alpha_d}}$ with $D^0 f = f$, and employ the shorthand $x^{\alpha}=\prod_{i=1}^d x_i^{\alpha_i}$. The shorthand $a\lesssim b$  designates $a\leq cb$ for a universal constant~$c>0$. The values of constants 
may change from line to line of a certain derivation. Lastly, we adopt the convention $0/0=0$, $c/0=\infty$ for $c>0$, $\infty \cdot 0=0$,  and $0 \log (c/0)=0$ for $c \geq 0$.

\subsection{{$f$}-Divergences}

$f$-divergences form a broad class of discrepancy measures between probability distributions, as defined next.
\begin{definition}[$f$-divergence]\label{def:f divergence}
Let $f:[0,\infty] \rightarrow (-\infty,\infty]$ be a convex function such that $f(1)=0$ and $f(0)= f(0^+)$. For $\mu,\nu \in \cP(\mathfrak{S})$, the $f$-divergence of $\mu$ from  $\nu$ is 
\begin{align}
\fdiv{\mu}{\nu}:=\int_{\mathfrak{S}} f \circ \left(\frac{p_{\mu}}{p_{\nu}}\right) d\nu. \label{eq:fdivdef}
\end{align}
\end{definition}

The class of $f$-divergences includes several popular dissimilarity measures, such as KL divergence, $\chi^2$ divergence, $\mathsf{H}^2$ distance, TV distance, and many more. Every $f$-divergence satisfies nonnegativity $(\fdiv{\mu}{\nu} \geq 0$, $\forall \mu,\nu$) and  joint convexity in the pair $(\mu,\nu)$, with additional properties holding for specific instances, as mentioned below.

\subsubsection{KL divergence}
Setting $f(x)=f_{\mathsf{KL}}(x):=x \log x$ in \eqref{eq:fdivdef} yields KL divergence,  $\kl{\mu}{\nu}:=\int_{\mathfrak{S}}  \log\left(p_{\mu}/p_{\nu} \right)d\mu$ for $\mu\ll \nu $, and $\kl{\mu}{\nu}=\infty$, otherwise.

\subsubsection{$\bm{\chi^2}$ divergence}
Setting $f(x)=f_{\chi^2}(x):=(x-1)^2$ in \eqref{eq:fdivdef} leads to  $\chi^2$ divergence,  $\chisq{\mu}{\nu}:=\int_{\mathfrak{S}} (p_{\mu}/p_{\nu}-1)^2d\nu$ for $\mu\ll \nu$, and $\chisq{\mu}{\nu}=\infty$, otherwise.

\subsubsection{$\bm{\mathsf{H}^2}$ distance}
Setting $f(x)=f_{\mathsf{H}^2}(x):=(\sqrt{x}-1)^2$ in \eqref{eq:fdivdef} leads to  $\mathsf{H}^2$ distance,  $\helsq{\mu}{\nu}:=\int_{\mathfrak{S}} \big(\sqrt{p_{\mu}}-\sqrt{p_{\nu}}\big)^2 d\rho$. $\mathsf{H}^2$ distance is symmetric in its arguments, and $0 \leq \helsq{\mu}{\nu} \leq 2$ for all $\mu,\nu \in \cP(\cS)$.

\subsubsection{TV distance}
Setting $f(x)\mspace{-2 mu}=\mspace{-2 mu}f_{\mathsf{TV}}(x)\mspace{-2 mu}:=\mspace{-2 mu}\abs{x-1}/2$ in \eqref{eq:fdivdef} yields the  TV distance, $\tv{\mu}{\nu}\mspace{-3 mu}:= \mspace{-3 mu} \frac 12\int_{\mathfrak{S}} \abs{p_{\mu}-p_{\nu}} d\rho$ $=\sup_{A \in \cS} \abs{\mu(A)-\nu(A)} $. TV distance is symmetric in its arguments, and $0 \leq \tv{\mu}{\nu} \leq 1$ for all $\mu,\nu \in \cP(\cS)$.

\subsection{Functional delta method} We derive  limit distribution theory for $f$-divergences via the functional delta method, which relies on the concept of Hadamard directional differentiability. The ideas are introduced below.
\begin{definition}[Hadamard directional differentiability \cite{Shapiro-1990,Romisch-2004,AVDV-book}]
Let $\mathfrak{D}$ and $\mathfrak{E}$ be linear topological spaces, and consider a function $\Phi:\Theta\subseteq \mathfrak{D}\to \mathfrak{E}$.
\begin{enumerate}[(i)]
    \item $\Phi$ is first order Hadamard directionally differentiable at $\theta\in \Theta$ tangentially to $\Theta$ if there exists a map $\Phi_{\theta}':\mathfrak{T}_{\Theta}(\theta) \rightarrow \mathfrak{E}$ such that 
 \begin{align}
     \lim_{n \rightarrow \infty} \frac{\Phi(\theta+t_nh_n)-\Phi(\theta)}{t_n}=\Phi_{\theta}'(h), 
 \end{align}
for any $h \in \mathfrak{T}_{\Theta}(\theta)$, $t_n \downarrow 0^+$ and $h_n \rightarrow h$ in $\mathfrak{D}$ such that $\theta+t_nh_n \in \Theta$, where 
\[
     \mathfrak{T}_{\Theta}(\theta):=\bigg\{h \in \mathfrak{D}:h=\lim_{n \rightarrow \infty} \frac{\theta_n-\theta}{t_n} \mbox{ for some }\theta_n \rightarrow \theta \mbox{ with } \theta_n \in \Theta \mbox{ and }t_n \downarrow 0^+\bigg\}
\]
is the \textit{tangent cone} to $\Theta$ at $\theta$.

\item For convex $\Theta$,\footnote{When $\Theta$ is convex, we have $\mathfrak{T}_{\Theta}(\theta)=\mathrm{cl}\big(\big\{(\tilde \theta-\theta)/t:\tilde \theta \in \Theta, ~t>0\big\}\big)$  for all $\theta \in \mathrm{cl}(\Theta)$; see \cite{Romisch-2004}.} we say that $\Phi$ is second order Hadamard directionally  differentiable at $\theta \in \Theta$ tangentially to $\Theta$ if it is first order Hadamard directionally differentiable at $\theta$ and there exists a map $\Phi_{\theta}'':\mathfrak{T}_{\Theta}(\theta) \rightarrow \mathfrak{E}$ such that 
 \begin{align}
    \Phi_{\theta}''(h)= \lim_{n \rightarrow \infty} \frac{\Phi(\theta+t_nh_n)-\Phi(\theta)-t_n\Phi_{\theta}'(h_n)}{\frac{1}{2}t_n^2}, 
 \end{align}
 for any $t_n \downarrow 0^+$ and $h_n \rightarrow h$ in $\mathfrak{D}$ (note: the convexity of $\Theta$ implies that $h_n\in\mathfrak{T}_\Theta(\theta)$, so that $\Phi'_\theta(h_n)$ is well-defined).
\end{enumerate}
The functions $ \Phi_{\theta}'$ and $\Phi_{\theta}''$ are called the first and second order Hadamard directional derivatives of $\Phi$ at $\theta$.
\end{definition}

For Hadamard directionally differentiable maps we have the following adaptation of the functional delta method \cite{Shapiro-1990,Romisch-2004,AVDV-book}. Let $\mathsf{int}(\cC)$ and $\partial \cC$ denote the interior and  boundary of a set $\cC$, respectively.
\begin{lemma}[Functional delta method] \label{Lem:extfuncdelta}
Let $\mathfrak{D}$ and $\mathfrak{E}$ be metrizable linear topological spaces,  $\Theta  \subseteq  \mathfrak{D}_{\Phi} \subseteq \mathfrak{D}$, and $\Phi:\mathfrak{D}_{\Phi}  \rightarrow \mathfrak{E}$ be Hadamard directionally differentiable at $\theta \in \Theta $ tangentially to $\Theta$ with derivative~$\Phi_{\theta}': \mathfrak{T}_{\Theta}(\theta) \rightarrow \mathfrak{E}$. Let $Z_n:\Omega \rightarrow \mathfrak{D}_{\Phi}$, $n\in\NN$,  be measurable maps such that  $r_n(Z_n-\theta) \trightarrow{w} Z$ for some sequence $r_n \rightarrow \infty$, where $Z$ is a random variable  that takes values in $\mathfrak{T}_{\Theta}(\theta)$.
\begin{enumerate}[(i)]
    \item If $\mathfrak{D}_{\Phi}=\Theta$, then  we have 
    \[r_n\big(\Phi(Z_n)-\Phi(\theta)\big) \trightarrow{d} \Phi_{\theta}'(Z),\]
    and if $\Theta$ is convex then $r_n\big(\Phi(Z_n)-\Phi(\theta)\big)=\phi'_\theta\big(r_n(Z_n-\theta)\big) +o_\PP(1)$. 
    \vspace{2mm}
    \item If $\mathfrak{D}_{\Phi}=\Theta$ is convex and $\Phi$ is also second order Hadamard directionally differentiable at $\theta \in \Theta$ tangentially to $\Theta$, with derivative $\Phi_{\theta}''$, then 
\[
    r_n^2\big(\Phi(Z_n)-\Phi(\theta)-\Phi_{\theta}'(Z_n-\theta)\big) \trightarrow{d} \frac{1}{2} \Phi_{\theta}''(Z),
\]
and $r_n^2\big(\Phi(Z_n)-\Phi(\theta)-\Phi_{\theta}'(Z_n-\theta)\big)=\frac 12 \Phi_{\theta}''\big(r_n(Z_n-\theta)\big)+o_\PP(1)$.
\item  For a convex $\Theta \subseteq \mathfrak{D}_{\Phi}$,  suppose that $Z_n$ and $Z$ satisfy $\PP^*(Z_n \notin \Theta) \rightarrow 0$ and $\PP\big(Z \in \partial \mathfrak{T}_{\Theta}(\theta)\big)=0$, where $\PP^*$ denotes outer probability (see \cite{AVDV-book}). Then, Part $(i)$ and $(ii)$ above holds. 
\end{enumerate}
\end{lemma}
Part $(i)$ and $(ii)$ above are standard in the literature, while Part $(iii)$ which incorporates probabilistic constraints is new and is useful for our purposes to derive limit distribution results under more general conditions. To establish this claim, we prove an adaptation of the extended continuous mapping theorem\cite[Theorem 1.11.1]{AVDV-book}  as stated next, which may be of independent interest.
\begin{theorem}[Generalized extended continuous mapping theorem]\label{thm:extcontmapthm}
 Let $\mathfrak{D}$ and $\mathfrak{E}$ be metrizable linear topological spaces, $\mathfrak{D}_n, \mathfrak{D}_g \subseteq \mathfrak{D}$, and $\mathfrak{D}_n^{\rightarrow}\subseteq \mathfrak{D}_n$, for all $n\in\NN$. Suppose  measurable functions $g_n:\mathfrak{D}_n \rightarrow \mathfrak E$ and  $g: \mathfrak{D}_g \rightarrow \mathfrak{E}$ satisfy the following: if $h_n \rightarrow h $ with $h_n \in  \mathfrak{D}_n^{\rightarrow}$ for all sufficiently large $n$ and $h \in \mathfrak{D}_g$, then $g_n(h_n) \rightarrow g(h)$. Let $H_n:\Omega \rightarrow \mathfrak{D}_n$ and $H:\Omega \rightarrow \mathfrak{D}_g$ be  measurable maps such that $H_n \trightarrow{w} H$,  $\PP(H \in \mathfrak{D}_*)=0$ and $\PP(H \in \mathfrak{D}_{\infty})=1$, where $\mathfrak{D}_{\infty}=\{h \in \mathfrak{D}: \exists  (h_n)_{n \in \NN},~ h_n \rightarrow h \mbox{ and } h_n \in\mathfrak{D}_n^{\rightarrow} ~\forall~n \mbox{ sufficiently large}\}$ and $\mathfrak{D}_* =\cap_{m=1}^{\infty} \overline{\cup_{n=m}^{\infty}  (\mathfrak{D} \setminus \mathfrak{D}_{n}^{\rightarrow})} $. Then,   $g_n(H_n) \trightarrow{w} g(H)$.
\end{theorem} 
 The proofs of Part $(iii)$ of Lemma \ref{Lem:extfuncdelta} and Theorem \ref{thm:extcontmapthm} are given in Appendices \ref{Lem:extfuncdelta-proof} and \ref{thm:extcontmapthm-proof}, respectively.

\section{Hadamard Differentiability Framework for $f$-Divergences} \label{Sec:Techframework}
We specialize the Hadamard differentiability framework to treat $f$-divergences. The framework is first described in abstract terms, after which we instantiate it to the case of KL divergence for concreteness (see Example \ref{ex:KL_HDD} below). The key idea is to find the right normed space over which $\mathsf{D}_f(\cdot\|\cdot)$ can be set up as a first and second order Hadamard directionally differentiable functional w.r.t. that norm. The construction is as follows.

\medskip

Let $\phi: \bar{\RR}_{\geq 0}\times \bar{\RR}_{\geq 0} \to \bar{\RR}$, where $\bar{\RR}$ is the extended reals and $\bar{\RR}_{\geq 0}$ is its  nonnegative~part.
 
\begin{assump} \label{Assump1}
$\phi(1,1)=0$, $\phi(0,0)=\lim_{y \downarrow 0} \phi(0,y)$, $\phi$ is continuous at $(0,c)$ for $c>0$, and all its partial derivatives of order two exists and are continuous in  $(0,\mspace{-2 mu}\infty) \mspace{-2 mu}\times \mspace{-2 mu}(0,\mspace{-2 mu}\infty)$,  possibly taking the values $\pm \infty$ only when at least one of its arguments is $0$ or $\infty$.
 \end{assump}
Let $g_1^{\star},g_2^{\star} \in  L_+^1(\rho)$ be such that   $\norm{g_1^{\star}}_{1,\rho} \vee \norm{g_2^{\star}}_{1,\rho} \leq \mspace{2 mu} 1$. 
For a multi-index $\alpha\in\NN_0^2$ with $|\alpha|=2$, let $\psi_{\alpha}:\RR_{\geq 0} \times \RR_{\geq 0} \rightarrow \RR_{\geq 0}$  be measurable functions, and  $\eta_1,\eta_2\ll \rho$  be positive measures on $\mathfrak{S}$ defined via their relative densities w.r.t. $\rho$ as
\begin{equation}\label{eq:normetadef}
\begin{split}
  p_{\eta_1}&:=1+ \big| \psi_{2,0}\circ(g_1^{\star},g_2^{\star})\big| +\big| \psi_{1,1}\circ(g_1^{\star},g_2^{\star})\big|,\\
    p_{\eta_2}&:=1+ \big| \psi_{0,2}\circ(g_1^{\star},g_2^{\star})\big| +\big| \psi_{1,1}\circ(g_1^{\star},g_2^{\star})\big|.
\end{split}
\end{equation}
Given $g_1^{\star}$, $g_2^{\star}$, and $\psi:=(\psi_{2,0},\psi_{0,2},\psi_{1,1})$ as above, we define the normed space
\begin{align}
     \mathfrak{D}_{g_1^{\star},g_2^{\star},\psi}:=\left\{(g_1-g_1^{\star},g_2-g_2^{\star}):\,g_1,g_2 \in L^{1}(\rho),~ \norm{(g_1-g_1^{\star},g_2-g_2^{\star})}_{\mathfrak{D}_{g_1^{\star},g_2^{\star},\psi}}\mspace{-2 mu}<\infty \right\}, \label{eq:normedspaceDdef}
\end{align}
where $\norm{(g,\tilde g)}_{\mathfrak{D}_{g_1^{\star},g_2^{\star},\psi}}:=\norm{g}_{2,\eta_1}+\norm{\tilde g}_{2,\eta_2}$. We henceforth use the shorthands $\mathfrak{D}$ and $\|\cdot\|_\mathfrak{D}$ for the space $\mathfrak{D}_{g_1^{\star},g_2^{\star},\psi}$ and its norm.

Setting $\Theta':=  \big\{(g_1-g_1^{\star},g_2-g_2^{\star}) \in \mathfrak{D}:g_1 \geq 0,g_2 \geq 0, \supp(g_1) \subseteq \supp(g_2),\,   \norm{g_1}_{1,\rho} \vee \norm{g_2}_{1,\rho} \leq 1\big\}$ and $\Theta'' :=\big\{(g_1-g_1^{\star},g_2-g_2^{\star}) \in \mathfrak{D}:g_1 \geq 0,g_2  \geq  0,\,   \norm{g_1}_{1,\rho} \vee \norm{g_2}_{1,\rho} \leq 1\big\}$, let $\Theta$ be a convex subset of $\Theta'$ or $\Theta''$ that  contains $(0,0)$, and consider the functional    
$\Phi:\Theta \rightarrow \bar{\RR}$ given by
\begin{align}
    \Phi(\theta_1,\theta_2):=\int_{\mathfrak{S}} \phi(g_1^{\star}+\theta_1,g_2^{\star}+\theta_2) d\rho.\label{eq:defPhifunc}
\end{align}
In addition to Assumption \ref{Assump1}, the following assumptions are needed to state our Hadamard differentiability result. 
\begin{assump} \label{Assump2}
At least one of the following conditions  hold:\\
  $(i)$ $ D^{(1,0)}\phi\circ (g_1^{\star},g_2^{\star}),D^{(0,1)}\phi\circ (g_1^{\star},g_2^{\star})\in L^2(\rho)$. \\
  $(ii)$ $ D^{(1,0)}\phi\circ (g_1^{\star},g_2^{\star})\in L^2(\rho)$ and $\Theta$ is such that $\theta_2=0$ for all $(\theta_1,\theta_2) \in \Theta$.\\
  $(iii)$  $ D^{(0,1)}\phi\circ (g_1^{\star},g_2^{\star})\in L^2(\rho)$   and $\Theta$ is such that $\theta_1=0$ for all $(\theta_1,\theta_2) \in \Theta$.
    \end{assump}
\begin{assump} \label{Assump3}  $\Phi$ is well-defined\,\footnote{When $\phi$ is convex and $\rho$ is a finite measure, Jensen's inequality and Assumption~\ref{Assump1} automatically imply that $\Phi$ is well-defined and nonnegative.} on $\Theta$ and $\abs{\Phi(\theta_1,\theta_2)}<\infty$ for all $(\theta_1,\theta_2) \in \Theta$. 
\end{assump}

\begin{assump} \label{Assump4}
For any $\alpha\in\NN_0^2$ with $|\alpha|=2$, there exists $v_{\alpha} \in L^1([0,1],\lambda)$, such that for all $\theta=(\theta_1,\theta_2) \in \Theta$ and $\tau \in [0,1]$, we have
    \begin{align}
    (1-\tau) \abs{\theta^{\alpha}D^{\alpha}\,\phi \circ\big((g_1^{\star},g_2^{\star})+\tau (\theta_1,\theta_2)\big)} \lesssim \abs{\theta^{\alpha}\psi_{\alpha}\circ(g_1^{\star},g_2^{\star})}v_{ \alpha}(\tau)\quad \rho- a.e.\label{eq:fdiv-integcond}
    \end{align}
\end{assump}
Under these assumptions, we have the desired Hadamard differentiability result.
\begin{prop}\label{Prop:Hadamarddiff-gen}
\begin{enumerate}[(i)]
    \item 
If Assumptions \ref{Assump1}-\ref{Assump4} hold for $\Theta \subseteq \Theta'$, then $\Phi$, as defined in \eqref{eq:defPhifunc}, is second order Hadamard differentiable  at $\theta^{\star}:=(0,0) \in \Theta$ tangentially to $\Theta$ with 
\begin{align}
 \Phi_{\theta^{\star}}'(h_1,h_2)&=\int_{\mathfrak{S}}\mspace{-2mu} \big(h_{1}\,D^{(1,0)}\phi\mspace{-2mu}\circ\mspace{-2mu}(g_1^{\star},g_2^{\star})\mspace{-2mu}+\mspace{-2mu}h_{2}\,D^{(0,1)}\phi\mspace{-2mu}\circ\mspace{-2mu}(g_1^{\star},g_2^{\star})\big) d \rho, \label{eq:Hadmardfirstderiv}\\
 \Phi_{\theta^{\star}}''(h_1,h_2)&=\int_{\mathfrak{S}} \big(h_{1}^2\,D^{(2,0)}\phi\mspace{-2mu}\circ\mspace{-2mu}(g_1^{\star},g_2^{\star})\mspace{-2mu}+\mspace{-2mu}h_2^2\,D^{(0,2)}\phi\circ(g_1^{\star},g_2^{\star})+2h_1h_2\,D^{(1,1)}\phi\circ(g_1^{\star},g_2^{\star})\big) d \rho, \label{eq:Hadmardsecderiv}
\end{align}
for all $(h_1,h_2) \in \mathfrak{T}_{\Theta}(\theta^{\star})=\mathrm{cl}\big(\big\{\theta/t:\theta=(\theta_1,\theta_2) \in \Theta, ~t>0\big\}\big)$.
\item The above claim further extends to $\Theta \subseteq \Theta''$ provided $\phi$ is also continuous at $(0,0)$.
\end{enumerate}
\end{prop}

The following example instantiates the above framework to the case of KL divergence. 

\begin{example}[KL divergence]\label{ex:KL_HDD}
Consider the KL divergence $\kl{\mu}{\nu}$ between probability  measures $\mu \ll \nu \in \cP(\mathfrak{S})$, in which case we may take $ \rho=\nu$. Assume $p_{\mu}>0$. The setup above specializes the KL divergence case with $\phi(x,y)=x\log  x$, $g_1^{\star}=p_{\mu}$, and $g_2^{\star}=1$, whereby $\kl{\mu}{\nu}=\Phi(0,0)$ (see \eqref{eq:defPhifunc}). 
Further set $\psi_{2,0}(x,y)=D^{(2,0)}\phi(x,y)=1/x$, $\psi_{1,1}(x,y)=D^{(1,1)}\phi(x,y)=0$, and  $\psi_{0,2}(x,y)=D^{(0,2)}\phi(x,y)$ $=0 $. The measures $\eta_1$ and $\eta_2$ are defined through the densities $p_{\eta_1}=1+\psi_{2,0}\circ(p_{\mu},1)=1+1/p_{\mu}$ and $p_{\eta_2}=1$ (note that $\eta_2=\nu$). The spaces of interest are taken as $\mathfrak{D}=\big\{(g_1-p_{\mu},g_2-1):\,g_1,g_2 \in L^1(\rho),~\norm{g_1-p_{\mu}}_{\eta_1}+\norm{g_2-1}_{\nu}<\infty \big\}$, $\Theta=\{(g_1-p_{\mu},0) \in \mathfrak{D}:g_1 \geq 0,~\norm{g_1}_{1,\nu} = 1\}$. The function $v_\alpha$, $\alpha\in\NN_0^2$, to satisfy \eqref{eq:fdiv-integcond} is chosen as $v_{0,2}=v_{1,1}=0$, and $v_{2,0}=1$. Indeed, the former null values are sufficient since $D^{(2,0)}\phi=D^{(1,1)}\phi=0$, while for the latter, with $\theta=(\theta_1,\theta_2)=(g_1-g_1^{\star},g_2-g_2^{\star})$, we have 
\[\mspace{-2.5mu}
       (1-\tau) \mspace{-2mu}\abs{\theta_1^2D^{(2,0)}\phi\mspace{-2mu} \circ\big((g_1^{\star},g_2^{\star})\mspace{-2mu}+\mspace{-2mu}\tau (\theta_1,\theta_2)\big)}\mspace{-1mu} =\mspace{-1mu} \frac{(1\mspace{-2mu}-\mspace{-2mu}\tau) \theta_1^2}{(1\mspace{-2mu}-\mspace{-2mu}\tau)g_1^{\star}\mspace{-2mu}+\mspace{-2mu}\tau g_1}\leq \frac{\theta_1^2}{g_1^{\star}}\mspace{-1mu}=\mspace{-1mu}\abs{\theta_1^2\psi_{2,0} \circ\mspace{-2mu} (g_1^{\star},g_2^{\star})}\mspace{-2mu}.
\]
Here, the inequality holds  since $g_1 \geq  0$. Note that    $\mathfrak{T}_{\Theta}(\theta^{\star})=\mathrm{cl}\big(\big\{\big((g_1-p_{\mu})/t,0\big):(g_1-p_{\mu},0) \in \Theta, ~t>0\big\}\big)$. Corresponding choices for other $f$-divergences can be found in Lemma \ref{Lem:assumpfdiv} (in  Section \ref{Sec:limittheorems-proof}) below.
\end{example}

\section{Limit Distribution for $f$-Divergences} \label{Sec:limittheorems}

We instantiate the above unified framework to derive a flexible limit distribution theory under general regularity conditions for several popular $f$-divergences---KL divergence, $\chi^2$~divergence, $\mathsf{H}^2$ distance, and TV distance. To maintain versatility, we model generic estimators of population probability distribution as random probability measures. These can be substituted with a specific estimator depending on the application.

\begin{definition}[Random probability measure]
A random probability measure on $\mathfrak{S}$ is a map $\zeta: \Omega \times  \cS \rightarrow [0,1]$ satisfying
\begin{enumerate}[(i)]
    \item  for every $\cC \in \cS$, $\omega \rightarrow \zeta(\omega, \cC)$ is measurable from $(\Omega,\cA)$ to $(\RR,\cB(\RR))$;
    \item for every  $\omega \in \Omega$,  $\zeta(\omega, \cdot) \in \cP(\mathfrak{S})$.
\end{enumerate}
\end{definition}

Let $\mu,\nu\in\cP(\mathfrak{S})$. Consider a sequence $(\mu_n,\nu_n)_{n \in \NN}$ of random probability measures on $\mathfrak{S}$ such that   
$(\mu_n,\nu_n)$ converges weakly to $(\mu,\nu)$. Accordingly,  $(\mu_n,\nu_n)$
can be viewed as an instance of weakly convergent 
estimators of the population distribution $(\mu,\nu)$. Below, the one- and two-sample settings refer to when only $\mu$ or both $(\mu,\nu)$ are approximated by $\mu_n$ or $(\mu_n,\nu_n)$, respectively. We also use the terms `null' and `alternative' for when $\mu=\nu$ or $\mu\neq \nu$, respectively. In the following results, $(r_n)_{n \in \NN}$ denotes a diverging sequence, $q$, $q_1$, and $q_2$ represent measurable functions. Also,  inequalities involving relative densities, e.g. $p_{\mu}>0$, are to be interpreted as holding $\rho$ a.e., and  regularity  conditions involving random measures, e.g. $\kl{\mu_n}{\nu}<\infty$, are required to hold only for sufficiently large $n$. 

 \subsection{KL divergence} \label{sec:KLdiverggenres}
\begin{theorem}
[Limit distribution for KL divergence]\label{Thm:KLdiv-limdist}
 The following hold: 
\begin{enumerate}[(i)]
    \item (One-sample null) Let $\mu_n \ll \mu=\rho$ be such that   $\kl{\mu_n}{\mu}<\infty$ almost surely  (a.s.). If $r_n(p_{\mu_n}-1) \trightarrow{w} B$ in $L^2(\mu)$,  then
    \begin{align}
    r_n^2 \kl{\mu_n}{\mu} \trightarrow{d} \frac{1}{2} \int_{\mathfrak{S}} B^2 d \mu.  \label{eq:KL-onesample-null}
\end{align}
 \item (One-sample alternative) Let  $\mu_n \ll \mu \ll \nu=\rho$  satisfy  $\ind_{\supp(p_{\mu})}\log p_{\mu}  \in L^2(\nu)$, 
 $\kl{\mu}{\nu}<\infty$, and $\kl{\mu_n}{\nu}<\infty$  a.s.  If $r_n(p_{\mu_n}-p_{\mu}) \trightarrow{w} B$ in $L^2(\eta)$,  where $\eta$ has relative density $p_{\eta}=1+(\ind_{\supp(p_{\mu})}/p_{\mu})$, then
 \begin{align}
    r_n \big(\kl{\mu_n}{\nu}-\kl{\mu}{\nu}\big) \trightarrow{d} \int_{\supp(p_{\mu})} B \log p_{\mu}   d \nu. \label{eq:KL-onesample-alt}
\end{align}
\item (Two-sample null) Let $\mu_n \ll \nu_n \ll \mu=\rho$ be such that   $\kl{\mu_n}{\nu_n}<\infty$ and $p_{\mu_n}/p_{\nu_n} \leq q$ a.s. Let  
$\eta_1=\mu$ and $\eta_2$ be the measure with relative density $p_{\eta_2}=1+q$. If $\big(r_n(p_{\mu_n}-1),r_n(p_{\nu_n}-1)\big) \trightarrow{w} (B_1,B_2)$ in $L^2(\eta_1) \times L^2(\eta_2)$, then
    \begin{align}
    r_n^2 \kl{\mu_n}{\nu_n} \trightarrow{d} \frac{1}{2} \int_{\mathfrak{S}} (B_1-B_2)^2 d \mu.  \label{eq:KL-twosample-null}
\end{align}
 \item (Two-sample alternative)
 Let  $\mu_n \ll \nu_n  \ll  \nu=\rho$ and $\mu_n \ll \mu  \ll  \nu$ satisfy    $p_{\mu}, \ind_{\supp(p_{\mu})}\log p_{\mu}  \in L^2(\nu)$, $\kl{\mu}{\nu}<\infty$,  $\kl{\mu_n}{\nu_n}<\infty$ and $p_{\mu_n}/p_{\nu_n} \leq q$ a.s. Let  $\eta_1$ and $\eta_2$ be measures with relative densities $p_{\eta_1}=1+(\ind_{\supp(p_{\mu})}/p_{\mu})$ and $p_{\eta_2}=1+p_{\mu}+q$, respectively.  
 If $\big(r_n(p_{\mu_n}-p_{\mu}),r_n(p_{\nu_n}-1)\big) \trightarrow{w} (B_1,B_2)$ in  $L^2(\eta_1) \times L^2(\eta_2)$,  then
 \begin{align}
    r_n \big(\kl{\mu_n}{\nu_n}-\kl{\mu}{\nu}\big) \trightarrow{d} \int_{\supp(p_{\mu})} B_1 \log p_{\mu}   d \nu-\int_{\supp(p_{\mu})} B_2    d \mu. \label{eq:KL-twosample-alt}
\end{align}
\end{enumerate}
\end{theorem}
The proof of Theorem \ref{Thm:KLdiv-limdist} utilizes the  functional delta method in Lemma \ref{Lem:extfuncdelta} and Proposition \ref{Prop:Hadamarddiff-gen}. For the purpose of applying the latter result, we note that the KL divergence functional  can be written in the form \eqref{eq:defPhifunc} with an appropriate $\phi$ in each of the cases above such that Assumptions \ref{Assump1}-\ref{Assump4} are satisfied under the conditions therein (see Example~\ref{ex:KL_HDD}). 
\begin{remark}[Limit distribution under probabilistic constraints]\label{rem:relaxeddeltmethod}
Let $\cQ_n=\{\omega \in \Omega: p_{\mu_n}(\omega,\cdot)/p_{\nu_n}(\omega,\cdot) \leq q(\cdot) \}$.  
It is possible to replace the constraints $p_{\mu_n}/p_{\nu_n} \leq q$ a.s. by the more relaxed constraint $\PP^*(\Omega \setminus \cQ_n) \rightarrow 0$  if $(B_1,B_2)$ is continuous. Here, $\PP^*$ denotes outer probability  which is needed since $\cQ_n$ may be non-measurable, in general. The continuity of $(B_1,B_2)$ means that for every Borel set $\cC 
 \in L^2(\eta_1) \times L^2(\eta_2)$,  $\PP\big( (B_1,B_2) \in \partial \cC\big )=0$, where $\partial \cC $ denotes the boundary of $\cC$. This holds, for instance, if $(B_1,B_2)$ is  Gaussian. The proof of the claim (see Appendix \ref{Sec:remarkproof}) follows similarly to that of Theorem \ref{Thm:KLdiv-limdist} by using Part $(iii)$ of Lemma \ref{Lem:extfuncdelta} in place of Part $(i)$ and $(ii)$. Similar remarks apply to Theorem \ref{Thm:chisqdiv-limdist} and \ref{Thm:Helsqdiv-limdist} which are omitted for brevity.
\end{remark}
A few remarks about the  regularity conditions in Theorem \ref{Thm:KLdiv-limdist} are in order. In the one-sample null case, the condition $r_n(p_{\mu_n}-1) \trightarrow{w} B$ is the weak convergence of the (centered) density of  $\mu_n$  in $L^2(\mu)$, which is a natural requirement for the existence of the KL limit. In the one-sample alternative case, we require the integrability condition $\ind_{\supp(p_{\mu})}\log p_{\mu}  \in L^2(\nu)$ along with $r_n(p_{\mu_n}-p_{\mu}) \trightarrow{w} B$ in $L^2(\eta)$. Using $\abs{\log x} \leq  \abs{x-1} \vee \abs{1-x^{-1}}$, the former condition is holds if $p_{\mu}, \ind_{\supp(p_{\mu})}/p_{\mu} \in L^2(\nu)$. For the latter (weak convergence) condition to hold it is sufficient that $r_n(p_{\mu_n}-p_{\mu}) \trightarrow{w} B$ in $L^2(\nu)$, so long as $\ind_{\supp(p_{\mu})}/p_{\mu} \in L^{\infty}(\nu)$. The corresponding two-sample results hold under similar conditions plus a requirement of existence of a function $q$ that dominates the ratio of the (random) densities of $\mu_n$ and $\nu_n$, i.e., $p_{\mu_n}/p_{\nu_n} \leq q$ a.s., which can be relaxed to a probabilistic constraint given in Remark \ref{rem:relaxeddeltmethod} provided the limit variables are continuous. We emphasize that this additional regularity condition necessitates that the one- and two-sample cases are stated separately, even though the one-sample limit distribution may be obtained from the two-sample result by setting $(B_1,B_2)=(B,0)$.

To gain further insight into the applicability of Theorem \ref{Thm:KLdiv-limdist}, we next consider several important examples.  
\begin{enumerate}[1)]
    \item \underline{Finite support:} Let $\mathfrak{S}$ be a finite set,  $\mu \ll  \nu$, and $(X_1,\ldots,X_n) $ and $(Y_1,\ldots,Y_n) $ be i.i.d. samples from $\mu$ and $\nu$, respectively. Set the empirical distributions  $\mu_n=n^{-1}\sum_{i=1}^n \delta_{X_i}$ and $\nu_n=(n^{-1}-n^{-2}) \sum_{i=1}^n \delta_{Y_i}+n^{-1} \mu_n$ as the random measures, where we add a vanishing regularization term $n^{-1} \mu_n$ so that $\mu_n \ll \nu_n $ and $\kl{\mu_n}{\nu_n}<\infty$.   Clearly, $\mu_n \ll \nu_n \ll \nu $ and $ \mu_n \ll \mu \ll \nu$. Notice also that all the other regularity conditions in Theorem  \ref{Thm:KLdiv-limdist} are satisfied. In particular,  the relevant weak convergences of random measures to Gaussian limits follow from the multivariate CLT. Moreover, by Hoeffding's inequality, there exists a constant $c$ such that $\PP(\norm{p_{\mu_n}/p_{\nu_n}}_{\infty} > c) \rightarrow 0$.  Hence, Remark \ref{rem:relaxeddeltmethod} applies with $q=c$ implying that   \eqref{eq:KL-onesample-null}-\eqref{eq:KL-twosample-alt} hold  with $r_n=n^{1/2}$ and $B$,  $(B_1,B_2)$ as  Gaussian vectors (and, of course, integral replaced by summation). 
 
    \item \underline{Compact support with smoothed empirical measures:} Consider compactly supported and continuous distributions $\mu \ll \nu $ on $\RR^d$ with Lebesgue densities bounded from above and away from zero on the support. In this case, the vanilla empirical distributions as defined above are not absolutely continuous w.r.t. $\mu$ and $\nu$, respectively. To resolve this and obtain a well-posed empirical approximation setting, we convolve the empirical distributions with a mollifier\footnote{A mollifier is a function 
 $f:\mathbb{R}^d \rightarrow \mathbb{R}$  which is both smooth (i.e., has continuous derivatives of all orders) and compactly supported, e.g., $\xi(x)=c e^{-1/(1-\norm{x}^2)}$ for $\norm{x} <1$ and zero elsewhere for some normalizing constant $c$.  } (or bump function). Let   $\mu_n$ and $\nu_n$ denote the smoothed empirical distributions with Lebesgue densities $p_{\mu_n}=n^{-1}\sum_{i=1}^n \delta_{X_i} * \xi $ and $p_{\nu_n}=(n^{-1}-n^{-2}) \sum_{i=1}^n \delta_{Y_i} * \xi+n^{-1} p_{\mu_n}$, respectively, where $\xi$ denotes a non-negative mollifier with $\norm{\xi}_{1}=1$. It follows by CLT in $L^r$ spaces (see Theorem \ref{Thm:CLT-in-Lp} below) that the weak-convergence requirements in Theorem \ref{Thm:KLdiv-limdist} hold with $\mu*\xi,\nu * \xi$ in place of $\mu,\nu$ and with $r_n=n^{1/2}$ and $B$, $(B_1,B_2)$ as appropriate  Gaussian processes indexed by a compact subset $\mathfrak{S}$ of $\RR^d$. Moreover, using concentration bounds for separable sub-Gaussian processes (see e.g.,\cite[Theorem 5.29]{VanHandel-book}) indexed by compact $\mathfrak{S}$, one can show that there exists a constant $c$ such that $\PP(\norm{p_{\mu_n}/p_{\nu_n} }_{\infty}> c) \rightarrow 0$.  It can be verified that the other regularity conditions in Theorem  \ref{Thm:KLdiv-limdist} also hold, and consequently the result applies via Remark \ref{rem:relaxeddeltmethod}.
    \item \underline{Unbounded support with Gaussian smoothing:} In Proposition \ref{Prop:GS-KL-limdist} below, we further specialize  Theorem \ref{Thm:KLdiv-limdist} to Gaussian-smoothed empirical distributions on $\RR^d$. This provides an instance of smooth distributions with unbounded support, for which we characterize the limit laws and derive primitive conditions in terms of $\mu$ and $\nu$ for their existence. 
\end{enumerate}
We mention here that while Theorem \ref{Thm:KLdiv-limdist} provides  general regularity conditions for existence of  limit distributions for KL divergence, certain assumptions therein are arguably stronger than what is necessary. For instance, the requirement $\ind_{\supp(p_{\mu})}\log p_{\mu}  \in L^2(\nu)$ for \eqref{eq:KL-onesample-alt} to hold arises from Assumption \ref{Assump2}, which furnishes sufficient conditions for the first order Hadamard derivatives to exist. However, from the proof of Proposition \ref{Prop:Hadamarddiff-gen}, it is evident that alternative sufficient conditions in lieu of Assumption \ref{Assump2} are plausible by setting up the function space $\mathfrak{D}$ of perturbations via the norm $\norm{(g,\tilde g)}_{\mathfrak{D}}:=\norm{g}_{r,\eta_1}+\norm{\tilde g}_{r,\eta_2}$ for some  $r \geq 1$. Then, the relevant condition 
for existence of limit distribution on account of  Lemma \ref{Lem:contmapLp}  would be $\ind_{\supp(p_{\mu})}\log p_{\mu}  \in L^{r'}(\nu)$, where $r'$ is the H\"{o}lder conjugate of $r$. Thus, there is some flexibility possible in the regularity conditions required for \eqref{eq:KL-onesample-alt} to hold. That said, we do not delve into this aspect  further within this paper.

 \subsection{$\chi^2$ divergence}
\begin{theorem}
[Limit distribution for $\chi^2$ divergence]\label{Thm:chisqdiv-limdist}
The following hold:
\begin{enumerate}[(i)]
    \item (One-sample null)  Let $\mu_n \ll \mu=\rho$ satisfy $\chisq{\mu_n}{\mu}<\infty$  a.s. If $r_n(p_{\mu_n}-1) \trightarrow{w} B$ in $L^2(\mu)$, then
    \begin{align}
    r_n^2 \chisq{\mu_n}{\mu} \trightarrow{d}  \int_{\mathfrak{S}} B^2 d \mu.  \label{eq:chisq-onesample-null}
\end{align}
 \item (One-sample  alternative) Let $\mu,\mu_n \ll \nu=\rho $ satisfy $\chisq{\mu}{\nu}<\infty$ and $\chisq{\mu_n}{\nu}<\infty$  a.s.  If $r_n(p_{\mu_n}-p_{\mu}) \trightarrow{w} B$ in $L^2(\nu)$, then
 \begin{align}
    r_n \big(\chisq{\mu_n}{\nu}-\chisq{\mu}{\nu}\big) \trightarrow{d} 2\int_{\mathfrak{S}} B    d \mu. \label{eq:chisq-onesample-alt}
\end{align}
 \item (Two-sample null)
 Let $\mu_n \ll \nu_n \ll \mu=\rho$ be such that  $\chisq{\mu_n}{\nu_n}<\infty$   and $p_{\mu_n}/p_{\nu_n} \leq q$ a.s. Let  $\eta_1$ and $\eta_2$ be measures with relative densities $p_{\eta_1}=1+q$ and $p_{\eta_2}=p_{\eta_1}+q^2$, respectively. 
 If $\big(r_n(p_{\mu_n}-1),r_n(p_{\nu_n}-1)\big) \trightarrow{w} (B_1,B_2)$ in  $L^2(\eta_1) \times L^2(\eta_2)$,   then
    \begin{align}
    r_n^2 \chisq{\mu_n}{\nu_n} \trightarrow{d}  \int_{\mathfrak{S}} (B_1-B_2)^2 d \mu. \label{eq:chisq-twosample-null}
\end{align}
 \item (Two-sample alternative) Let $\mu_n \ll \nu_n \ll   \nu=\rho$ and $\mu_n \ll \mu \ll   \nu$ satisfy  $p_{\mu}  \in L^4(\nu)$,  $\chisq{\mu}{\nu}<\infty$,  $\chisq{\mu_n}{\nu_n}<\infty$ and $p_{\mu_n}/p_{\nu_n} \leq q$ a.s.  Let $\eta_1$ and $\eta_2$ be measures with relative densities $p_{\eta_1}=1+p_{\mu}+q$ and $p_{\eta_2}=p_{\eta_1}+q^2+p_{\mu}^2$, respectively.  If $\big(r_n(p_{\mu_n}-p_{\mu}),r_n(p_{\nu_n}-1)\big) \trightarrow{w} (B_1,B_2)$ in  $L^2(\eta_1) \times L^2(\eta_2)$, then
 \begin{align}
    r_n \big(\chisq{\mu_n}{\nu_n}-\chisq{\mu}{\nu}\big) \trightarrow{d} 2\int_{\mathfrak{S}} B_1   d \mu-\int_{\mathfrak{S}} B_2 p_{\mu}   d \mu. \label{eq:chisq-twosample-alt}
\end{align}
\end{enumerate}
\end{theorem}
Theorem \ref{Thm:chisqdiv-limdist} utilizes similar proof methodology as in Theorem \ref{Thm:KLdiv-limdist} via the machinery from Section \ref{Sec:Techframework}, except that Parts $(i)$ and $(ii)$ follow in a simpler manner via the continuous mapping theorem \cite[Theorem 1.3.6]{AVDV-book} owing to the specific structure of the $\chi^2$ divergence functional. Also note that the one- and two-sample null limits of the $\chi^2$ and KL divergences are the same up to a factor of 0.5, which is in line with the fact that  KL divergence is locally~$\chi^2$ (see \cite[Theorem 7.18]{polyanskiy2022codingbook}). Via an argument similar to that in Remark \ref{rem:relaxeddeltmethod}, we observe that  Theorem \ref{Thm:chisqdiv-limdist} applies to discrete and smoothed compactly supported distributions (as discussed after Theorem \ref{Thm:KLdiv-limdist} above), with  $r_n=n^{1/2}$ and Gaussian $B$, $(B_1,B_2)$. An application  to Gaussian-smoothed empirical distributions will be given in Proposition \ref{Prop:GS-chisq-limdist} in Appendix \ref{App:GSfdiv-chisq}.

\subsection{ $\mathsf{H}^2$ distance}
\begin{theorem}
[Limit distribution for $\mathsf{H}^2$ distance]\label{Thm:Helsqdiv-limdist}
The following hold:
\begin{enumerate}[(i)]
    \item (One-sample null)  Let $\mu_n, \mu \ll \rho$ for some finite measure $\rho$ such that $p_{\mu} >0$. If $r_n(p_{\mu_n}-p_{\mu}) \trightarrow{w} B$ in $L^2(\eta)$, where  $\eta$ has relative density $p_{\eta}=1+(1/p_{\mu})$,  then
    \begin{align}
    r_n^2 \helsq{\mu_n}{\mu} \trightarrow{d}  \frac 14 \int_{\mathfrak{S}} \frac{B^2}{p_{\mu}}d \rho. \label{eq:Hel-onesample-null}
\end{align}
 \item (One-sample alternative) Let $\mu_n, \mu,\nu \ll \rho$ for some finite measure $\rho$ such that $p_{\mu},p_{\nu}>0$, and suppose that  $p_{\nu}/p_{\mu} \in L^1(\rho)$. If $r_n(p_{\mu_n}-p_{\mu}) \trightarrow{w} B$ in $L^2(\eta)$, where $\eta$  has relative  density $p_{\eta}=1+\big(p_{\nu}^{1/2}/p_{\mu}^{3/2}\big)$, then
 \begin{align}
    r_n \big(\helsq{\mu_n}{\nu}-\helsq{\mu}{\nu}\big) \trightarrow{d} -\int_{\mathfrak{S}} \left(\frac{p_{\nu}}{p_\mu}\right)^{\frac 12}  B d\rho. \label{eq:Helsq-onesample-alt}
\end{align}
\item (Two-sample null) Let $\mu_n,\nu_n, \mu\ll \rho$ for some finite measure $\rho$ such that $ p_{\mu}>0 $,  $p_{\mu_n} \leq q_1$ and $ p_{\nu_n}\leq q_2$ a.s. Let $\eta_1$ and $\eta_2$ be measures with relative densities $p_{\eta_1}=1+(1/p_{\mu})+\big(q_2^{1/2}/p_{\mu}^{3/2}\big)$ and $p_{\eta_2}=1+(1/p_{\mu})+\big(q_1^{1/2}/p_{\mu}^{3/2}\big)$, respectively. If $\big(r_n(p_{\mu_n}-p_{\mu}),r_n(p_{\nu_n}-p_{\mu})\big) \trightarrow{w} (B_1,B_2)$ in $L^2(\eta_1) \times L^2(\eta_2) $, then
 \begin{align}
    r_n^2 \helsq{\mu_n}{\nu_n} \trightarrow{d}  \frac 14 \int_{\mathfrak{S}} \frac{(B_1-B_2)^2}{p_{\mu}} d \rho. \label{eq:Helsq-twosample-null}
\end{align}
 \item (Two-sample alternative) Let $\mu_n, \nu_n, \mu,\nu \ll \rho$ for some finite measure $\rho$ such that  $ p_{\mu},p_{\nu}>0 $,  $p_{\nu}/p_{\mu},p_{\mu}/p_{\nu} \in L^1(\rho)$,   $p_{\mu_n} \leq q_1$ and $p_{\nu_n} \leq q_2$ a.s. Let $\eta_1$ and $\eta_2$ be measures with  relative densities 
\begin{align}
   p_{\eta_1}&=1+\Big(p_{\nu}^{\frac 12}+q_2^{\frac 12}\Big)p_{\mu}^{-\frac 32}+p_{\mu}^{-\frac 12}p_{\nu}^{-\frac 12},  \notag \\
   p_{\eta_2}&=1+\Big(p_{\mu}^{\frac 12}+q_1^{\frac 12}\Big)p_{\nu}^{-\frac 32}+p_{\mu}^{-\frac 12}p_{\nu}^{-\frac 12},\notag
\end{align} 
 respectively. 
 If $\big(r_n(p_{\mu_n}-p_{\mu}),r_n(p_{\nu_n}-p_{\nu})\big) \trightarrow{w} (B_1,B_2)$ in in $L^2(\eta_1) \times L^2(\eta_2) $, then
 \begin{align}
    r_n \big(\helsq{\mu_n}{\nu_n}-\helsq{\mu}{\nu}\big) \trightarrow{d} -\int_{\mathfrak{S}} \left(\frac{p_{\nu}}{p_\mu}\right)^{\frac 12}  B_1 d\rho -\int_{\mathfrak{S}} \left(\frac{p_{\mu}}{p_\nu}\right)^{\frac 12}  B_2 d\rho. \label{eq:Helsq-twosample-alt}
\end{align}
\end{enumerate}
\end{theorem}
The proof of Theorem \ref{Thm:Helsqdiv-limdist} is similar to that of Theorem \ref{Thm:KLdiv-limdist}, and is given in Section \ref{Thm:Helsqdiv-limdist-proof}. Note that in contrast to Theorem \ref{Thm:KLdiv-limdist} and \ref{Thm:chisqdiv-limdist}, the two-sample alternative limit distribution is symmetric in $\mu$ and $\nu$, which is to be expected given the symmetry  of $\helsq{\mu}{\nu}$ itself. Also, observe that  the regularity conditions in Theorem \ref{Thm:Helsqdiv-limdist} are satisfied for the case of discrete (when $\mu \ll \gg \nu$) and compactly supported smoothed empirical distributions considered in Section \ref{sec:KLdiverggenres}. Consequently, \eqref{eq:Hel-onesample-null}-\eqref{eq:Helsq-twosample-alt} hold with $r_n=n^{1/2}$  and Gaussian $B$,  $(B_1,B_2)$. In Proposition \ref{Prop:GS-helsq-limdist} in Appendix \ref{App:GSfdiv-helsq}, we will further discuss the  application of Theorem \ref{Thm:Helsqdiv-limdist} to Gaussian-smoothed empirical distributions.
\subsection{TV distance}
\begin{theorem}
[Limit distribution for TV distance]\label{Thm:TVdiv-limdist}
Let $\mathrm{sgn}(x)=x/|x|$ for $x \neq 0$, and  for given  $\mu,\nu \ll \rho$, let $\cQ:=\{s \in \mathfrak{S}: p_{\mu}(s)=p_{\nu}(s)\}$.
 Then, the following hold:
\begin{enumerate}[(i)]
 \item (One-sample null and alternative) Let $\mu_n,\mu,\nu \ll \rho$ for some measure $\rho$. If $r_n(p_{\mu_n}-p_{\mu}) \trightarrow{w} B$ in $L^1(\rho)$,  then
    \begin{align}
    r_n \big(\tv{\mu_n}{\nu}-\tv{\mu}{\nu}\big) \trightarrow{d} &\frac 12 \int_{\cQ}\abs{B} d\rho +\frac 12 \int_{\mathfrak{S}\setminus \cQ} \mathrm{sgn}\big(p_{\mu}-p_{\nu}\big)Bd\rho. \label{eq:TV-onesample-alt}
\end{align}
 \item (Two-sample null and alternative)
 Let $\mu_n, \nu_n,\mu,\nu \ll \rho$ for some measure $\rho$. If$\mspace{5 mu}$\footnote{Alternatively, we may state the weak convergence requirement in terms of the stronger condition $\big(r_n(p_{\mu_n}-p_{\mu}), r_n(p_{\nu_n}-p_{\nu})\big) \trightarrow{w} (B_1,B_2)$ which implies $r_n(p_{\mu_n}-p_{\nu_n}-p_{\mu}+p_{\nu}) \trightarrow{w} B$ with $B=B_1-B_2$. } $r_n(p_{\mu_n}-p_{\nu_n}-p_{\mu}+p_{\nu}) \trightarrow{w} B$ in $L^1(\rho)$,  then
    \begin{align}
    r_n \big(\tv{\mu_n}{\nu_n}-\tv{\mu}{\nu} \big)\trightarrow{d} \frac 12 \int_{\cQ}\abs{B} d\rho +\frac 12 \int_{\mathfrak{S}\setminus \cQ} \mathrm{sgn}\big(p_{\mu}-p_{\nu}\big)Bd\rho. \label{eq:TV-twosample-alt}
\end{align}
\end{enumerate}
\end{theorem}
The proof for the one-sample and two-sample null cases above  follow via an application of the continuous mapping theorem. The proof of the alternative requires an argument different from the divergences considered until now. To see why, recall from  \eqref{eq:fdivdef} that TV distance corresponds to an $f$-divergence with $f=\abs{x-1}/2$, which is not differentiable at $x=1$. Hence, Proposition \ref{Prop:Hadamarddiff-gen} does not apply  for computing the Hadamard derivative since Assumption \ref{Assump1} is violated. However, utilizing the equivalence of Hadamard differentiability to  G\^{a}teaux differentiability under  Lipschitzness of the functional (see \cite{Shapiro-1990}), we instead compute the latter simpler quantity via a direct argument. From there, the claims follow as usual via an application of the functional delta method. We note that  Theorem \ref{Thm:TVdiv-limdist} applies to the case of discrete and compactly supported smoothed empirical distributions  in  Section \ref{sec:KLdiverggenres}. Its application to Gaussian-smoothed case is discussed in Appendix \ref{App:GSfdiv-TV} (Proposition \ref{Prop:GS-TV-limdist}).

\section{Limit Distribution Theory for Gaussian-Smoothed $f$-Divergences}\label{Sec:GS limit theorems}
We study distributional limits of Gaussian-smoothed $f$-divergence, i.e., the population objective is $\mathsf{D}_f(\mu*\gamma_\sigma\|\nu*\gamma_\sigma)$, where $\gamma_\sigma=N(0,\sigma^2 I_d)$ is the isotropic Gaussian kernel \cite{Goldfeld-2020}. Our goal is to approximate $\mu$ (or both $\mu$ and $\nu$) from samples, while assuming that the Gaussian kernel is known. The Gaussian smoothing alleviates mismatched support issues that $f$-divergences often suffer from and gives rise to a well-posed empirical approximation setting. 
This setup was studied in our preliminary work \cite{Goldfeld-Kato-2020} for the $\chi^2$ divergence and the TV distance under the one-sample null setting. The results herein significantly generalize and broaden those of \cite{Goldfeld-Kato-2020}. We focus on the Gaussian-smoothed KL divergence in this section and relegate analogous results for $\chi^2 $ divergence, $\mathsf{H}^2$ distance and TV distance to Appendix \ref{App:GSfdiv} for brevity. Throughout this section, we assume $\mathfrak{S}=\RR^d$ and $\cS=\cB\big(\RR^d\big)$. Some preliminaries are due before stating the results.

\medskip
\noindent \textbf{Empirical measures:} In defining the empirical measures of $\mu$ and $\nu$ we allow arbitrary correlation between their samples, which is necessary for the application to auditing DP studied in Section \ref{Sec:appkldiffprivacy}. 
Let $(X,Y) \sim \pi\in \cP(\RR^d \times \RR^d)$ with $X,Y$ marginals $\mu,\nu$, respectively. Set $\hat{\mu}_n = n^{-1}\sum_{i=1}^n \delta_{X_i}$ as the empirical distribution of  $(X_1,\ldots,$ $X_n) $ and  $\hat{\nu}_n = n^{-1}\sum_{i=1}^n \delta_{Y_i}$ be that of $(Y_1,\ldots,Y_n)$, where $(X_i,Y_i) \sim \pi$, $1 \leq i \leq n$, are i.i.d. Recalling that $\varphi_{\sigma}$ denotes the density of $\gamma_{\sigma}$, the Lebesgue densities of $\hat \mu_n*\gamma_\sigma$ and $\hat \nu_n*\gamma_\sigma$ are $\hat \mu_n*\varphi_{\sigma}$ and $\hat \nu_n*\varphi_{\sigma}$, respectively. We study distributional limits of $\mathsf{D}_f(\hat{\mu}_n*\gamma_\sigma\|\nu*\gamma_\sigma)$ as well as its two-sample analogues, under the null ($\mu=\nu$) and the alternative ($\mu\neq \nu$). 

\medskip
\noindent \textbf{Gaussian process:} Our limit variables are characterized as integral forms of a certain Gaussian process, which is introduced next. Consider the 2-dimensional centered Gaussian process $(G_{\mu,\sigma},G_{\nu,\sigma})$ $\mspace{-4 mu}:=\mspace{-4 mu}\big(G_{\mu,\sigma}(x), G_{\nu,\sigma}(y)\big)_{(x,y) \in \RR^d \times \RR^d}$ with covariance function  $\Sigma_{\mu,\nu,\sigma}: (\RR^d \times \RR^d) \times (\RR^d \times \RR^d) \rightarrow \RR^{2 \times 2}$ given by
\begin{align}
\Sigma_{\mu,\nu,\sigma}\big((x,y),(\tilde x,\tilde y)\big):&=
 \Bigg[ {\begin{array}{cc}
   \EE\big[G_{\mu,\sigma}(x)G_{\mu,\sigma}(\tilde x)\big] & \EE\big[G_{\mu,\sigma}(x)G_{\nu,\sigma}(\tilde y)\big] \\[2 pt]
   \EE\big[G_{\nu,\sigma}(y)G_{\mu,\sigma}(\tilde x)\big] & \EE\big[G_{\nu,\sigma}(y)G_{\nu,\sigma}(\tilde y)\big] \\
  \end{array} } \Bigg] \label{eq:covfuncGP} \\
  &=\Bigg[ {\begin{array}{cc}
  \mathsf{cov}\big(\varphi_{\sigma}(x-X),\varphi_{\sigma}(\tilde x-X)\big) & \mathsf{cov}\big(\varphi_{\sigma}(x-X),\varphi_{\sigma}(\tilde y-Y)\big) \\[2 pt]
   \mathsf{cov}\big(\varphi_{\sigma}(y-Y),\varphi_{\sigma}(\tilde x-X)\big) & \mathsf{cov}\big(\varphi_{\sigma}(y-Y),\varphi_{\sigma}(\tilde y-Y)\big)  \notag
  \end{array} } \Bigg].
\end{align}
For $i,j \in \{1,2\}$, denote the $(i,j)$-th entry of $\Sigma_{\mu,\nu,\sigma}$ by $\Sigma^{(i,j)}_{\mu,\nu,\sigma}$. Note that each such entry depends only on two coordinates among $\big((x,y),(\tilde x,\tilde y)\big)$.  Hence, by some abuse of notation, we omit the redundant coordinates and use the remaining coordinates in the same order they appear, e.g.,   $\Sigma^{(2,1)}_{\mu,\nu,\sigma}\big(y,\tilde x\big)$ for $\Sigma^{(2,1)}_{\mu,\nu,\sigma}\big((x,y),(\tilde x,\tilde y)\big)$. Further, when $\nu=\mu$ (and consequently $X\stackrel{d}{=}Y$), we denote $G_{\nu,\sigma}$ by $\tilde G_{\mu,\sigma}$ to avoid confusion with $G_{\mu,\sigma}$. 

\medskip
\noindent \textbf{Asymptotic variance:} When the limiting variable is Gaussian, the asymptotic variance for different $f$-divergences can be expressed in a unified form, using the following definitions. For a function $f$ corresponding to an $f$-divergence (see Definition \ref{def:f divergence}), let 
\begin{align} \notag
\begin{split}
    v_{1,f}^2(\mu,\nu,\sigma)&:=\int_{\RR^d}\int_{\RR^d} \Sigma^{(1,1)}_{\mu,\nu,\sigma}(x,y) L_{1,f}(x) L_{1,f}(y) dx\, dy,  \\
    v_{2,f}^2(\mu,\nu,\sigma)&:= \sum_{1 \leq i,j \leq 2} \int_{\RR^d}\int_{\RR^d}  \Sigma^{(i,j)}_{\mu,\nu,\sigma}(x,y) L_{i,f}(x)L_{j,f}(y) d x\, d y,
    \end{split}
\end{align}
where
\begin{align}
L_{1,f}(x)&:=L_{1,f,\mu,\nu,\sigma}(x):= f' \left(\frac{\mu*\varphi_{\sigma}(x)}{\nu*\varphi_{\sigma}(x)}\right), \notag \\
    L_{2,f}(x)&:= L_{2,f,\mu,\nu,\sigma}(x):= f \left(\frac{\mu*\varphi_{\sigma}(x)}{\nu*\varphi_{\sigma}(x)}\right)-\frac{\mu*\varphi_{\sigma}(x)}{\nu*\varphi_{\sigma}(x)}f'\left(\frac{\mu*\varphi_{\sigma}(x)}{\nu*\varphi_{\sigma}(x)}\right). \notag
\end{align}
Here, $f'$ denotes the derivative of $f$ and $v_{1,f}^2(\mu,\nu,\sigma)$, $v_{2,f}^2(\mu,\nu,\sigma)$ will be well-defined and finite in the settings we consider below. The quantities give the asymptotic variance of  Gaussian-smoothed $f$-divergence (except TV distance) in the one-sample and two-sample alternative cases, respectively.

\subsection{KL Divergence}
The following proposition characterizes the limit distribution for Gaussian-smoothed KL divergence. 
\begin{prop}[Limit distribution for Gaussian-smoothed KL divergence] \label{Prop:GS-KL-limdist} The following hold:
\begin{enumerate}[(i)]
    \item  (One-sample null) If 
    \begin{equation}
       \int_{\RR^d}\frac{\mathsf{Var}_{\mu}\big(\varphi_{\sigma}(x-\cdot)\big)}{\mu*\varphi_{\sigma}(x)}\, dx <\infty,\label{ASSUM:KL_null}
    \end{equation} 
then there exists a version of $G_{\mu,\sigma}$ such that $G_{\mu,\sigma}/\sqrt{\mu*\varphi_{\sigma}}$ is $L^2(\RR^d)$-valued, and 
\begin{align}
  n \kl{\hat \mu_n*\gamma_\sigma}{\mu*\gamma_\sigma} \trightarrow{d}   \frac 12  \int_{\RR^d} \frac{G_{\mu,\sigma}^2(x)}{\mu*\varphi_{\sigma}(x)}d x, \label{eq:KL-GS-onesamp-null}
\end{align}
where the limit can be represented as a weighted sum of i.i.d. $\chi^2$ random variables with 1 degree of freedom (see Remark \ref{Rem:variance-alt}).
In particular, \eqref{ASSUM:KL_null} and \eqref{eq:KL-GS-onesamp-null} holds for $\beta$-sub-Gaussian $\mu$  with $\beta<\sigma$. Conversely, if  \eqref{ASSUM:KL_null} is violated, then we have $\liminf_{n \rightarrow \infty}n\EE\big[\kl{\hat \mu_n*\gamma_\sigma}{\mu*\gamma_\sigma}\big]= \infty$.
\medskip
\item  (One-sample alternative)  If \eqref{ASSUM:KL_null} holds, $\norm{(\nu*\varphi_{\sigma})^2/\mu*\varphi_{\sigma}}_{\infty}<\infty$,  $\log \big(\mu*\varphi_{\sigma}/\nu*\varphi_{\sigma}\big) \in L^2(\nu*\varphi_{\sigma})$ and 
    \begin{equation}
      \int_{\RR^d}\frac{\mathsf{Var}_{\mu}\big(\varphi_{\sigma}(x-\cdot)\big)}{\nu*\varphi_{\sigma}(x)}\, dx <\infty,\label{ASSUM:KL_alt}
    \end{equation} 
then 
\begin{align}
   n^{\frac 12}\big(\kl{\hat \mu_n*\gamma_\sigma}{\nu*\gamma_{\sigma}}-\kl{\mu*\gamma_{\sigma}}{\nu*\gamma_{\sigma}}\big) \trightarrow{d} & N\big(0,v_{1,f_{\mathsf{KL}}}^2(\mu,\nu,\sigma)\big),
   \label{eq:KL-GS-onesamp-alt} 
   \end{align}
   where  $v_{1,f_{\mathsf{KL}}}^2(\mu,\nu,\sigma)$ is  given in \eqref{eq:var-smoothedKL1samp}. 
 In particular,  \eqref{eq:KL-GS-onesamp-alt} holds for $\beta$-sub-Gaussian $\mu$  with $\beta<\sigma$ such that  $\nu \ll \mu \ll \nu$ and $\norm{d \mu/d \nu}_{\infty} \vee \norm{d \nu/d \mu}_{\infty} <\infty$.
\item (Two-sample null) If 
$\mu$ has compact support, then  there exists a version of $G_{\mu,\sigma}$,  $\tilde G_{\mu,\sigma}$ such that $G_{\mu,\sigma}/\sqrt{\mu*\varphi_{\sigma}}$ and $\tilde G_{\mu,\sigma}/\sqrt{\mu*\varphi_{\sigma}}$ are $L^2(\RR^d)$-valued, and
\begin{align}
  n \kl{\hat \mu_n*\gamma_\sigma}{\hat \nu_n*\gamma_\sigma} \trightarrow{d}    \frac 12 \int_{\RR^d} \frac{\big(G_{\mu,\sigma}(x)-\tilde G_{\mu,\sigma}(x)\big)^2}{\mu*\varphi_{\sigma}(x)}d x, \label{eq:KL-GS-twosamp-null}
\end{align}
where the limit can be represented as a weighted sum of i.i.d. $\chi^2$ random variables with 1 degree of freedom.
\medskip
\item (Two-sample alternative) If  $\mu,\nu$ have compact supports, then 
\begin{align}
   & n^{\frac 12}\big(\kl{\hat \mu_n*\gamma_\sigma}{\hat \nu_n*\gamma_{\sigma}}-\kl{\mu*\gamma_{\sigma}}{\nu*\gamma_{\sigma}}\big)\trightarrow{d} N\big(0,v_{2,f_{\mathsf{KL}}}^2(\mu,\nu,\sigma)\big), \label{eq:KL-GS-twosamp-alt}
\end{align}
where  $v_{2,f_{\mathsf{KL}}}^2(\mu,\nu,\sigma)$ is given in \eqref{eq:var-smoothedKL2samp}.
\end{enumerate}
\end{prop}
The proof of bulk of the claims in Proposition \ref{Prop:GS-KL-limdist} hinges on  Theorem \ref{Thm:KLdiv-limdist} by identifying primitive conditions in terms of $\mu,\nu$, and $\sigma$ that guarantee the regularity assumptions in Theorem \ref{Thm:KLdiv-limdist}. The proof of the final claim in Part $(i)$ above, i.e., the necessity of Condition \eqref{ASSUM:KL_null}, relies on the following lemma, which is also applicable to other $f$-divergences  with a twice differentiable $f$.

\begin{lemma}[Lower bound on expected $f$-divergence limit] \label{lem:fdivnullonesamplb}
Consider the $f$-divergence in Definition \ref{def:f divergence}. Assume that $f$ is continuously twice differentiable in $(0,\infty)$ with a nonnegative second  derivative $f''$. Then,
\begin{align}
 \liminf_{n \rightarrow \infty} n \EE\big[\fdiv{\hat \mu_n*\gamma_\sigma} {\mu*\gamma_\sigma}\big] &\geq  \frac{f''(1)}{2}\int_{\RR^d}\frac{\mathsf{Var}_{\mu}\big(\varphi_{\sigma}(x-\cdot)\big)}{\mu*\varphi_{\sigma}(x)}\, dx. \label{eq:lowerbndfdiv}
\end{align}
\end{lemma}
The proof of Lemma \ref{lem:fdivnullonesamplb} is given in Section \ref{lem:fdivnullonesamplb-proof}, and is based on Taylor's theorem.

\begin{remark}[Simplified null limit distribution]\label{Rem:variance-alt}
The RHS of \eqref{eq:KL-GS-onesamp-null} and \eqref{eq:KL-GS-twosamp-null}  admit a representation as a weighted sum of i.i.d. $\chi^2$ random variables (with 1 degree of freedom). This follows since centered $L^2(\RR^d)$-valued Gaussian random variables admit a representation of the form
$\sum_{i \in \NN} W_i b_i$, where $(W_i)_{n \in \NN}$ is an i.i.d. sequence of standard Gaussian (real) random variables, and $(b_i)_{i \in \NN}$ is an orthonormal basis of the  reproducing kernel Hilbert space (see \cite[Theorem 4.3]{VanderVaart-2008}) associated with the $L^2(\RR^d)$-valued Gaussian variable. Hence, the RHS of \eqref{eq:KL-GS-onesamp-null} and \eqref{eq:KL-GS-twosamp-null}  can be represented as  $\sum_{i \in \NN}\norm{b_i}_{2}^2W_i^2/2$, from which the claim follows. The same conclusion also applies to the other divergences, except for  TV distance.   
\end{remark}
\begin{remark}[Unequal sample sizes]
While we consider equal number of samples from both $\mu$ and $\nu$ for simplicity, the results readily extend to the mismatched scenario. Suppose $n$ and $m$ denote the number of samples from $\mu$ and $\nu$, respectively, such that $m/(n+m)\rightarrow \tau$ for some $\tau \in (0,1)$ as $m,n \rightarrow \infty$. By minor modifications to the proof of \eqref{eq:KL-GS-onesamp-alt} and \eqref{eq:KL-GS-twosamp-alt}, one may verify that under the same respective conditions, we have
\begin{align}
  \frac{nm}{n+m} \kl{\hat \mu_n*\gamma_\sigma}{\hat \nu_m*\gamma_\sigma} \trightarrow{d}   \frac 12 \int_{\RR^d} \frac{\big(\sqrt{\tau} G_{\mu,\sigma}(x)-\sqrt{1-\tau}\tilde G_{\mu,\sigma}(x)\big)^2}{\mu*\varphi_{\sigma}(x)}d x, \label{eq:limdistnullmism}
\end{align}
and 
\begin{align}
   & \left(\frac{nm}{n+m}\right)^{\frac 12}\big(\kl{\hat \mu_n*\gamma_\sigma}{\hat \nu_m*\gamma_{\sigma}}-\kl{\mu*\gamma_{\sigma}}{\nu*\gamma_{\sigma}}\big) \trightarrow{d} N\big(0,  \bar v_{2,f_{\mathsf{KL}}}^2(\tau,\mu,\nu,\sigma)\big), \notag
\end{align}
where $ \bar v_{2,f_{\mathsf{KL}}}^2(\tau,\mu,\nu,\sigma):=\sum_{1 \leq i,j \leq 2} \int_{\RR^d}\int_{\RR^d} c_{i,j}(\tau) \Sigma^{(i,j)}_{\mu,\nu,\sigma}(x,y) L_{i,f_{\mathsf{KL}}}(x)L_{j,f_{\mathsf{KL}}}(y) d x\, d y$, with $c_{1,1}(\tau)$ $=\tau$, $c_{1,2}(\tau)=c_{2,1}(\tau)=\sqrt{\tau(1-\tau)}$, and $c_{2,2}(\tau)=1-\tau$. 
Similar claims hold for other $f$-divergences. Also, note that Remark~\ref{Rem:variance-alt} applies to~\eqref{eq:limdistnullmism} as $\big(\sqrt{\tau} G_{\mu,\sigma}-\sqrt{1-\tau}\tilde G_{\mu,\sigma}\big)/\sqrt{\mu*\varphi_{\sigma}}$ is an $L^2(\RR^d)$-valued Gaussian random variable.
\end{remark}

 \subsubsection{Bootstrap consistency}
The limit distribution from Theorem \ref{Prop:GS-KL-limdist} are non-pivotal in the sense that they depend on the population distributions $\mu$ and $\nu$, which are unknown in practice. To circumvent this difficulty and enable inference (e.g., construction of confidence intervals) we apply the bootstrap, which is a computationally tractable resampling method for estimating distributional limits. This section establishes consistency of the bootstrap for the Gaussian-smoothed KL divergence. 
\medskip

Given the data $(X_1,Y_1),\ldots,(X_n,Y_n)$, let $(X_1^B,\ldots,X_n^B) \sim \hat{\mu}_n^{\otimes n}$ and $(Y_1^B,\ldots,Y_n^B)\sim\hat{\nu}_n^{\otimes n}$ denote bootstrap samples drawn independently from  empirical distributions $\hat{\mu}_n$ and $\hat{\nu}_n$, respectively. Set $\hat{\mu}_n^B = n^{-1}\sum_{i=1}^n \delta_{X_i^B}$ and $\hat{\nu}_n^B = n^{-1}\sum_{i=1}^n \delta_{Y_i^B}$ as the bootstrap empirical distributions. Denoting by $\PP_{B}$ the conditional probability given $(X_1,Y_1),\ldots,(X_n,$ $Y_n)$, we have the following bootstrap consistency claim for the Gaussian-smoothed KL divergence limit distribution from Proposition  \ref{Prop:GS-KL-limdist}.

\begin{cor}[Bootstrap consistency for KL divergence limit distribution] \label{Cor:GS-KL-bs} Consider the setting of Proposition \ref{Prop:GS-KL-limdist}, and let $v_{1,f_{\mathsf{KL}}}^2(\mu,\nu,\sigma)$ and $v_{2,f_{\mathsf{KL}}}^2(\mu,\nu,\sigma)$ be as given in  \eqref{eq:var-smoothedKL1samp} and~\eqref{eq:var-smoothedKL2samp}, respectively. Then, 
the following hold:
\begin{enumerate}[(i)]
 
\item  (One-sample null and alternative)  Under the conditions of Proposition \ref{Prop:GS-KL-limdist}$(ii)$,  
\begin{align}
\sup_{t \in \RR}\abs{\PP_{B}\left(n^{\frac 12}\left(\mathsf{D}_{\mathsf{KL}}\big(\hat \mu_n^B*\gamma_\sigma \|\nu*\gamma_{\sigma}\big)-\kl{\hat \mu_n*\gamma_{\sigma}}{\nu*\gamma_{\sigma}} \right)\leq t\right)-\PP(W_1 \leq t)} =o_{\PP}(1),   \notag
\end{align}
where  $W_1 \sim N\big(0,v_{1,f_{\mathsf{KL}}}^2(\mu,\nu,\sigma)\big)$.
\medskip
 \item (Two-sample null and alternative) If  $\mu,\nu$ have compact supports, then 
 \begin{align}
     \sup_{t \in \RR}\abs{\PP_{B}\left(n^{\frac 12}\left(\mathsf{D}_{\mathsf{KL}}\big(\hat \mu_n^B*\gamma_\sigma \|\hat \nu_n^B*\gamma_{\sigma}\big)-\kl{\hat \mu_n*\gamma_{\sigma}}{\nu_n*\gamma_{\sigma}}\right) \leq t\right)-\PP(W_2 \leq t)}=o_{\PP}(1), \notag
 \end{align}
where  $ W_2 \sim N\big(0,$ $ v_{2,f_{\mathsf{KL}}}^2(\mu,\nu,\sigma)\big)$.
\end{enumerate}
\end{cor}
The proof follows from  \cite[Theorem 23.9]{Va1998} and the linearity of the first order Hadamard derivative of the KL divergence functional (see Proposition \ref{Prop:Hadamarddiff-gen}), once we verify that the tangent cone (domain of the Hadamard derivative) contains a non-trivial linear subspace. The details are provided in Section \ref{Sec:bsconsist-proof}. Bootstrap consistency results for other $f$-divergences (except TV distance) may be derived in an analogous manner, but we skip the formal statements below for brevity.

\section{Application to Auditing Differential Privacy} \label{Sec:appkldiffprivacy}
We consider the application of our limit distribution theory to auditing DP, which was introduced in \cite{DMKS-2006} as an approach for quantifying privacy leakage of privatization mechanisms. We recall some DP notions that are relevant to our setting. Consider a set~$\mathfrak{U}$ with a relation~$\sim$ such that $u \sim v$, for $u,v \in \mathfrak{U}$, denotes that $u$ and $v$ are adjacent. In the DP context, $\mathfrak{U}$ is a set of databases and $u\sim v$ denotes that $u$ and $v$ are adjacent databases, differing on a single entry. 

\begin{definition}[DP mechanisms\cite{DMKS-2006,DKMMN-2006,WLF-2016}]
Let $\epsilon, \delta \geq 0$. A randomized (measurable) mechanism  $f:\mathfrak{U} \rightarrow \RR^d$ is 
\begin{enumerate}[(i)]
    \item $\epsilon$-differentially private if  $\PP\left(f(u) \in \cT\right) \leq e^{\epsilon}\,\PP\left(f(v) \in \cT\right)$ for every $u \sim v$ and $\cT \in \cB(\RR^d)$;
   \item $(\epsilon,\delta)$-differentially private if $\PP\left(f(u) \in \cT\right) \leq e^{\epsilon}\,\PP\left(f(v) \in \cT\right)+\delta$  for every $u \sim v$ and  $\cT \in \cB(\RR^d)$;
   \item $\epsilon$-KL differentially private if $\kl{\mu_{u}}{\mu_{v}} \leq \epsilon$ for every $u \sim v$,  where $\mu_{u} \in \cP(\RR^d)$ is the distribution of $f(u)$.
\end{enumerate}
\end{definition}
In addition, we say that a privacy mechanism is $\epsilon$-smoothed KL differentially private~if $\kl{\mu_u*\gamma_{\sigma}}{\mu_v*\gamma_{\sigma}} \leq \epsilon$ for every $u \sim v$,  where $\sigma>0$ is a pre-specified parameter.

\medskip
Standard noise-injection mechanisms for DP operate as follows. Consider $g:\mathfrak{U} \rightarrow \RR^d$, a (deterministic) query function to be published about the database and define its  $\ell^r$-sensitivity as $\Delta_r(g):=\sup_{u,v \in \mathfrak{U}:\, u \sim v} \|g(u)-$ $g(v)\|_r$. A noise-injection DP mechanism is given by $f(\cdot)=g(\cdot)+W$, where $W$ is an additive noise random variable. For instance, the Laplace mechanism\cite{Dwork-Roth-2014} takes $W \sim \mathrm{Lap}(0,b)^{\otimes d}$ (i.e., whose Lebesgue density is $\propto e^{-\|\,\cdot\,\|_1/b}$) with $b \geq \Delta_1(g)/\epsilon$, which guarantees $\epsilon$-DP. Similarly, the Gaussian mechanism sets $W \sim \gamma_\sigma$ with $\sigma\geq \Delta_2(g)\sqrt{2\log (1.25/\delta)}/\epsilon$, which guarantees $(\epsilon,\delta)$-DP. While the above DP mechanisms add unbounded noise, a privacy mechanism that adds bounded noise in an adaptive query setting and ensures $(\epsilon,\delta)$-DP (asymptotically) for any $\delta>0$  is proposed in \cite{Dagan-Kur-2022}. As $\epsilon$-DP is equivalent to $\sup_{u\sim v}\mathsf{D}_\infty(\mu_u\|\mu_v)\leq \epsilon$, where $\mathsf{D}_\infty$ is the $\infty$-order Renyi divergence, it is clear that KL DP is a relaxation of DP.\footnote{In particular,  $\epsilon$-DP implies $\epsilon\big(1 \wedge (e^{\epsilon}-1)\big)$-KL DP by \cite[Lemma 3.18]{Dwork-Roth-2014}.} By the data processing inequality, we further have that smoothed KL DP is a relaxation of KL DP.

In practice, given output samples from a privacy mechanism, one encounters the problem of ascertaining whether the mechanism is differentially private or not, referred to as auditing DP. In \cite{domingo2022auditing}, a hypothesis test for auditing DP using regularized kernel R\'{e}nyi divergence is proposed, where the null hypothesis is that the mechanism satisfies $(\epsilon,\delta)$-DP. The authors propose a decision rule achieving any non-zero significance level (type I error probability), leaving the characterization of the power (equivalently, type II error probability) open. Here, utilizing the limit distribution theory from Section \ref{Sec:GS limit theorems}, we put forth a principled hypothesis testing pipeline for DP auditing using the Gaussian-smoothed and classic KL divergence as the privacy measures of interest. Our analysis accounts for both significance and power of the test. We start from the smoothed KL DP test.

\subsection{Smoothed KL DP test} The main objective of a privacy audit is to identify violations. For that reason, we set up an hypothesis test where the null $H_0$ corresponds to when privacy holds, and consider a sequence of local alternatives $H_{1,n}$ that become harder to distinguish from $H_0$ as $n$ grows. This models a situation where the alternative hypothesis is arbitrarily close to the null, and we seek a powerful test that successfully rejects the null, even under these local alternatives. To define the local alternatives, we consider a sequence of privacy mechanisms that violate $\epsilon$-smoothed KL DP by an $O(n^{-1/2})$ amount. 

Fix $\sigma,\epsilon,b,C>0$ and, for $n\in\NN_0$, let $f_n:\mathfrak{U}\to\cI_b:=[-b,b]^d$ be a sequence of 
privacy mechanisms. Denote a pair adjacent databases\footnote{The results  in this section do not depend on the database distribution $\tilde \pi$ per say. Also, note that the current model subsumes the case of deterministic $(U,V)$ by taking $\tilde \pi$ to be a point mass on a pair of adjacent databases. In this case, the randomness in $\big(f_n(U),f_n(V))$ only comes from the mechanism.} by $(U,V) \sim \tilde \pi \in \cP(\mathfrak{U} \times \mathfrak{U})$.  
Let $\pi_n:=(f_n,f_n)_\#\tilde \pi$ be the joint distribution of $\big(f_n(U),f_n(V)\big)$, where $\#$ is the pushforward operation. 
The first and second marginals of $\pi_n$ are denoted by $\mu_n$ and $\nu_n$, respectively. We impose the following assumption on the sequence $(\pi_n)_{n\in\NN_0}$.

\begin{assump} \label{Assump-localalt}
The sequence $(\pi_n)_{n \in \NN_0}$ is such that
\begin{enumerate}[(i)]
    \item there exists $0\neq h\in L^2(\pi_0)$ with $n \helsq{\pi_n}{\pi_0} \rightarrow \norm{h/2}^2_{2,\pi_0}$, $\int_{\RR^d \times \RR^d}h\, d\pi_0=0$, and 
    \begin{align}
   & \left( n^{1/2}\left(\frac{\mu_n*\varphi_{\sigma}-\mu_0*\varphi_{\sigma}}{\nu_0*\varphi_{\sigma}}\right), n^{1/2}\left(\frac{\nu_n*\varphi_{\sigma}}{\nu_0*\varphi_{\sigma}}-1\right)\right) \notag \\
    &\qquad \qquad \qquad\longrightarrow   \left(\frac{\EE_{\pi_0}\left[h(X,Y)\varphi_{\sigma}(\cdot-X)\right]}{\nu_0*\varphi_{\sigma}},\frac{\EE_{\pi_0}\left[h(X,Y)\varphi_{\sigma}(\cdot-Y)\right]}{\nu_0*\varphi_{\sigma}}\right) \mbox{ in }L^{\infty}(\lambda) \times L^{\infty}(\lambda).\label{eq:convmarglocalt}
    \end{align}
    \item $\kl{\mu_0*\gamma_\sigma}{\nu_0*\gamma_\sigma}\leq \epsilon$ and $\kl{\mu_n*\gamma_\sigma}{\nu_n*\gamma_\sigma} \geq \epsilon_{n,C}:=\epsilon +Cn^{-1/2}$ for all $n$ sufficiently large.
\end{enumerate}
\end{assump}
Observe that Assumption \ref{Assump-localalt}$(ii)$ implies that $f_0$ satisfies $\epsilon$-smoothed KL DP while $f_n$  violates it for all  $n$ sufficiently large. On the other hand, Assumption \ref{Assump-localalt}$(i)$ is a technical requirement which ensures that Gaussian-smoothed KL divergence limit theorems relevant for our purpose continue to hold in the local alternatives setting in which $\pi_n$ changes with $n$. Proposition \ref{prop:pertconstr} below presents an explicit construction of $(\pi_n)_{n \in \NN_0}$ that satisfies Assumption \ref{Assump-localalt} for any $\sigma,\epsilon,b,C>0$. 
For now, under this assumption, consider the following binary hypothesis test with a sequence of alternatives:
\begin{align}\label{HT-smoothed-KL-DP}
\begin{split}
    H_0&: \kl{\mu_0*\gamma_{\sigma}}{\nu_0*\gamma_{\sigma}} \leq \epsilon, \\
    H_{1,n}&: \kl{\mu_n*\gamma_{\sigma}}{\nu_n*\gamma_{\sigma}}  \geq  \epsilon_{n,C}.
    \end{split}
\end{align}

Let $(X_1,Y_1),\ldots,(X_n,Y_n) \sim \pi $ be pairwise i.i.d. samples of the privacy mechanism's output when acting on i.i.d.
pairs of adjacent databases, where  $ \pi=\pi_0$ under $H_0$ and $\pi=\pi_{n}$ under $H_{1,n}$. Denote the empirical measures of $(X_1,\ldots,X_n)$ and $(Y_1,\ldots,Y_n)$ by $\hat \mu_{n}$ and $\hat \nu_{n}$, respectively. 
For a test statistic $T_n=T_n(X_1,\ldots,X_n,Y_1,\ldots,$ $Y_n)$, a standard class of tests rejects $H_0$ if $T_n>t_n$, where $t_n$ is a critical value chosen according to the desired level $\tau\in(0,1)$. The operational meaning of rejecting $H_0$ is declaring that $\epsilon$-smoothed KL DP is violated, and hence, also $\epsilon$-DP itself. We say that such a sequence has \textit{asymptotic level $\tau$} if $\limsup_{n\to\infty}\PP(T_n>t_n|H_0)\leq \tau$. The \textit{power} of a test is the probability that it correctly rejects $H_0$, i.e., $\PP(T_n> t_n|H_{1,n})$, and the \textit{asymptotic power} is $\liminf_{n\to\infty}\PP(T_n> t_n|H_{1,n})$. Lastly, the sequence of tests is called \textit{asymptotically consistent} if its asymptotic power is 1. 
The above definitions specialize to the case of a fixed alternative $H_1$ by taking $H_{1,n}=H_1$ and $\pi_n=\pi_1$ for all $n \in \NN$.
 
\medskip
For $\tau \in [0,1]$, let $Q^{-1}(\tau)=\inf\big\{z\in \RR:(2 \pi)^{-1/2} \int_{z}^{\infty} e^{-u^2/2}du \leq \tau\big\}$ be the inverse complimentary cumulative distribution function of $N(0,1)$.
The following proposition provides a test statistic for the above hypothesis test and characterizes its asymptotic level and asymptotic power against local alternatives.   

\begin{prop}[Smoothed KL DP audit]\label{prop:HTperf}
Suppose Assumption \ref{Assump-localalt} holds. For $0<\tau,\tau' \leq 1$, let $C_{b,d,\sigma,\tau,\tau'}=c_{b,d,\sigma}\big(Q^{-1}(\tau)-Q^{-1}(1-\tau')\big)$, where $c_{b,d,\sigma}$ is given in \eqref{eq:threshold-HT}. Then the test statistic $T_n=\kl{\hat \mu_{n}*\gamma_{\sigma}}{\hat \nu_{n}*\gamma_{\sigma}}$ with critical value $t_n= \epsilon+c_{b,d,\sigma}Q^{-1}(\tau) n^{-1/2}$
achieves an asymptotic level $\tau$ and asymptotic power at least $1-\tau'$ for the test in \eqref{HT-smoothed-KL-DP}, whenever $C>C_{b,d,\sigma,\tau,\tau'} \vee 0$.
\end{prop}

The proof of the above claim relies on the  limit distribution result for smoothed KL divergence given in \eqref{eq:KL-GS-twosamp-alt} along with its refinement to account for the local alternatives scenario, i.e., to  account for a sequence of distributions $(\mu_n*\gamma_{\sigma},\nu_n*\gamma_{\sigma})_{n \in \NN}$ instead of a fixed one. This refinement (see \eqref{eq:localaltlimdist}) is derived  under Assumption \ref{Assump-localalt} by invoking Le Cam's third lemma \cite[Theorem 3.10.7]{AVDV-book}. Given these results and the fact that the relevant limit distributions are Gaussian,  the claim in Proposition \ref{prop:HTperf} follows by an analysis of the asymptotic level and power via the Portmanteau theorem \cite[Theorem 1.3.4]{AVDV-book}. Note that the constant $C_{b,d,\sigma,\tau,\tau'}$ is positive whenever $\tau+\tau'< 1$, which is when the requirement $C>C_{b,d,\sigma,\tau,\tau'}$ is active. Operationally, $\tau+\tau'< 1$ means that the sum of type I and type II error probabilities is less than 1, which is the interesting regime for hypothesis testing; otherwise, a test based on a random coin flip is preferable.

\medskip
We conclude this part by providing an explicit construction of a sequence of couplings $(\pi_n)_{n\in\NN_0}$ that satisfies Assumption \ref{Assump-localalt}. 
\begin{prop}[Construction satisfying Assumption \ref{Assump-localalt}]\label{prop:pertconstr}
\begin{enumerate}[(i)]
    \item 
Let $\pi_0 \in \cP(\RR^d\times \RR^d)$ be such that $\mu_0 \neq \nu_0$, $\nu_0 \otimes \mu_0 \ll \pi_0$, $\norm{d(\nu_0 \otimes \mu_0)/d\pi_0}_{\infty,\pi_0} 
<\infty$ and $\norm{h_{\pi_0,\bar c}}_{2,\pi_0}<\infty$, where
\begin{align}
   h_{\pi_0,\bar c}:=\bar c\left(\frac{d (\mu_0 \otimes \nu_0)}{d \pi_0}-\frac{d (\nu_0 \otimes \mu_0)}{d \pi_0}\right), \label{eq:pertjntkl}
\end{align} 
and $\bar c>0$ is an arbitrary constant. Let $\pi_n \ll \pi_0$ be the probability measure specified by the relative density
\begin{align}
\frac{d\pi_n}{d\pi_0}=1+n^{-\frac12}h_{\pi_0,\bar c},\label{eq:pertconst}
\end{align}
whenever the RHS is non-negative $\pi_0$-a.s.; otherwise, set $\pi_n=\pi_0$. Then, $\pi_n$ satisfies Assumption \ref{Assump-localalt}$(i)$ with $h=h_{\pi_0,\bar c}$.
\item Let $\pi_0 \in \cP(\cI_b \times \cI_b)$ and $\sigma$ be such that $\mu_0\neq\nu_0$, $\nu_0  \otimes  \mu_0  \ll  \pi_0 $, $\norm{d(\nu_0 \otimes \mu_0)/d\pi_0}_{\infty,\pi_0} $ $\vee \norm{d(\mu_0 \otimes \nu_0)/d\pi_0}_{2,\pi_0} $ $
<\infty$ and $ \kl{\mu_0*\gamma_{\sigma}}{\nu_0*\gamma_{\sigma}}=\epsilon$. 
Then, there exists a sufficiently large $\bar c$, such that $\pi_n$ as  defined in \eqref{eq:pertconst} satisfies Assumption \ref{Assump-localalt} with $h=h_{\pi_0,\bar c}$ for any $C>0$. 
\end{enumerate}
\end{prop}
Proposition \ref{prop:pertconstr}$(ii)$, which is proven in Section \ref{prop:pertconstr-proof}, provides a method  of constructing $\pi_n$ for the hypothesis test in \eqref{HT-smoothed-KL-DP}, given  $\pi_0$ that satisfies the aforementioned regularity assumptions. 
This can be achieved, for instance, by choosing $\pi_0$ such that $ \mu_0 \ll \gg \nu_0 \ll \lambda$, $\norm{d\nu_0 /d\mu_0}_{\infty} \vee \norm{d\mu_0 /d\nu_0}_{\infty} 
<\infty$ and $ \kl{\mu_0*\gamma_{\sigma}}{\nu_0*\gamma_{\sigma}}=\epsilon$.

\subsection{KL DP test} A more stringent DP audit is realized via a hypothesis test for detecting $\epsilon$-KL DP violations, rather than its Gaussian-smoothed version. We next provide such a test against a fixed alternative, namely: 
\begin{align} \label{HT-KL-DP}
\begin{split}
  & H_0: \kl{\mu_0}{\nu_0} \leq \epsilon,   \\
  & H_{1}: \kl{\mu_1}{\nu_1} \geq  \tilde{\epsilon},
\end{split}
\end{align}
where $\tilde \epsilon>\epsilon>0$. For this test, we again employ the test statistic $T_n$ from Proposition~\ref{prop:HTperf} with appropriately chosen $\sigma$. Doing so requires additional assumption on the output distributions of the DP mechanism, namely that $\mu_i,\nu_i$, for $i=0,1$, to have smooth Lebesgue densities belonging to the following class. 
\begin{definition}[Lipschitz class, see \cite{DL-1993}] \label{def:Lipschitzclass} 
For  $r \in (0,\infty]$, $m \in \NN$, 
and $f \in L^r\big(\RR^d\big)$, the $m$-th modulus of smoothness of $f$ is
\begin{equation}
  \kappa_{m,r}(f,t):=\sup_{y \in \RR^d, \norm{y} \leq t}\norm{\Delta_{y}^m f}_{r}, \label{modsmdefn}
\end{equation}
where $\Delta_{y}^m f(x)=\sum_{j=0}^{m} (-1)^{m-j}f(x+jy)$. For $0<s \leq 1$ and $\cX \subseteq \RR^d$,
 the Lipschitz class with smoothness parameter $s$ and norm parameter $M$ is
\begin{align}
\mathrm{Lip}_{s,r}(M,\cX):=\big\{f \in L^r\big(\RR^d\big): \norm{f}_{\mathrm{Lip}(s,r)} \leq M, \supp(f)\subseteq \cX \big\}, \notag
\end{align}
where 
$\norm{f}_{\mathrm{Lip}(s,r)}:=\norm{f}_{r}+\sup_{t>0} t^{-s}\kappa_{1,r}(f,t)$ is the Lipschitz seminorm. 
\end{definition} 
\begin{assump} \label{Assump:KLDPtest}
For $i=0,1$, the Lebesgue densities
$p_{\mu_i}, p_{\nu_i} \mspace{-4 mu}\in \mspace{-3mu}\mathrm{Lip}_{s,1}\mspace{-2 mu}(M,\cI_b)$ and $ \norm{p_{\mu_i}/p_{\nu_i}}_{\infty} $ $\vee \norm{p_{\nu_i}/p_{\mu_i}}_{\infty} \leq M$ for some $0<s \leq 1$ and $M>0$. Further, $\kl{\mu_0}{\nu_0} \leq \epsilon$ and $\kl{\mu_1}{\nu_1} \geq \tilde \epsilon$ for some $\tilde \epsilon>\epsilon>0$.
\end{assump}

Assumption \ref{Assump:KLDPtest} is not very restrictive in practice. Indeed, the definition of DP itself necessitates that $\norm{p_{\mu_u}/p_{\mu_{v}}}_{\infty}$ is bounded uniformly for all $u,v \in \mathfrak{U}$ with $u \sim v$. Moreover, the class of Lipschitz functions grows as we shrink the smoothness parameter $s$, and thus $\cup_{M \geq 0} \mathrm{Lip}_{1,1}(M,\cI_b) \subseteq \cup_{M \geq 0} \mathrm{Lip}_{s,1}(M,\cI_b)$ under our assumption that $0<s\leq 1$. As the class of functions with bounded variation (for $d=1$) over $\cI_b$ is contained in $\cup_{M \geq 0}\mathrm{Lip}_{1,1}(M,\cI_b)$, Assumption \ref{Assump:KLDPtest} allows for most densities of practical interest. 
\medskip

We are now ready to state the $\epsilon$-KL DP audit result. As it may be unrealistic to assume that the exact values of $M$, $s$, and $\tilde \epsilon$ are known when constructing $T_n$ and choosing critical values, the following proposition only requires the existence of known constants $\bar M$ ,$\bar \epsilon$, $\underaccent{\bar}{s}$, and $\bar s$  such that $M \leq \bar M<\infty$, $\epsilon<\bar \epsilon\leq \tilde \epsilon$, and $0 <\underaccent{\bar}{s} \leq s \leq \bar s \leq 1$.
\begin{prop}[KL DP audit]\label{prop:HTperfKLDP}
Suppose Assumption \ref{Assump:KLDPtest} holds. Let $0<\tau\leq 1$ and $0<\sigma<\sigma_{ \epsilon,\bar \epsilon,\underaccent{\bar}{s},\bar s,d,\bar M}$, where $\sigma_{ \epsilon,\bar \epsilon,\underaccent{\bar}{s},\bar s,d,\bar M}$ is the solution of \eqref{eq:sigmachoice}. Then the test statistic $T_n=\kl{\hat \mu_{n}*\gamma_{\sigma}}{\hat \nu_{n}*\gamma_{\sigma}}$ with critical value $t_n= \epsilon+c_{b,d,\sigma} Q^{-1}(\tau) n^{-1/2}$, where $c_{b,d,\sigma}$ is given in \eqref{eq:threshold-HT}, is asymptotically consistent and achieves an asymptotic level $\tau$ for the test in \eqref{HT-KL-DP}.
\end{prop} 

The key difference between the proof of this claim and Proposition \ref{prop:HTperf} is that given $\bar M$, $\bar \epsilon$, $\underaccent{\bar}{s}$, and $\bar s$, it is possible to choose $\sigma>0$ small enough so that $\kl{\mu_1*\gamma_{\sigma}}{\nu_1*\gamma_{\sigma}}>\epsilon$ while $\kl{\mu_0*\gamma_{\sigma}}{\nu_0*\gamma_{\sigma}}  \leq \epsilon$. Choosing such a $\sigma$, the claim then follows by utilizing   \eqref{eq:KL-GS-twosamp-alt} along with the Portmanteau theorem to bound the type I and type II error probabilities associated with $T_n$. The aforementioned choice of $\sigma$ relies on a stability lemma for smoothed KL divergence given next, which may be of independent interest.
\begin{lemma}[Stability lemma for smoothed KL divergence] \label{lem:smoothKL-stability}
Let $\cX \subseteq \RR^d$, and $\mu,\nu \in \cP(\cX)$ have Lebesgue densities $p_{\mu}$ and $p_{\nu}$, respectively. Further, assume that $p_{\mu}, p_{\nu} \in \mathrm{Lip}_{s,1}(M,\cX)$ and $ \norm{p_{\mu}/p_{\nu}}_{\infty} \vee \norm{p_{\nu}/p_{\mu}}_{\infty} \leq M$ for some $M \geq 1$. Then,
\begin{align}
    \abs{\kl{\mu}{\nu}-\kl{\mu*\gamma_{\sigma}}{\nu*\gamma_{\sigma}}} \leq c_{d,s} M\left(M+1+\log M\right) \sigma^s, \label{eq:KL-stability-ub}
\end{align}
where $c_{d,s}:=\int_{\RR^d} \norm{z}^s \varphi_1(z) d z$.
\end{lemma}
The proof of Lemma \ref{lem:smoothKL-stability} upper bounds the left-hand side (LHS) of \eqref{eq:KL-stability-ub} using Taylor's theorem, and then exploits the boundedness and Lipschitzness of the densities to control the resulting terms. 

\section{Proofs} \label{Sec:proofs}
This section contains proofs of the results from Section \ref{Sec:Techframework}-\ref{Sec:appkldiffprivacy}.
\subsection{Proof of Proposition \ref{Prop:Hadamarddiff-gen}}

The derivation uses the following lemma whose proof is given in Appendix \ref{Lem:contmapLp-proof} for completeness. 
\begin{lemma}[Generalized Slutsky's theorem] \label{Lem:contmapLp}
Let $r, r' \geq 1$ be conjugate indices, i.e., $1/r+1/r'=1$. The following hold:
    \begin{enumerate}[(i)]
        \item  If $f_n \rightarrow f$ in $L^{r}(\rho)$ and $g_n \rightarrow g$ in $L^{r'}(\rho)$,   then $f_n g_n \rightarrow fg $ in $L^{1}(\rho)$.
        \item If $Y_n \trightarrow{w} Y$ in $L^{r}(\rho)$ and $Z_n \trightarrow{w} z$ in $L^{ r'}(\rho)$,  where $z$ is deterministic,  then $Y_n Z_n \trightarrow{w} Yz$ in $L^{1}(\rho)$. 
    \end{enumerate}
\end{lemma}

Having that, recall that all partial derivatives  of $\phi$ of order two exists in  $(0,\infty) \times (0,\infty)$. Consequently, the multivariate second-order Taylor's expansion of $\phi(x,y)$ around $\phi(x^{\star},y^{\star})$, where $x,y,x^\star,y^\star>0$, yields
\begin{align}  \phi(x,y)&=\phi(x^{\star},y^{\star})+(x-x^{\star})D^{(1,0)}\phi(x^{\star},y^{\star})+(y-y^{\star}) D^{(0,1)}\phi(x^{\star},y^{\star}) + (x-x^{\star})^2\int_{0}^1 D^{(2,0)}\phi(u_{\tau}) (1-\tau) d \tau \notag \\
    &\quad+ (y-y^{\star})^2\int_{0}^1 D^{(0,2)}\phi(u_{\tau})(1-\tau) d \tau + 2(x-x^{\star}) (y-y^{\star})\int_{0}^1 D^{(1,1)}\phi(u_{\tau})(1-\tau) d \tau, \notag
\end{align}
where $u_{\tau}:=(1-\tau)(x^{\star},y^{\star})+\tau(x,y)$.  Substituting $g_1(x),g_1^{\star}(x),g_2(y),g_2^{\star}(y)$ for $x,x^{\star},y,y^{\star}$, respectively, in the above equation, and setting $u_{j,\tau}=(1-\tau)g_j^{\star}+\tau g_j$ for $j=1,2$,  $\theta=(\theta_1,\theta_2)=\big(g_1-g_1^{\star},g_2-g_2^{\star}\big)$,
we obtain  
\begin{align}  \phi \mspace{-2 mu}\circ \mspace{-2 mu}(g_1,g_2)&=\phi \mspace{-2 mu}\circ \mspace{-2 mu}(g_1^{\star},g_2^{\star})+\theta_1 D^{(1,0)}\phi \circ(g_1^{\star},g_2^{\star})+\theta_2 D^{(0,1)}\phi \circ(g_1^{\star},g_2^{\star}) \mspace{-2 mu} +\mspace{-2 mu}\int_{0}^1 \mspace{-2 mu}\theta_1^2 D^{(2,0)}\phi \mspace{-2 mu}\circ\mspace{-2 mu}(u_{1,\tau},u_{2,\tau}) (1-\tau) d \tau \notag \\
    & \quad  +\int_{0}^1 \theta_2^2 D^{(0,2)}\phi\circ(u_{1,\tau},u_{2,\tau}) (1-\tau) d \tau   + 2\int_{0}^1 \theta_1 \theta_2 D^{(1,1)}\phi\circ(u_{1,\tau},u_{2,\tau})(1-\tau) d \tau, \notag
\end{align}
for all $(x,y)$ with $g_1(x),g_1^{\star}(x),g_2(y),g_2^{\star}(y)>0$.
The validity of the above equation when $g_1(x) \geq 0$, and $g_2(y),g_1^{\star}(x),$ $g_2^{\star}(y)>0$  then follows by taking limits $g_1(x) \downarrow 0$ via Assumption \ref{Assump1} (specifically, the continuity of $\phi$ at $(0,c)$ for $c>0$, and the continuity of second order partial derivatives) along with the  dominated convergence theorem applied to the last three integrals using   Assumption \ref{Assump4}. Likewise, the above equation extends to the case  $g_1(x) =0 $, $g_2(y) \geq 0$ and $g_1^{\star}(x),$ $g_2^{\star}(y)>0$ by taking  limits $g_2(y) \downarrow 0$ and using $\lim_{z \downarrow 0}\phi(0,z)=\phi(0,0)$ in Assumption \ref{Assump1}. Note that the above scenarios correspond to the constraints $\{(g_1,g_2): g_1 \geq 0,g_2 \geq 0, \supp(g_1) \subseteq \supp(g_2)\}$ in the definition of $\Theta'$.  In a similar vein, the above equation also generalizes to the case $g_1(x),g_2(y) \geq 0$, $g_1^{\star}(x),g_2^{\star}(y)>0$  by taking limits $g_1(x),g_2(y) \downarrow 0$, provided that $\phi$ is  continuous at $(0,0)$. This corresponds to the constraints defining the set $\Theta''$.

Integrating w.r.t. $\rho$  then gives 
\begin{align}
    \Phi(\theta_1,\theta_2) &=\Phi(0,0)+\int_{\mathfrak{S}}\theta_1 D^{(1,0)}\phi \circ(g_1^{\star},g_2^{\star})d \rho+\int_{\mathfrak{S}}\theta_2 D^{(0,1)}\phi \circ(g_1^{\star},g_2^{\star})d \rho  \notag \\
    & \quad +\int_{\mathfrak{S}}\int_{0}^1\theta_1^2 D^{(2,0)}\phi\circ(u_{1,\tau},u_{2,\tau}) (1-\tau) d \tau d \rho+\int_{\mathfrak{S}}\int_{0}^1 \theta_2^2 D^{(0,2)}\phi\circ(u_{1,\tau},u_{2,\tau}) (1-\tau) d \tau d \rho \notag \\
    &\quad + 2\int_{\mathfrak{S}}\int_{0}^1 \theta_1 \theta_2 D^{(1,1)}\phi\circ(u_{1,\tau},u_{2,\tau})(1-\tau) d \tau d \rho, \numberthis\label{eq:Taylorexpfunc}
\end{align}
where we have used the definition of $\Phi$ in \eqref{eq:defPhifunc}. The terms in \eqref{eq:Taylorexpfunc}  are well-defined and finite due to the following reasons. The finiteness of $ \Phi(\theta_1,\theta_2)$ and $\Phi(0,0)$  is straightforward from Assumption \ref{Assump3}, while that of the first two integrals is a consequence of Assumption~\ref{Assump2} and H\"{o}lder's inequality. The remaining integrals exists and are finite 
since for any $\alpha$ with $|\alpha|=2$, we have 
\begin{flalign}
 & \int_{\mathfrak{S}}\mspace{-2mu}\int_{0}^1\mspace{-5mu} \big|\theta_1^{2}\mspace{-2mu}D^{(2,0)}\mspace{-2mu}\phi\mspace{-3mu}\circ\mspace{-3mu}(u_{1,\tau},\mspace{-2mu}u_{2,\tau})\mspace{-3mu}+\mspace{-3mu}\theta_2^{2}D^{(0,2)}\mspace{-2mu}\phi\mspace{-3mu}\circ\mspace{-3mu}(u_{1,\tau},\mspace{-2mu}u_{2,\tau})\mspace{-3mu}+\mspace{-3mu}2\theta_1\theta_2D^{(1,1)}\mspace{-2mu}\phi\mspace{-3mu}\circ\mspace{-3mu}(u_{1,\tau},\mspace{-2mu}u_{2,\tau})\big|(1\mspace{-2mu}-\mspace{-2mu}\tau) d \tau d \rho \notag \\  
  &\lesssim \int_{\mathfrak{S}} \left(\big|\theta_1^2D^{(2,0)}\psi_{2,0}\circ(g_1^{\star},g_2^{\star})\big|+\big|\theta_2^2D^{(0,2)}\psi_{0,2}\circ(g_1^{\star},g_2^{\star})\big|+\big| \theta_1\theta_2D^{(1,1)}\psi_{1,1}\circ(g_1^{\star},g_2^{\star})\big|\right) d \rho \notag \\
 &\leq \int_{\mathfrak{S}} \left(\big|\theta_1^2D^{(2,0)}\psi_{2,0}\circ(g_1^{\star},g_2^{\star})\big|+\big|\theta_2^2D^{(0,2)}\psi_{0,2}\circ(g_1^{\star},g_2^{\star})\big|\right)\mspace{-3mu} d \rho \notag \\
 & \ \quad\qquad\qquad\qquad \qquad\qquad \qquad\qquad +\left(\int_{\mathfrak{S}}\theta_1^2\big| D^{(1,1)}\psi_{1,1}\circ(g_1^{\star},g_2^{\star})\big|d\rho\right)^{\frac 12}  \left(\int_{\mathfrak{S}}\theta_2^2\big| D^{(1,1)}\psi_{1,1}\circ(g_1^{\star},g_2^{\star})\big|d\rho\right)^{\frac 12}\notag\\
  & <\infty, \label{eq:intermbnd} &&
\end{flalign}
where the first step uses \eqref{eq:fdiv-integcond} and  $v_{ \alpha} \in L_+^1([0,1],\lambda)$, the second one invokes the Cauchy-Schwarz inequality, while the finiteness follows from \eqref{eq:normetadef} and $(\theta_1,\theta_2) \in \Theta \subset \mathfrak{D}$. 
 
 \medskip
 
Given the expansion from \eqref{eq:Taylorexpfunc}, setting $(\theta_{1,n},\theta_{2,n})=(g_{1,n}-g_1^{\star},g_{2,n}-g_2^{\star}) \in \Theta$ and 
 $h_{j,n}=\theta_{j,n}/t_n$, $t_n>0$, $j=1,2$, we obtain 
\begin{align}
 & \frac{\Phi\big(t_nh_{1,n},t_n h_{2,n}\big)-\Phi(0,0)}{t_n}=\int_{\mathfrak{S}}\Big(h_{1,n}D^{1,0}\phi \circ(g_1^{\star},g_2^{\star})+h_{2,n}D^{(0,1)}\phi \circ(g_1^{\star},g_2^{\star})\Big)d \rho +J_n, \label{eq:phininteg}
\end{align}
where $J_n:= t_n\int_{\mathfrak{S}} (h_{1,n}^2  I_{1,n} +h_{2,n}^2 I_{2,n}+h_{1,n}h_{2,n}I_{3,n}) d\rho$ with 
\begin{align}
 I_{1,n}&:=   \int_0^1  D^{(2,0)}\phi\circ(g_1^{\star}+\tau  t_n h_{1,n},g_2^{\star}+\tau t_n h_{2,n})(1-\tau) d \tau, \notag \\
 I_{2,n}&:=\int_{0}^1 D^{(0,2)}\phi\circ(g_1^{\star}+\tau  t_nh_{1,n},g_2^{\star}+\tau t_n h_{2,n})(1-\tau) d \tau,  \notag \\
  I_{3,n}&:=2\int_{0}^1 D^{(1,1)}\phi\circ(g_1^{\star}+\tau  t_nh_{1,n},g_2^{\star}+\tau t_n h_{2,n})(1-\tau) d \tau. \notag
\end{align}
We will show that 
the limit of the RHS of \eqref{eq:phininteg} as $n \rightarrow \infty$ evaluates to the RHS of \eqref{eq:Hadmardfirstderiv} for all $t_n \downarrow 0^+$ and $(h_{1,n},h_{2,n}) \rightarrow (h_1,h_2)$ in $\mathfrak{D}$ for some 
$(h_1,h_2) \in \mathfrak{T}_{\Theta}(\theta^{\star})$, which will prove \eqref{eq:Hadmardfirstderiv}. Here, $\mathfrak{T}_{\Theta}(\theta^{\star})$ is as given in Proposition \ref{Prop:Hadamarddiff-gen} since $\Theta$ is convex and $\theta^{\star} \in \Theta$.   
\medskip

Under Assumption \ref{Assump2}, the first two terms in \eqref{eq:phininteg} converge to the RHS of \eqref{eq:Hadmardfirstderiv} by Lemma~\ref{Lem:contmapLp}$(i)$ since $h_{1,n} \rightarrow h_1$ and  $h_{2,n} \rightarrow h_2$ in $L^2(\rho)$. 
Thus, it remains to show that $J_n$ converges to zero. To that end we will show that every subsequence of $J_n$ has a further subsequence that converges to 0. Fix $(t_n, h_{1,n},h_{2,n})_{n \in \NN}$ and consider a subsequence $(n_k)_{k \in \NN}$ of $\NN$. Since $(h_{1,n_k},h_{2,n_k})$ $ \rightarrow (h_1,h_2)$ in $\mathfrak{D}$ implies $h_{i,n_k}\rightarrow h_i$ in $L^2(\eta_i)$, for $i=1,2$, and $L^2$ convergence implies convergence in measure, there exists a further subsequence $(n_{k_l})_{l \in \NN}$ such that $h_{i,n_{k_l}} \rightarrow h_i$  $\eta_i$-a.e., for $i=1,2$. Consequently, \eqref{eq:fdiv-integcond} with $v_{\alpha} \in L_+^1([0,1],\lambda)$ and  dominated convergence theorem implies that
$\rho$-a.e.,
\begin{align}
   &\lim_{l \rightarrow \infty}   I_{1,n_{k_l}}\mspace{-2 mu}+\mspace{-2 mu}I_{2,n_{k_l}}\mspace{-2 mu}+\mspace{-2 mu}I_{3,n_{k_l}} 
  \mspace{-2 mu} =\mspace{-2 mu} \frac 12 \Big(D^{(2,0)}\phi \circ(g_1^{\star},g_2^{\star})  +  D^{(0,2)}\phi\circ(g_1^{\star},g_2^{\star}) +2   D^{(1,1)}\phi\circ(g_1^{\star},g_2^{\star}) \Big). \notag
\end{align}
Next, we claim that  $\big(h_{1,n_{k_l}}^2I_{1,n_{k_l}}+h_{2,n_{k_l}}^2I_{2,n_{k_l}}+h_{1,n_{k_l}}h_{2,n_{k_l}}I_{3,n_{k_l}}\big)_{l \in \NN}$ is uniformly integrable w.r.t. $\rho$. 
This along with the above equation and Vitali's convergence theorem then leads to
\begin{align}
   &\lim_{l\rightarrow \infty} \int_{\mathfrak{S}} \big(h_{1,n_{k_l}}^2I_{1,n_{k_l}}+h_{2,n_{k_l}}^2I_{2,n_{k_l}}+h_{1,n_{k_l}}h_{2,n_{k_l}}I_{3,n_{k_l}}\big)   d\rho\notag \\
   & \qquad\qquad \qquad\qquad   = \frac 12 \int_{\mathfrak{S}} \left(h_{1}^2  D^{(2,0)}\phi\circ(g_1^{\star},g_2^{\star})  +  h_{2}^2 D^{(0,2)}\phi\circ(g_1^{\star},g_2^{\star}) +2  h_{1}h_{2} D^{(1,1)}\phi\circ(g_1^{\star},g_2^{\star}) \right)   d\rho, \label{eq:conv-Iterms}
\end{align}
and hence, $\lim_{l\rightarrow \infty} J_{n_{k_l}}=0$ as $t_{n_{k_l}} \downarrow 0^+$. Thus,  every subsequence $(J_{n_{k}})_{k \in \NN}$ has a further subsequence $(J_{n_{k_l}})_{l \in \NN}$ which converges to $0$, which implies $\lim_{n\rightarrow \infty} J_{n}=0$ and shows that 
\begin{align}
   \lim_{n \rightarrow \infty}  \frac{\Phi\big(t_nh_{1,n},t_n h_{2,n}\big)-\Phi(0,0)}{t_n}&=\int_{\mathfrak{S}}\Big(h_{1}D^{1,0}\phi \circ(g_1^{\star},g_2^{\star})+h_{2}D^{(0,1)}\phi \circ(g_1^{\star},g_2^{\star})\Big)d \rho. \notag
 \end{align}
 Since the above holds for any $t_n \downarrow 0^+$ and $(h_{1,n},h_{2,n}) \rightarrow (h_1,h_2)$, \eqref{eq:Hadmardfirstderiv} follows.

\medskip

To show the uniform integrability claim mentioned above, note that for any $\cD \subset \mathfrak{S} $ 
\begin{flalign}
   & \int_{\mathfrak{S}}  \ind_{\cD} \abs{h_{1,n_{k_l}}^2I_{1,n_{k_l}}+h_{2,n_{k_l}}^2I_{2,n_{k_l}}+h_{1,n_{k_l}}h_{2,n_{k_l}}I_{3,n_{k_l}}} d \rho\notag \\
    & \leq \int_{\mathfrak{S}}  \ind_{\cD} \abs{h_{1,n_{k_l}}^2I_{1,n_{k_l}}} d \rho+\int_{\mathfrak{S}}  \ind_{\cD} \abs{h_{2,n_{k_l}}^2I_{2,n_{k_l}}} d \rho+\int_{\mathfrak{S}}  \ind_{\cD} \abs{h_{1,n_{k_l}}h_{2,n_{k_l}}I_{3,n_{k_l}}} d \rho\notag \\
    & \lesssim \int_{\mathfrak{S}} \ind_{\cD}\left(h_{1,n_{k_l}}^2\big|\psi_{2,0} \circ (g_1^{\star},g_2^{\star})\big|+h_{2,n_{k_l}}^2\big|\psi_{0,2} \circ(g_1^{\star},g_2^{\star})\big|\right) d \rho\notag\\
    &\qquad\qquad \qquad\qquad \qquad\qquad\qquad+\left(\int_{\mathfrak{S}} \ind_{\cD}h_{1,n_{k_l}}^2\big| \psi_{1,1} \circ(g_1^{\star},g_2^{\star})\big|d\rho\right)^{\frac 12}\left(\int_{\mathfrak{S}} \ind_{\cD}h_{2,n_{k_l}}^2\big| \psi_{1,1} \circ(g_1^{\star},g_2^{\star})\big|d\rho\right)^{\frac 12} \notag \\
    & \lesssim \norm{\ind_{\cD} h_{1,n_{k_l}}}_{2,\eta_1}^2+\norm{\ind_{\cD}h_{2,n_{k_l}}}_{2,\eta_2}^2. \notag &&
\end{flalign}
where the penultimate inequality above follows via similar steps leading to the bound in \eqref{eq:intermbnd}.
Also, recall that $h_{i,n_{k_l}} \rightarrow h_i$ in $L^2(\eta_i)$ implies uniform integrability of $(h_{i,n_{k_l}})_{l \in \NN}$, for $i=1,2$. This along with  the above equation then shows the desired uniform integrability.

 \medskip

To compute the second order Hadamard derivative, from \eqref{eq:phininteg}, we have 
 \begin{align}
     \frac{\Phi(t_nh_{1,n},t_nh_{2,n})-\Phi(0,0)-t_n\Phi_{\Theta}'(h_{1,n},h_{2,n})}{t_n^2}
     &=\int_{\mathfrak{S}} (h_{1,n}^2I_{1,n} +h_{2,n}^2I_{2,n}+h_{1,n}h_{2,n}I_{3,n}) d\rho. \notag
 \end{align} 
Using similar argument leading to \eqref{eq:conv-Iterms} one readily shows that the RHS multiplied by 2 above converges to the RHS of \eqref{eq:Hadmardsecderiv}.  Since this holds for any $t_n \downarrow 0^+$ and $(h_{1,n},h_{2,n}) \rightarrow (h_1,h_2)$, \eqref{eq:Hadmardsecderiv} follows, thus completing the proof of the proposition. 
\subsection{Proofs for Section \ref{Sec:limittheorems}} \label{Sec:limittheorems-proof}
We first state a lemma that shows that Assumption \ref{Assump1} and Assumption \ref{Assump4} are satisfied by the functionals corresponding to $f$-divergences.
\begin{lemma} \label{Lem:assumpfdiv}
Consider $\rho$, $g_1^{\star},g_2^{\star}$ and $\Theta$ as in Section \ref{Sec:Techframework}. The following hold:
\begin{enumerate}[(i)]
    \item The functions
\begin{align}
       \phi_{\mathsf{KL}}(x,y)&:=y f_{\mathsf{KL}}(x/y)=x \log (x/y), \notag \\
    \phi_{\chi^2}(x,y)&:=yf_{\chi^2}(x/y)= (x-y)^2/y, \notag \\
    \phi_{\mathsf{H}^2}(x,y)&:=yf_{\mathsf{H}^2}(x/y)=(\sqrt{x}-\sqrt{y})^2, \notag
\end{align}
are convex and satisfy Assumption \ref{Assump1}. Moreover,  
$\phi_{\mathsf{H}^2}$ is continuous in $\RR_{\geq 0} \times \RR_{\geq 0}$, while $\phi_{\mathsf{KL}}$ and $\phi_{\chi^2}$ are continuous in $[0,\infty) \times (0,\infty)$.  
\medskip

\item The above functions satisfy Assumption \ref{Assump4} under conditions enumerated below: 
\begin{enumerate}[(a)]
    \item $\phi_{\mathsf{KL}}$ with $\psi_{2,0}\circ\big(g_1^{\star},g_2^{\star}\big)=1/g_1^{\star}$, $\psi_{0,2}\circ\big(g_1^{\star},g_2^{\star}\big)=g_1^{\star}/g_2^{\star2}+q/g_2^{\star}$,  $\psi_{1,1}\circ\big(g_1^{\star},g_2^{\star}\big)=1/g_2^{\star}$,  $v_{2,0}=v_{0,2}=v_{1,1}=1$, and 
 \begin{equation}
  \Theta \subseteq \bar\Theta(q) := \left\{ \begin{aligned}
   (g_1-g_1^{\star},g_2-g_2^{\star}) \in \mathfrak{D}&:g_1 \geq 0,g_2 \geq  0, \supp(g_1) \subseteq \supp(g_2),  \norm{g_1}_{1,\rho} \vee \norm{g_2}_{1,\rho} \leq 1,\\
   &\abs{g_1/g_2} \leq q,~ \rho  \mbox{-a.e.}    
  \end{aligned}\right\}.   \label{eq:thetbarqset}
 \end{equation}
    \item $\phi_{\chi^2}$ with  $\psi_{2,0}\circ\big(g_1^{\star},g_2^{\star}\big)=1/g_2^{\star}$, $\psi_{0,2}\circ\big(g_1^{\star},g_2^{\star}\big)=g_1^{\star2}/g_2^{\star3}+(q^2/g_2^{\star})$,  $\psi_{1,1}\circ\big(g_1^{\star},g_2^{\star}\big)=g_1^{\star}/g_2^{\star2}+q/g_2^{\star}$,  $v_{2,0}=v_{0,2}=v_{1,1}=1$, and  $\Theta \subseteq \bar \Theta(q)$;
\item $\phi_{\mathsf{H}^2}$ with $\psi_{2,0}\circ\big(g_1^{\star},g_2^{\star}\big)=\big(g_2^{\star1/2}+q_2^{1/2}\big)/g_1^{\star3/2}$, $\psi_{0,2}\circ\big(g_1^{\star},g_2^{\star}\big)=\big(g_1^{\star1/2}+q_1^{1/2}\big)/g_2^{\star3/2}$, $\psi_{1,1}\circ\big(g_1^{\star},g_2^{\star}\big)=\big(g_1^{\star}g_2^{\star}\big)^{-1/2}$,  $v_{2,0}=v_{0,2}=\tau^{1/2}(1-\tau)^{-1/2}$,  $v_{1,1}=1$, and  
\begin{flalign}
  \Theta \subseteq \check \Theta(q_1,q_2)\mspace{-4 mu}:=\mspace{-4 mu}\left\{(g_1-g_1^{\star},g_2-g_2^{\star}) \mspace{-2 mu}\in \mspace{-2 mu}\mathfrak{D}: g_1,g_2 \geq 0,  \norm{g_1}_{1,\rho}\mspace{-2 mu} \vee \mspace{-2 mu}\norm{g_2}_{1,\rho}\mspace{-3 mu} \leq \mspace{-3 mu}1,|g_1| \leq q_1,|g_2| \leq q_2,\mspace{1 mu} \rho\mbox{-a.e.}\right\}. \notag
\end{flalign}
\end{enumerate}
\end{enumerate}
\end{lemma}
The proof of Lemma \ref{Lem:assumpfdiv} is given in Appendix \ref{Lem:assumpfdiv-proof}. We proceed with the proofs of the results in Section \ref{Sec:limittheorems}.
\subsubsection{Proof of Theorem \ref{Thm:KLdiv-limdist}}
We will invoke Proposition \ref{Prop:Hadamarddiff-gen} to prove the claim. To that end, we specialize the Hadamard differentiability framework of Section \ref{Sec:Techframework} by identifying the relevant quantities and showing that the required assumptions hold.
For brevity, we will only prove  the two-sample case and highlight its differences with the one-sample case at the end.

\medskip
\noindent\textbf{Part $\bm{(iii)}$:} 
Let $\rho=\mu$, $g_1^{\star}=g_2^{\star}=1$, $\psi_{2,0}\circ(1,1)=\psi_{1,1}\circ(1,1)=1$, $\psi_{0,2}\circ(1,1)=1+q$,  $v_{2,0}=v_{0,2}=v_{1,1}=1$, $p_{\eta_1}=3$ and $p_{\eta_2}=3+q$. 
Note that $\kl{\mu_n}{\nu_n}=\Phi(p_{\mu_n}-1,p_{\nu_n}-1)$  with $\phi(x,y)=\phi_{\mathsf{KL}}(x,y)=x \log (x/y)$ in \eqref{eq:defPhifunc}. Also, observe that under the assumptions of Part $(i)$ of Theorem \ref{Thm:KLdiv-limdist}, we have 
$ (p_{\mu_n}-1,p_{\nu_n}-1) \in \Theta $ a.s., where 
\begin{equation}\label{eq:thetanormspacekl3}
\begin{split} 
\Theta &= \left\{ \begin{aligned}
    (g_1-1,g_2-1) \in \mathfrak{D}:&g_1\geq 0, g_2 \geq 0, \supp(g_1) \subseteq \supp(g_2),   \norm{g_1}_{1,\mu}=\norm{g_2}_{1,\mu}=1,\\ &\abs{g_1/g_2} \leq q,~ \mu  \mbox{-a.s.}
\end{aligned} \right\},\\
\mathfrak{D}&=\left\{(g_1,g_2):\, g_1,g_2 \in  L^1(\mu), \norm{(g_1-1,g_2-1)}_{\mathfrak{D}}<\infty \right\}.
\end{split}
 \end{equation}
 Note that $\Theta \subseteq \bar\Theta(q)$ as defined in \eqref{eq:thetbarqset}. 
Then, Part $(ii)(a)$ of Lemma \ref{Lem:assumpfdiv} implies that
Assumption \ref{Assump1} and  \ref{Assump4} are satisfied, while Assumption \ref{Assump3} holds by hypothesis. Assumption \ref{Assump2} holds since from \eqref{eq:KL2dervs}, we obtain
$  D^{(1,0)}\phi_{\mathsf{KL}}\circ (1,1)=1$ and $ D^{(0,1)}\phi_{\mathsf{KL}}\circ (1,1)=-1$.  Proposition \ref{Prop:Hadamarddiff-gen}$(i)$ and \eqref{eq:KL2dervs}  then yield
\[
 \Phi_{\theta^{\star}}'(h_1,h_2)=\int_{\mathfrak{S}} h_{1}d \mu -\int_{\mathfrak{S}} h_{2}d \mu,\qquad \mbox{and}\qquad\Phi_{\theta^{\star}}''(h_1,h_2)=\int_{\mathfrak{S}} \big(h_{1}-h_2\big)^2 d \mu, \notag
\]
for all $(h_1,h_2) \in\mathfrak{T}_{\Theta }(\theta^{\star})=\mathrm{cl}\big(\big\{\big((g_1-1)/t,(g_2-1)/t\big):(g_1-1,g_2-1) \in \Theta, ~t>0\big\}\big)$ (note that $\Theta$ is convex).

Next, note that $\Phi(0,0)=\Phi_{\theta^{\star}}'(h_1,h_2)=0$ for all $(h_1,h_2) \in \mathfrak{T}_{\Theta}(\theta^{\star})$. The latter follows from that fact that if $(h_1,h_2) \in \mathfrak{T}_{\Theta}(\theta^{\star})$, then $h_1,h_2 \in L^1(\mu)$ and there exists a sequence $(h_{1,n},h_{2,n})_{n \in \NN}$, where  $h_{1,n}=(g_{1,n}-1)/t_n$, $h_{2,n}=(g_{2,n}-1)/t_n$, $(g_{1,n}-1,g_{2,n}-1) \in \Theta$, such that $\norm{h_{1,n}-h_1}_{1,\mu} \vee \norm{h_{2,n}-h_2}_{1,\mu} \rightarrow 0$. As $\int_{\mathfrak{S}}h_{1,n} d\mu =0$ due to $\norm{g_{1,n}}_{1,\mu}=1$ for every $n$, we further have
\begin{align}
    \abs{\int_{\mathfrak{S}} h_1 d\mu}=\abs{\lim_{n \rightarrow \infty}\int_{\mathfrak{S}} h_1 d\mu-\int_{\mathfrak{S}} h_{1,n} d\mu } \leq \lim_{n \rightarrow \infty} \int_{\mathfrak{S}} \abs{h_1- h_{1,n}  }d\mu =0. \label{eq:integzero}
\end{align}
Hence, $\int_{\mathfrak{S}} h_1 d\mu=0$ and similarly, $\int_{\mathfrak{S}} h_2 d\mu=0$.

To conclude, we observe that  $\big(r_n(p_{\mu_n}-1),r_n(p_{\nu_n}-1)\big) \trightarrow{w} (B_1,B_2)$ in~$\mathfrak{D}$ (w.r.t. norm of $L^2(\eta_1) \times L^2(\eta_2)$) and that $\big(r_n(p_{\mu_n}-1),r_n(p_{\nu_n}-1)\big)$ and $(B_1,B_2)$ take values in $\mathfrak{T}_{\Theta}(\theta^{\star})$. An application of Lemma \ref{Lem:extfuncdelta}$(ii)$ then yields \eqref{eq:KL-twosample-null}.

\medskip

\noindent\textbf{Part $\bm{(iv)}$:} 
Taking $\rho=\nu$,  first consider the case when $p_{\mu}>0$. Let $g_1^{\star}=p_{\mu}$, $g_2^{\star}=1$, $\psi_{2,0}\circ\big(p_{\mu},1\big)=1/p_{\mu}$, $\psi_{1,1}\circ\big(p_{\mu},1\big)=1$, $\psi_{0,2}\circ\big(p_{\mu},1\big)=p_{\mu}+q$,  $v_{2,0}=v_{0,2}=v_{1,1}=1$, $p_{\eta_1}=2+(1/p_{\mu})$ and $p_{\eta_2}=2+p_{\mu}+q$. We have $\kl{\mu_n}{\nu_n}=\Phi(p_{\mu_n}-p_{\mu},p_{\nu_n}-1)$  with $\phi(x,y)=\phi_{\mathsf{KL}}(x,y)=x \log (x/y)$ in \eqref{eq:defPhifunc}. 
Under the hypothesis in Part $(iv)$,  
$(p_{\mu_n}-p_{\mu},p_{\nu_n}-1) \in \Theta  $ a.s., where 
\begin{equation}\label{eq:thetanormspacekl4}
\begin{split} 
\Theta&= \left\{\begin{aligned}
    (g_1-p_{\mu},g_2-1) \in \mathfrak{D}&:g_1\geq 0, g_2 \geq 0,  \supp(g_1) \subseteq \supp(g_2), \norm{g_1}_{1,\nu}=\norm{g_2}_{1,\nu}=1,\\& \abs{g_1/g_2} \leq q,~ \nu \mbox{-a.s.}
\end{aligned} \right\},\\
\mathfrak{D}&=\left\{(g_1,g_2):\, g_1,g_2 \in  L^1(\nu), \norm{(g_1,g_2)}_{\mathfrak{D}}<\infty\right\}.
\end{split}
 \end{equation}
Observe that $\Theta \subseteq \bar\Theta(q)$ as given in \eqref{eq:thetbarqset}, and that Assumptions \ref{Assump1}, \ref{Assump3}, and \ref{Assump4} are satisfied (via Part $(i)$ and $(ii)(a)$ of Lemma \ref{Lem:assumpfdiv}).  Assumption \ref{Assump2} also holds since  
$  D^{(1,0)}\phi_{\mathsf{KL}}\circ (p_{\mu},1)=1+\log p_{\mu}\in L^2(\nu)$ and $ D^{(0,1)}\phi_{\mathsf{KL}}\circ (p_{\mu},1)=- p_{\mu} \in L^2(\nu)$ (see \eqref{eq:KL2dervs}) by hypothesis. It then follows from Proposition \ref{Prop:Hadamarddiff-gen}$(i)$ and \eqref{eq:KL2dervs}  that
\begin{align}
 \Phi_{\theta^{\star}}'(h_1,h_2)&=\int_{\mathfrak{S}} \big(1+\log  p_{\mu}\big)h_{1}d \nu-\int_{\mathfrak{S}}h_{2}p_{\mu}d \nu=\int_{\mathfrak{S}} h_{1} \log  p_{\mu} \,d \nu-\int_{\mathfrak{S}}h_{2}d \mu, \label{eq:hadderklgen}
\end{align}
for all $(h_1,h_2)\mspace{-2.5mu}\in\mspace{-2.5mu}\mathfrak{T}_{\Theta}(\theta^{\star})\mspace{-2mu}=\mspace{-2mu}\mathrm{cl}\big(\mspace{-2mu}\big\{\mspace{-2mu}\big((g_1\mspace{-2mu}-\mspace{-2mu}p_{\mu})/t,(g_2\mspace{-2mu}-\mspace{-2mu}1)/t\big)\mspace{-2mu}:\,(g_1\mspace{-2mu}-\mspace{-2mu}p_{\mu},g_2\mspace{-2mu}-\mspace{-2mu}1)\mspace{-3mu} \in\mspace{-3mu}\Theta,\,t\mspace{-3mu}>\mspace{-1mu}0\mspace{-1mu}\big\}\mspace{-2mu}\big)$. Noting that $\Phi(0,0)=\kl{\mu}{\nu}$, we obtain \eqref{eq:KL-twosample-alt} from Lemma \ref{Lem:extfuncdelta}$(i)$ and the fact that $\big(r_n(p_{\mu_n}-p_{\mu}),r_n(p_{\nu_n}-1)\big) \trightarrow{w} (B_1,B_2)$ in $\mathfrak{D}$ taking values in $\mathfrak{T}_{\Theta}(\theta^{\star})$. 

Finally, consider the case $p_{\mu} \ngtr 0$. To handle this scenario, we restrict the probability measures to the support of $p_{\mu}$,  i.e., $\mathfrak{S}= \supp(p_{\mu})$,
and apply Proposition \ref{Prop:Hadamarddiff-gen}$(ii)$. Set $\rho= \tilde \nu:=\nu_{|_{\supp(p_{\mu})}}$, where $\nu_{|_{\supp(p_{\mu})}}$ denotes the aforementioned restriction of $\nu$. Then, observe that $\tilde \mu:=\mu_{|_{\supp(p_{\mu})}}$ and $\tilde \mu_n:=\mu_{n|_{\supp(p_{\mu})}}$ are probability measures since $\mu_n \ll \mu$, while $\tilde \nu$ and $\tilde \nu_n:=\nu_{n|_{\supp(p_{\mu})}}$ are possibly deficient probability measures, i.e., $0<\tilde \nu(\supp(p_{\mu})) \leq 1$ and $0 \leq \tilde \nu_n(\supp(p_{\mu})) \leq 1$. Defining
\begin{align}
    \Theta&= \left\{\begin{aligned}
    (g_1-p_{\mu},g_2-1) \in \mathfrak{D}&:g_1\geq 0, g_2 \geq 0,  \supp(g_1) \subseteq \supp(g_2), \norm{g_1}_{1,\tilde \nu}=1, \norm{g_2}_{1,\tilde \nu} \leq 1,\\& \abs{g_1/g_2} \leq q,~ \nu \mbox{-a.s.}
\end{aligned} \right\}, \notag
\end{align}
we have  $(p_{\tilde \mu_n}-p_{\tilde \mu},p_{\tilde \nu_n}-1) \in \Theta  $ a.s. It now follows that $\big(r_n( p_{\tilde \mu_n}- p_{\tilde \mu}), r_n( p_{\tilde \nu_n}- 1)\big) \trightarrow{w} (\tilde B_1,\tilde B_2)$, where $(\tilde B_1,\tilde B_2)$ is the restriction of $(B_1,B_2)$ to $\supp(p_{\mu})$. Having that, the same argument as above with $p_{\mu}$, $p_{\mu_n}$, $p_{\nu_n}$, $ \nu$ replaced by $p_{\tilde \mu}$, $p_{\tilde \mu_n}$, $p_{\tilde \nu_n}$, $\tilde \nu$, respectively, yields
 \begin{align}
       r_n \big(\kl{\tilde \mu_n}{\tilde \nu_n}-\kl{\tilde \mu}{\tilde \nu}\big) \trightarrow{d} \int_{\supp(p_{\mu})} \tilde B_1 \log p_{\mu}   d \nu-\int_{\supp(p_{\mu})} \tilde B_2    d \mu=\int_{\supp(p_{\mu})}  B_1 \log p_{\mu}   d \nu-\int_{\supp(p_{\mu})}  B_2    d \mu.  \notag
\end{align}
The claim then follows by noting that $\kl{\tilde \mu_n}{\tilde \nu_n}=\kl{ \mu_n}{ \nu_n}$ and $\kl{\tilde \mu}{\tilde \nu}=\kl{\mu}{\nu}$.

The proof for the one-sample null  follows via similar arguments to the two-sample null  by considering $\rho=\mu$, $g_1^{\star}=g_2^{\star}=1$, $\psi_{2,0}\circ\big(g_1^{\star},g_2^{\star}\big)=1$, $\psi_{0,2}=\psi_{1,1}=0$, $v_{2,0}=v_{0,2}=v_{1,1}=1$, $\eta_1=2 \mu$, $\eta_2=\mu$, and
\begin{align}
     \mathfrak{D}&=\left\{(g_1-1,g_2-1):\, g_1,g_2 \in  L^1(\mu), \norm{(g_1-1,g_2-1)}_{\mathfrak{D}}<\infty \right\}, \notag \\
     \Theta &=\left\{(g_1-1,0) \in \mathfrak{D}: g_1 \geq 0,  \norm{g_1}_{1,\mu} = 1\right\}. \notag
\end{align}
Likewise, the proof for the one-sample alternative is obtained via analogous steps to the two-sample alternative by taking  $\mu \ll \nu=\rho$, $g_1^{\star}=p_{\mu}$, $g_2^{\star}=1$, $\psi_{2,0}\circ\big(g_1^{\star},g_2^{\star}\big)=1/p_{\mu}$, $\psi_{0,2}=\psi_{1,1}=0$, $v_{2,0}=v_{0,2}=v_{1,1}=1$, $p_{\eta_1}=1+(1/p_{\mu})$, $p_{\eta_2}=1$, and 
\begin{align}
     \mathfrak{D}&=\left\{(g_1-p_{\mu},g_2-1):\, g_1,g_2 \in  L^1(\nu), \norm{(g_1-p_{\mu},g_2-1)}_{\mathfrak{D}}<\infty \right\}, \notag \\
     \Theta &=\left\{(g_1-p_{\mu},0) \in \mathfrak{D}: g_1 \geq 0,  \norm{g_1}_{1,\nu}=1\right\}.\notag
\end{align}
 This completes the proof of the theorem.

\subsubsection{Proof of Theorem \ref{Thm:chisqdiv-limdist}} 
For proving Part $(i)$ and $(ii)$, the continuous mapping theorem suffices, while for Part $(iii)$ and $(iv)$, we will use the functional delta method. 

\noindent\textbf{Part $\bm{(i)}$:} The claim follows from the continuous mapping theorem \cite[Theorem 1.3.6]{AVDV-book}, by noting that $r_n^2 \chisq{\mu_n}{\mu}$ $=\int_{\mathfrak{S}} \big(r_n(p_{\mu_n}-1)\big)^2 d \mu$, $r_n(p_{\mu_n}-1) \trightarrow{w} B$ in $L^2(\mu)$, and  $f\mapsto\norm{f}_{2,\mu}^2$ is a continuous functional in $L^2(\mu)$. Also, observe that $r_n(p_{\mu_n}-1) \in L^2(\mu)$ implies that $\chisq{\mu_n}{\mu}=\norm{p_{\mu_n}-1}_{2,\mu}^2<\infty$.

\medskip

\noindent\textbf{Part $\bm{(ii)}$:}  
Note that
\begin{align}
  \chisq{\mu_n}{\nu}  =\norm{p_{\mu_n}-1}_{2,\nu}^2=\norm{p_{\mu_n}-p_{\mu}}_{2,\nu}^2+\norm{p_{\mu}-1}_{2,\nu}^2+2 \int_{\mathfrak{S}}(p_{\mu_n}-p_{\mu})(p_{\mu}-1) d\nu. \notag 
\end{align}
Hence, 
\begin{align}
  r_n\big(\chisq{\mu_n}{\nu} -\chisq{\mu}{\nu}\big) &=\int_{\mathfrak{S}} \Big(r_n^{\frac 12}\big(p_{\mu_n}-p_{\mu}\big)\Big)^2 d\nu+2 \int_{\mathfrak{S}}r_n(p_{\mu_n}-p_{\mu})(p_{\mu}-1) d\nu \notag \\
  &=\int_{\mathfrak{S}} \Big(r_n^{\frac 12}\big(p_{\mu_n}-p_{\mu}\big)\Big)^2 d\nu+2 \int_{\mathfrak{S}}r_n(p_{\mu_n}-p_{\mu})p_{\mu} d\nu. \notag 
\end{align}
Since $r_n(p_{\mu_n}-p_{\mu}) \trightarrow{w} B$ in $L^2(\nu)$, Slutsky's theorem implies that $r_n^{1/2}(p_{\mu_n}-p_{\mu}) \trightarrow{w} 0$ in $L^2(\nu)$. Consequently, first term in the RHS above converges weakly to zero in $L^2(\nu)$ by the continuous mapping theorem applied to the continuous functional $\norm{\cdot}_{2,\nu}^2$. Next, note that $\chisq{\mu}{\nu}<\infty$ and $\chisq{\mu_n}{\nu}<\infty$ imply  $p_{\mu} \in L^2(\nu) $ and $p_{\mu_n} \in L^2(\nu) $, respectively.  Lemma \ref{Lem:contmapLp}$(ii)$ together with the fact that $r_n(p_{\mu_n}-p_{\mu}) \trightarrow{w} B$ then imply that $r_n(p_{\mu_n}-p_{\mu})p_{\mu} \trightarrow{w} B p_{\mu}$ in $L^1(\nu)$. Having that, the claim follows from the continuous mapping theorem since  $f\mapsto \int_{\mathfrak{S}} f\,d \nu$ is a continuous functional in $L^1(\nu)$. 

\medskip

\noindent\textbf{Part $\bm{(iii)}$:}
Let $\rho=\mu$, $g_1^{\star}=g_2^{\star}=1$, $\psi_{2,0}\circ(1,1)=1$, $\psi_{1,1}\circ(1,1)=1+q$, $\psi_{0,2}\circ(1,1)=1+q^2$,  $v_{2,0}=v_{0,2}=v_{1,1}=1$, $p_{\eta_1}=3+q$ and $p_{\eta_2}=3+q+q^2$. We have $\chisq{\mu_n}{\nu_n}=\Phi(p_{\mu_n}-1,p_{\nu_n}-1)$  with $\phi(x,y)=\phi_{\chi^2}(x,y)=(x-y)^2/y$ in \eqref{eq:defPhifunc}.  Also, under the hypothesis in Part $(iii)$, we have $p_{\nu_n}>0$ and $(p_{\mu_n}-1,p_{\nu_n}-1) \in \Theta= \bar\Theta(q) $ a.s., where $\bar\Theta(q)$ and the ambient space $\mathfrak{D}$ are given in \eqref{eq:thetanormspacekl3}. To prove  \eqref{eq:chisq-twosample-null}, we next verify that Assumptions \ref{Assump1}-\ref{Assump4} hold in this setting and then invoke Proposition \ref{Prop:Hadamarddiff-gen}.

Observe that Part $(ii)(b)$ of Lemma \ref{Lem:assumpfdiv} implies that
Assumption \ref{Assump1} and Assumption \ref{Assump4} are satisfied. Furthermore, Assumption \ref{Assump3} holds by hypothesis, while Assumption \ref{Assump2} is satisfied since 
$ \abs{D^{(1,0)}\phi_{\chi^2}\circ (1,1)}=\abs{D^{(0,1)}\phi_{\chi^2}\circ (1,1)}$$=0$ from \eqref{eq:chisq2dervs}. Then, it follows via Proposition \ref{Prop:Hadamarddiff-gen}$(i)$ and \eqref{eq:chisq2dervs}   that 
\begin{align}
 \Phi_{\theta^{\star}}'(h_1,h_2)&=0 \qquad \mbox{and}\qquad \Phi_{\theta^{\star}}''(h_1,h_2)=\int_{\mathfrak{S}}2 \big(h_{1}-h_2\big)^2 d \mu, \notag
\end{align}
for all $(h_1,h_2) \in \mathfrak{T}_{\Theta}(\theta^{\star})=\mathrm{cl}\big(\big\{\big((g_1-1)/t,(g_2-1)/t\big):(g_1-1,g_2-1) \in \Theta, ~t>0\big\}\big)$.~Lastly, since $\Phi(0,0)=0$,  $\big(r_n(p_{\mu_n}-1),r_n(p_{\nu_n}-1)\big) \trightarrow{w} (B_1,B_2)$ in $\mathfrak{D}$, and $\big(r_n(p_{\mu_n}-1),r_n(p_{\nu_n}-1)\big)$ as well as $(B_1,B_2)$ take values in $\mathfrak{T}_{\Theta}(\theta^{\star})$, the convergence in  \eqref{eq:chisq-twosample-null} follows from Lemma \ref{Lem:extfuncdelta}$(ii)$. 

\medskip

\noindent\textbf{Part $\bm{(iv)}$:}
Let $\rho=\nu$ and assume first that $p_{\mu}>0$. Set $g_1^{\star}=p_{\mu}$, $g_2^{\star}=1$, $\psi_{2,0}\circ\big(p_{\mu},1\big)=1$, $\psi_{0,2}\circ\big(p_{\mu},1\big)=p_{\mu}^2+q^2$, $\psi_{1,1}\circ\big(p_{\mu},1\big)=p_{\mu}+q$,  $v_{2,0}=v_{0,2}=v_{1,1}=1$, $p_{\eta_1}=2+p_{\mu}+q$ and $p_{\eta_2}=1+p_{\mu}+q+p_{\mu}^2+q^2$. We have $\chisq{\mu_n}{\nu_n}=\Phi(p_{\mu_n}-p_{\mu},p_{\nu_n}-1)$  with $\phi(x,y)=\phi_{\chi^2}(x,y)=(x-y)^2/y$ in \eqref{eq:defPhifunc}.  Under the hypothesis in Part $(iii)$, 
$(p_{\mu_n}-p_{\mu},p_{\nu_n}-1) \in \Theta= \bar\Theta(q) $ a.s., where $\bar\Theta(q)$ and $\mathfrak{D}$ are as in \eqref{eq:thetanormspacekl4}.   Assumptions \ref{Assump1}, \ref{Assump3}, and \ref{Assump4} are satisfied for the same reason as in Part $(iii)$. Assumption \ref{Assump2} holds since  \eqref{eq:chisq2dervs} together with the hypothesis $p_{\mu}  \in L^4(\nu)$ implies $D^{(1,0)}\phi_{\chi^2}\circ(p_{\mu},1)=2 (p_{\mu}-1) \in L^2(\nu)$ and $ D^{(0,1)}\phi_{\chi^2}\circ(p_{\mu},1)=1-p_{\mu}^2 \in L^2(\nu)$. Applying Proposition \ref{Prop:Hadamarddiff-gen}$(i)$ while using \eqref{eq:chisq2dervs}, we have
\begin{align}
 \Phi_{\theta^{\star}}'(h_1,h_2)&=\int_{\mathfrak{S}} 2\big(p_{\mu}-1\big)h_{1}d \nu+\int_{\mathfrak{S}}\big(1-p_{\mu}^2\big)h_{2}d \nu=2\int_{\mathfrak{S}} h_{1}  \,d \mu-\int_{\mathfrak{S}}h_{2} p_{\mu} d \mu, \notag
\end{align}
for all $(h_1,h_2) \in \mathfrak{T}_{\Theta}(\theta^{\star})=\mathrm{cl}\big(\big\{\big((g_1-p_{\mu})/t,(g_2-1)/t\big):(g_1-p_{\mu},g_2-1) \in \Theta, ~t>0\big\}\big)$. Here, the last equality follows since $\int_{\mathfrak{S}} h_1 d\nu=\int_{\mathfrak{S}} h_2 d\nu=0$ for
$(h_1,h_2) \in  \mathfrak{T}_{\Theta}(\theta^{\star})$, similarly to \eqref{eq:integzero}. Finally, \eqref{eq:chisq-twosample-alt} follows from the above equation and Lemma \ref{Lem:extfuncdelta}$(i)$ by  noting that $\Phi(0,0)=\chisq{\mu}{\nu}$ and  $\big(r_n(p_{\mu_n}-p_{\mu}),r_n(p_{\nu_n}-1)\big) \trightarrow{w} (B_1,B_2)$ in $\mathfrak{D}$, taking values in $\mathfrak{T}_{\Theta}(\theta^{\star})$.
Finally,  the case $p_{\mu} \ngtr 0$ is handled  by restricting the space to $\supp(p_{\mu})$ as given at the end of the proof of Theorem \ref{Thm:KLdiv-limdist}, Part $(iv)$. The claim then follows by noting that  
$\chisq{\tilde \mu_n}{\tilde \nu_n}=\chisq{\mu_n}{\nu_n}$ and  $\chisq{\tilde \mu}{\tilde \nu}=\chisq{\mu}{\nu}$, where $\tilde \mu_n$, $\tilde \nu_n$, $\tilde \mu$ and $\tilde \nu$ are as defined  therein.

\subsubsection{Proof of Theorem \ref{Thm:Helsqdiv-limdist}} \label{Thm:Helsqdiv-limdist-proof}
The proof uses  Proposition \ref{Prop:Hadamarddiff-gen} by identifying the relevant quantities in the Hadamard differentiability framework of Section \ref{Sec:Techframework} and showing that the pertinent assumptions hold. Again, we only prove the two-sample case, delineating the difference for the one-sample case at the end.

\medskip
\noindent\textbf{Part $\bm{(iii)}$:}
Let  $\mu_n,\nu_n, \mu\ll \rho$ for some finite measure $\rho$, and set $g_1^{\star}=g_2^{\star}=p_{\mu}$,  $\psi_{2,0}\circ\big(p_{\mu},p_{\mu}\big)=\big(p_{\mu}^{1/2}+q_2^{1/2}\big)/p_{\mu}^{3/2}$, $\psi_{0,2}\circ\big(p_{\mu},p_{\mu}\big)=\big(p_{\mu}^{1/2}+q_1^{1/2}\big)/p_{\mu}^{3/2}$, $\psi_{1,1}\circ\big(p_{\mu},p_{\mu}\big)=1/p_{\mu}$, $v_{2,0}=v_{0,2}=\tau^{1/2}(1-\tau)^{-1/2}$,  $v_{1,1}=1$, $p_{\eta_1}=p_{\eta_2}=1+(1/p_{\mu})+\big(q^{1/2}/p_{\mu}^{3/2}\big)$, and 
\begin{align}
  \mathfrak{D}&=\left\{(g_1-p_{\mu},g_2-p_{\mu}):\, g_1,g_2 \in  L^1(\rho), \norm{(g_1-p_{\mu},g_2-p_{\mu})}_{\mathfrak{D}}<\infty \right\}, \notag \\
    \Theta &= \left\{(g_1-p_{\mu},g_2-p_{\mu}) \in \mathfrak{D}: g_1,g_2 \geq 0,  \norm{g_1}_{1,\rho}=\norm{g_2}_{1,\rho}=1,|g_1|\vee |g_2| \leq q,~ \rho  \mbox{-a.e.}\right\}. \notag
\end{align}
The hypothesis in Part $(iii)$ entails $ (p_{\mu_n}-p_{\mu},p_{\nu_n}-p_{\mu}) \in  \Theta $ a.s. Further, $\helsq{\mu_n}{\nu_n}=\Phi(p_{\mu_n}-p_{\mu},p_{\nu_n}-p_{\mu})$ with $\phi(x,y) = \phi_{\mathsf{H}^2}(x,y)=(\sqrt{x}-\sqrt{y})^2$ in \eqref{eq:defPhifunc}. Consequently, Parts $(i)$ and $(ii)(c)$ of Lemma \ref{Lem:assumpfdiv} shows that Assumptions \ref{Assump1} and \ref{Assump4} are satisfied since $\Theta \subseteq \check \Theta(q_1,q_2)$ and $ \phi_{\mathsf{H}^2}$ is continuous at $(0,0)$. Assumption \ref{Assump3} holds by boundedness of $\mathsf{H}^2$ distance. Assumption \ref{Assump2} follows since 
$D^{(1,0)}\phi_{\mathsf{H}^2}\circ(p_{\mu},p_{\mu})= D^{(0,1)}\phi_{\mathsf{H}^2}\circ (p_{\mu},p_{\mu})=0 $ from \eqref{eq:helsq2dervs}.  Proposition \ref{Prop:Hadamarddiff-gen}$(ii)$ and \eqref{eq:helsq2dervs}  now yield 
\begin{align} \Phi_{\theta^{\star}}'(h_1,h_2)&=0 \qquad  \mbox{and} \qquad \Phi_{\theta^{\star}}''(h_1,h_2)=\int_{\mathfrak{S}}\frac{ \big(h_{1}-h_2\big)^2}{2 p_{\mu}} d \rho, \notag
\end{align}
for all $(h_1,h_2) \in \mathfrak{T}_{\Theta}(\theta^{\star})=\mathrm{cl}\big(\big\{\big((g_1-p_{\mu})/t,(g_2-p_{\mu})/t\big):(g_1-p_{\mu},g_2-p_{\mu}) \in \Theta, ~t>0\big\}\big)$.\\ The convergence in \eqref{eq:Helsq-twosample-null} is then a consequence of Lemma \ref{Lem:extfuncdelta}$(ii)$, together with $\Phi(0,0)=0$, $\big(r_n(p_{\mu_n}-p_{\mu}),r_n(p_{\nu_n}-p_{\mu})\big) \trightarrow{w} (B_1,B_2)$ in $\mathfrak{D}$, where $\big(r_n(p_{\mu_n}-p_{\mu}),r_n(p_{\nu_n}-p_{\mu})\big)$ and $(B_1,B_2)$ take values in $\mathfrak{T}_{\Theta}(\theta^{\star})$.

\medskip
\noindent\textbf{Part $\bm{(iv)}$:} Let  $\mu_n, \nu_n, \mu,\nu \ll \rho$ for some finite measure $\rho$, and set $g_1^{\star}=p_{\mu}$, $g_2^{\star}=p_{\nu}$, $\psi_{2,0}\circ\big(p_{\mu},p_{\nu}\big)=\big(p_{\nu}^{1/2}+q_2^{1/2}\big)/p_{\mu}^{3/2}$, $\psi_{0,2}\circ\big(p_{\mu},p_{\nu}\big)=\big(p_{\mu}^{1/2}+q_1^{1/2}\big)/p_{\nu}^{3/2}$, $\psi_{1,1}\circ\big(p_{\mu},p_{\nu}\big)=1/\big(p_{\mu}p_{\nu}\big)^{1/2}$, $v_{2,0}=v_{0,2}=\tau^{1/2}(1-\tau)^{-1/2}$, and $v_{1,1}=1$. 
Further define the densities
\begin{align}
p_{\eta_1}&=1+\Big(p_{\nu}^{\frac 12}+q_2^{\frac 12}\Big)p_{\mu}^{-\frac 32}+p_{\mu}^{-\frac 12}p_{\nu}^{-\frac 12} 
\qquad  \mbox{and} \qquad p_{\eta_2}=1+\Big(p_{\mu}^{\frac 12}+q_1^{\frac 12}\Big)p_{\nu}^{-\frac 32}+p_{\mu}^{-\frac 12}p_{\nu}^{-\frac 12},\notag
\end{align} 
and consider the spaces
\begin{equation} \notag
\begin{split} 
 \Theta &=\check \Theta(q_1,q_2)= \left\{(g_1-p_{\mu},g_2-p_{\nu}) \in \mathfrak{D}:g_1,g_2 \geq 0,  \norm{g_1}_{1,\rho}=\norm{g_2}_{1,\rho}=1,|g_1| \leq q_1, |g_2| \leq q_2,~ \rho  \mbox{-a.e.} \right\},\\
\mathfrak{D}&=\left\{(g_1-p_{\mu},g_2-p_{\nu}):\, g_1,g_2 \in  L^1(\rho), \norm{(g_1-p_{\mu},g_2-p_{\nu})}_{\mathfrak{D}}<\infty  \right\}. 
\end{split}
 \end{equation}

Note that under the hypothesis in Part $(iv)$, 
$ (p_{\mu_n}-p_{\mu},p_{\nu_n}-p_{\nu}) \in  \Theta $ a.s., and  $\helsq{\mu_n}{\nu_n}=\Phi(p_{\mu_n}-p_{\mu},p_{\nu_n}-p_{\nu})$  with $\phi(x,y) = \phi_{\mathsf{H}^2}(x,y)=(\sqrt{x}-\sqrt{y})^2$ in \eqref{eq:defPhifunc}. Assumptions \ref{Assump1} and \ref{Assump4} as well as the continuity of $ \phi_{\mathsf{H}^2}$ at $(0,0)$ holds via Lemma \ref{Lem:assumpfdiv}, Parts $(i)$ and $(ii)(c)$, while  Assumption \ref{Assump3} follows by boundedness of $\mathsf{H}^2$. Assumption \ref{Assump2} is satisfied since from \eqref{eq:helsq2dervs}, we have both
$D^{(1,0)}\phi_{\mathsf{H}^2}\circ(p_{\mu},p_{\nu})=1-p_{\nu}^{1/2}p_{\mu}^{-1/2}$ and $ D^{(0,1)}\phi_{\mathsf{H}^2}\circ (p_{\mu},p_{\nu})=1-p_{\mu}^{ 1/2}p_{\nu}^{- 1/2}$ belonging to $L^2(\rho)$ because $p_{\nu}/p_{\mu},p_{\mu}/p_{\nu} \in L^1(\rho)$ and $\rho$ is a finite measure. Invoking  Proposition \ref{Prop:Hadamarddiff-gen}$(ii)$ and  \eqref{eq:helsq2dervs}, we obtain
\begin{align}
\Phi_{\theta^{\star}}'(h_1,h_2)&=\int_{\mathfrak{S}}\Big(1-p_{\nu}^{\frac 12}p_{\mu}^{-\frac 12}\Big) h_1 d\rho+\int_{\mathfrak{S}}\Big(1-p_{\mu}^{\frac 12}p_{\nu}^{-\frac 12}\Big) h_2 d\rho=-\int_{\mathfrak{S}}p_{\nu}^{\frac 12}p_{\mu}^{-\frac 12} h_1 d\rho-\int_{\mathfrak{S}}p_{\mu}^{\frac 12}p_{\nu}^{-\frac 12} h_2 d\rho, \notag
\end{align}
for all $(h_1,h_2) \in \mathfrak{T}_{\Theta}(\theta^{\star})=\mathrm{cl}\big(\big\{\big((g_1-p_{\mu})/t,(g_2-p_{\nu})/t\big):(g_1-p_{\mu},g_2-p_{\nu}) \in \Theta, ~t>0\big\}\big)$.\\ The desired result, namely \eqref{eq:Helsq-twosample-alt}, then follows from Lemma \ref{Lem:extfuncdelta}$(i)$, along with $\Phi(0,0)=\helsq{\mu}{\nu}$,  $\big(r_n(p_{\mu_n}-p_{\mu}),r_n(p_{\nu_n}-p_{\nu})\big) \trightarrow{w} (B_1,B_2)$ in $\mathfrak{D}$ with $\big(r_n(p_{\mu_n}-p_{\mu}),r_n(p_{\nu_n}-p_{\nu})\big)$ and $(B_1,B_2)$ taking values in $\mathfrak{T}_{\Theta}(\theta^{\star})$.

The proof for the one-sample null follows via analogous arguments to the two-sample null by taking $\mu_n, \mu \ll \rho$ for some finite measure $\rho$,    $g_1^{\star}=g_2^{\star}=p_{\mu}$, $\psi_{2,0}\circ\big(p_{\mu},p_{\mu}\big)=1/p_{\mu}$, $\psi_{0,2}=\psi_{1,1}=0$, $v_{2,0}=(1-\tau)^{-1/2}$, $v_{0,2}=v_{1,1}=1$, $p_{\eta_1}=1+(1/p_{\mu})$, $\eta_2=\rho$, and
\begin{align}
   &  \mathfrak{D}=\left\{(g_1-p_{\mu},g_2-p_{\mu}):\, g_1,g_2 \in  L^1(\rho), \norm{(g_1-p_{\mu},g_2-p_{\mu})}_{\mathfrak{D}}<\infty \right\}, \notag \\
  &\Theta =\left\{(g_1-p_{\mu},0) \in \mathfrak{D}: g_1 \geq 0,  \norm{g_1}_{1,\rho}=1\right\}. \notag
\end{align}
 In a similar vein, the proof  for the one-sample alternative  is   akin to the two-sample alternative, by considering $\mu_n, \mu,\nu \ll \rho$ for some finite measure $\rho$, and set $g_1^{\star}=p_{\mu}$, $g_2^{\star}=p_{\nu}$,  $\psi_{2,0}\circ\big(p_{\mu},p_{\nu}\big)=p_{\nu}^{1/2}/p_{\mu}^{3/2}$, $\psi_{0,2}=\psi_{1,1}=0$,  $v_{2,0}=(1-\tau)^{-1/2}$,  $v_{0,2}=v_{1,1}=1$, $p_{\eta_1}=1+\big(p_{\nu}^{1/2}/p_{\mu}^{3/2}\big)$, $\eta_2=\rho$, and
\begin{align}
     &\mathfrak{D}=\left\{(g_1-p_{\mu},g_2-p_{\nu}):\, g_1,g_2 \in  L^1(\rho), \norm{(g_1-p_{\mu},g_2-p_{\nu})}_{\mathfrak{D}}<\infty \right\}, \notag \\
  &\Theta= \left\{(g_1-p_{\mu},0) \in \mathfrak{D}: g_1 \geq 0,  \norm{g_1}_{1,\rho}=1\right\}. \notag
\end{align}
\subsubsection{Proof of Theorem \ref{Thm:TVdiv-limdist}}\label{Thm:TVdiv-limdist-proof}
We prove the one- and two-sample null cases together. Then, we prove both statements under the alternative.

\medskip
\noindent\textbf{(One-sample and two-sample null):} 
The results in these cases  follow directly by the continuous mapping theorem. To see this, note that $f\mapsto\norm{f}_{\mathsf{TV}}=\norm{f}_{1,\rho}/2$ is a continuous functional in $L^1(\rho)$. With that, the convergence of $r_n(p_{\mu_n}-p_{\mu}) \trightarrow{w} B$  and  $r_n(p_{\mu_n}-p_{\nu_n}) \trightarrow{w} B$ in $L^1(\rho)$ imply
  \begin{align}
    r_n \tv{\mu_n}{\mu} \trightarrow{d} \frac{1}{2} \int_{\mathfrak{S}} \abs{B} d \rho \qquad  \mbox{and} \qquad
    r_n \tv{\mu_n}{\nu_n} \trightarrow{d} \frac{1}{2} \int_{\mathfrak{S}} \abs{B} d \rho,  \notag 
\end{align}
respectively. Note that the above limit distributions are special cases of \eqref{eq:TV-onesample-alt} and \eqref{eq:TV-twosample-alt}~for $\mu=\nu$, since  $\cQ=\mathfrak{S}$ in this case and the second terms in the respective equations vanish.

\medskip
\noindent\textbf{(One-sample and two-sample alternative):}
Note that $\tv{\mu}{\nu}=\Phi(0,0)$ with $\phi(x,y)=\abs{x-y}/2$, $g_1^{\star}=p_{\mu}$ and $g_2^{\star}=p_{\nu}$ in \eqref{eq:defPhifunc}. Hence Assumption \ref{Assump1} is violated (at $x=y$), and the Hadamard differentiability framework in Section \ref{Sec:Techframework} does not apply directly. To circumvent that, we provide a direct argument for computing the Hadamard directional derivative in this case.

\medskip
Let $\mathfrak{D}=L^1(\rho)$, $g_1^{\star}, g_2^{\star} \in L^1(\rho)$ be arbitrary, set
\[
  \hat \Theta=\big\{g=g_1-g_2:g_1,g_2 \in L^1(\rho),~g_1,g_2 \geq 0, \norm{g_1}_{1,\rho}=\norm{g_2}_{1,\rho}=1\big\},
\]
and consider the functional $\Upsilon:\hat \Theta \rightarrow \RR_{\geq 0}$ given by $\Upsilon(g)=\int_{\mathfrak{S}}\abs{g}/2\, d \rho$. We will show that $\Upsilon$ is locally Lipschitz at $\theta^{\star}=g_1^{\star}-g_2^{\star}\in \hat \Theta$, whereby its Hadamard directional derivative at $\theta^{\star}$ coincides with the G\^{a}teaux directional derivative \cite{Shapiro-1990,Romisch-2004}, which is defined by  
\[
   \Upsilon'_{\theta^{\star}}(h)=\lim_{t_n \downarrow 0^+} \frac{\Upsilon(g_1^{\star}-g_2^{\star}+t_n h)-\Upsilon(g_1^{\star}-g_2^{\star})}{t_n},\quad h\in\mathfrak{D}.
\]
For the local Lipschitzness, observe that
\[
\Upsilon(g)-\Upsilon(\tilde g) = \frac 12 \int_{\mathfrak{S}} \abs{g} d \rho  -\frac 12 \int_{\mathfrak{S}} \abs{\tilde g} d \rho    \notag \leq \frac 12 \int_{\mathfrak{S}} \abs{g-\tilde g} d \rho =  \frac 12  \norm{g-\tilde g}_{1,\rho},
\]
and switch the roles of $g$ and $\tilde g$ to conclude that
$\abs{\Upsilon(g)-\Upsilon(\tilde g)} \leq \norm{g-\tilde g}_{1,\rho}/2$.

\medskip
To compute the G\^{a}teaux derivative of $\Upsilon$, set $\cQ^{\star}:=\{s \in \mathfrak{S}: g_1^{\star}(s)=g_1^{\star}(s)\}$ and write
\begin{flalign}
 \frac{\Upsilon(g_1^{\star}-g_2^{\star}+t h)-\Upsilon(g_1^{\star}-g_2^{\star})}{t}&=\frac 12 \int_{\cQ^{\star}}\abs{h} d\rho +\frac 12 \int_{\mathfrak{S}\setminus \cQ^{\star}}  \frac{\abs{g_1^{\star}-g_2^{\star}+t h}-\abs{g_1^{\star}-g_2^{\star}}}{t}  d\rho.\label{eq:gateauxder}
\end{flalign}
Note that the integrand in the last term above is dominated by $\abs{h}\in L^1(\rho)$ $\rho$-a.e. and that on $\mathfrak{S}\setminus \cQ^{\star}$, we have
\begin{align}
\lim_{t \downarrow 0^+} \frac{\abs{g_1^{\star}-g_2^{\star}+t h}-\abs{g_1^{\star}-g_2^{\star}}}{t}   =\mathrm{sgn}\big(g_1^{\star}-g_2^{\star}\big)h,\notag
\end{align}
where $\mathrm{sgn}(x)=x/|x|$ for $x \neq 0$ and $\mathrm{sgn}(0)=0$. An application of dominated convergence theorem to the last term in \eqref{eq:gateauxder} then yields the G\^{a}teaux derivative at $\theta^{\star}$ given by 
\begin{align}
  \Upsilon'_{\theta^{\star}}(h):= \lim_{t \downarrow 0^+} \frac{\Upsilon(g_1^{\star}-g_2^{\star}+t h)-\Upsilon(g_1^{\star}-g_2^{\star})}{t}&=\frac 12 \int_{\cQ^{\star}}\abs{h} d\rho +\frac 12 \int_{\mathfrak{S}\setminus \cQ^{\star}} \mathrm{sgn}\big(g_1^{\star}-g_2^{\star}\big)hd\rho,\label{eq:HadderTVdist}
\end{align}
which coincides with the Hadamard derivative due to the aforementioned local Lipschitzness of the functional.

\medskip
We are now in place to prove the limit distribution results. For the one-sample alternative case, let $\mu_n,\mu,\nu \ll \rho$ for some measure $\rho$, set $\theta^{\star}=p_{\mu}-p_{\nu}$, and consider 
\begin{align}
\Theta=\big\{g_1-p_{\nu}:g_1 \in L^1(\rho),~g_1 \geq 0, \norm{g_1}_{1,\rho}=1\big\} \subset \hat \Theta. \notag
\end{align}
Note that $\mathfrak{T}_{\Theta}(\theta^{\star})=\mathrm{cl}\big(\big\{\big(g_1-p_{\mu})/t: g_1-p_{\nu} \in \Theta, ~t>0\big\}\big)$ since $ \Theta$ is convex. Setting $\cQ:=\{s \in \mathfrak{S}: p_{\mu}(s)=p_{\nu}(s)\}$, the Hadamard directional derivative of $\Upsilon$ at $\theta^\star$ is
\begin{align}
\Upsilon'_{\theta^{\star}}(h)=\frac 12 \int_{\cQ}\abs{h} d\rho +\frac 12 \int_{\mathfrak{S}\setminus \cQ} \mathrm{sgn}\big(p_{\mu}-p_{\nu}\big)hd\rho,\quad h \in \mathfrak{T}_{\Theta}(\theta^{\star}).\notag
\end{align}
Lemma \ref{Lem:extfuncdelta}$(i)$ along with the fact that $r_n(p_{\mu_n}-p_{\mu}) \trightarrow{w} B$ in $L^1(\rho)$ with $r_n(p_{\mu_n}-p_{\mu})$ and $B$ taking values in $\mathfrak{T}_{\Theta}(\theta^{\star})$, leads to
\begin{align}
  r_n\big(\tv{p_{\mu_n}}{p_{\nu}}-\tv{p_{\mu}}{p_{\nu}}\big)=  r_n\big(\Upsilon(p_{\mu_n}-p_{\nu})-\Upsilon(p_{\mu}-p_{\nu})\big) \trightarrow{d} \Upsilon'_{\theta^{\star}}(B), \notag
\end{align}
which proves \eqref{eq:TV-onesample-alt}. 
\medskip

For two-sample case, let $\mu_n, \nu_n,\mu,\nu \ll \rho$ for some measure $\rho$, $\theta^{\star}=p_{\mu}-p_{\nu}$, and $\Theta=\hat \Theta$. Again, $\mathfrak{T}_{\Theta}(\theta^{\star})=\mathrm{cl}\big(\big\{\big(g-(p_{\mu}-p_{\nu})\big)/t:g \in \Theta, ~t>0\big\}\big)$ since $\Theta$ is convex. The claim in \eqref{eq:TV-twosample-alt} then follows using same arguments as above via Lemma \ref{Lem:extfuncdelta}$(i)$ together with  $r_n\big(p_{\mu_n}-p_{\nu_n}-(p_{\mu}-p_{\nu})\big) \trightarrow{w} B$ in $L^1(\rho)$ with $r_n\big(p_{\mu_n}-p_{\nu_n}-(p_{\mu}-p_{\nu})\big)$ and $B$ supported in $\mathfrak{T}_{\Theta}(\theta^{\star})$. This completes the proof.

\subsection{Proofs for Section \ref{Sec:GS limit theorems}}
To establish the limit distribution for Gaussian-smoothed $f$-divergences, we first use the CLT in $L^r$ spaces  to deduce weak convergence of the smooth empirical process, and then invoke the general limit distribution theorems from Section \ref{Sec:limittheorems}. The CLT in $L^r$ spaces is stated next.

\begin{theorem}[Proposition 2.1.11 in \cite{AVDV-book}] \label{Thm:CLT-in-Lp}
Let $1 \leq r <\infty$, and $Z,Z_1,\ldots,Z_n$ be i.i.d. $L^r(S,\cS,\rho)$-valued random variables (recall $\rho$ is $\sigma$-finite) with zero mean (in the sense of Bochner).  The following are equivalent:
\begin{enumerate}[(i)]
    \item There exists a centered Gaussian process $G$ in $L^r$ with  same covariance function as $Z$ such that $n^{-1/2}\sum_{i=1}^n Z_i$ converges weakly in $L^r$ to $G$.
    \item $\int_{\mathfrak{S}}\big(\EE\big[\abs{Z(s)}^2\big]\big)^{r/2} d \rho(s) <\infty$ and $\PP(\norm{Z}_{r,\rho} \geq t)=o(t^{-2})$ as $t \rightarrow \infty$.
\end{enumerate}
\end{theorem}

Henceforth, in this section, we  apply the general results from Section 4, namely, Theorems 1-4, with
$\mu \leftarrow \mu*\gamma_{\sigma}$, $\nu \leftarrow \nu*\gamma_{\sigma}$, $\mu_n=\hat {\mu}_n*\gamma_\sigma$ and $\nu_n=\hat {\nu}_n*\gamma_\sigma$.
The reference measure $\rho$  will be adapted on a case-by-case basis and specified in the proofs below.
\subsubsection{Proof of Proposition \ref{Prop:GS-KL-limdist} }

\noindent\textbf{Part $\bm{(i)}$:} To prove \eqref{eq:KL-GS-onesamp-null}, we apply Theorem \ref{Thm:KLdiv-limdist}(i) with $\rho=\mu*\gamma_{\sigma}$. To verify that the required conditions hold, first observe that $\mu*\gamma_{\sigma} \ll\gg \hat {\mu}_n*\gamma_\sigma$. Furthermore, for $\hat X \sim \hat \mu_n$  and  $X \sim \mu$ independent of $Y \sim \mu$, we have 
\begin{align}
    \EE \big[ \kl{\hat {\mu}_n*\gamma_{\sigma}}{\mu*\gamma_{\sigma}}\big] & =       \EE \Big[    \kl{\EE_{\hat \mu_n}\big[N\big(\hat X,\sigma^2I_d\big)\big]}{\EE_{ \mu}\big[N\big(Y,\sigma^2I_d\big)\big]}\Big]\notag \\
     &\stackrel{(a)}{\leq}\EE \Big[ \kl{N\big( X,\sigma^2I_d\big)}{N\big( Y,\sigma^2I_d\big)}\Big]\notag \\
     &\stackrel{(b)}{=} \frac{1}{2 \sigma^2}\EE \Big[\norm{X-Y}^2\Big]< \infty, \notag
 \end{align}
 where $(a)$ is by Jensen's inequality applied to KL divergence which is jointly convex in its arguments, $(b)$ uses the closed-form expression for KL divergence between multivariate Gaussian distributions (see  \cite{GIL-2013}), and the last inequality is because all moments of a sub-Gaussian distribution are finite (see\cite{jin-2019-subgaussnorm}). Since KL divergence is nonnegative, this implies that $\kl{\hat {\mu}_n*\gamma_{\sigma}}{\mu*\gamma_{\sigma}}<\infty$ a.s. Setting $r_n=n^{1/2}$, \eqref{eq:KL-GS-onesamp-null} follows from \eqref{eq:KL-onesample-null} provided that
\begin{align}
  &  n^{1/2} \left(\frac{\hat\mu_n*\varphi_{\sigma}}{\mu*\varphi_{\sigma}}-1\right) \trightarrow{w} \frac{G_{\mu,\sigma}}{\mu*\varphi_{\sigma}}\quad\mbox{ in }L^2(\mu*\gamma_\sigma). \label{eq:weakconv-emprat-null1}
    \end{align}

It thus remains to establish \eqref{eq:weakconv-emprat-null1}, for which we use Theorem \ref{Thm:CLT-in-Lp}. Note that $n^{1/2} \big((\hat \mu_n*\varphi_{\sigma}/ \mu*\varphi_{\sigma})-1\big)=n^{-1/2} \sum_{i=1}^n Z_i$, where
\begin{align}
 Z_i=\frac{\varphi(\cdot-X_i)-\mu*\varphi_{\sigma}}{\mu*\varphi_{\sigma}},\quad  i=1,\ldots,n, \label{eq:defZivars} 
\end{align}
are centered  i.i.d. random variables with the same distribution as $Z=\big(\varphi(\cdot-X)-\mu*\varphi_{\sigma}\big)/ \mu*\varphi_{\sigma}$. Further observe that the process $Z_i$ is jointly measurable when viewed as a map from $\big(\Omega \times \RR^d, \cA \times \cB(\RR^d)\big)$ to $\RR$,  and
has paths a.s. in $L^2(\mu*\gamma_\sigma)$; indeed,
\begin{align}
 \EE\big[\norm{Z_i}_{2,\mu*\gamma_\sigma}^2 \big]=\int_{\RR^d}\EE_{\mu}\left[\abs{Z_i(x)}^2\right] \mu*\varphi_{\sigma}(x) dx=\int_{\RR^d}\frac{\mathsf{Var}_{\mu}\big(\varphi_{\sigma}(x-\cdot)\big)}{\mu*\varphi_{\sigma}(x)}\, dx <\infty, \label{eq:pathsasL2}   
\end{align}
by Fubini's theorem and the assumption in \eqref{ASSUM:KL_null}. This finiteness also implies the conditions in Theorem \ref{Thm:CLT-in-Lp}(ii). The first follows directly from the equation above, while for the second, we show that if  $\PP(\norm{Z}_{2,\mu*\gamma_\sigma} \geq t)=o(t^{-2})$ as $t \rightarrow \infty$ does not hold, then $\EE\big[\norm{Z_i}_{2,\mu*\gamma_\sigma}^2 \big]<\infty$ is contradicted. To see this, note that the violation of the former condition implies that there exists constants $c>0$ and $t_0 \geq 1$ such that for all $t \geq t_0$, $\PP(\norm{Z}_{2,\mu*\gamma_\sigma} \geq t) \geq ct^{-2}$. Then
\begin{align}
  \EE\big[\norm{Z}_{2,\mu*\gamma_\sigma}^2\big]=\int_{0}^{\infty} \PP(\norm{Z}_{2,\mu*\gamma_\sigma} \geq \sqrt{t}) dt \geq \int_{t_0}^{\infty} ct^{-1} dt=\infty, \notag
\end{align}
which is the desired contradiction.
   Hence, the conditions in Theorem \ref{Thm:CLT-in-Lp}(ii) are satisfied under \eqref{ASSUM:KL_null}. Consequently, by Theorem \ref{Thm:CLT-in-Lp}$(i)$, there exists a centered Gaussian process $G_{\mu,\sigma}/\mu*\varphi_{\sigma}$ such that 
\[
n^{-1/2} \sum_{i=1}^n Z_i \trightarrow{w} \frac{G_{\mu,\sigma}}{\mu*\varphi_{\sigma}}\quad\mbox{ in }L^2(\mu*\gamma_\sigma),\] 
which completes the proof of \eqref{eq:KL-GS-onesamp-null}.

The claim that \eqref{ASSUM:KL_null} and \eqref{eq:KL-GS-onesamp-null} holds for $\beta$-sub-Gaussian $\mu$  with $\beta<\sigma$ follows from Proposition 1 in \cite{AZYA-2022} by noting that the LHS of \eqref{ASSUM:KL_null} equals $I_{\chi^2}(V;W)$, where $I_{\chi^2}(V;W):=\chisq{P_{V,W}}{P_V \times P_W}$ is the $\chi^2$ mutual information between $V$ and $W=V+Z$, where $V \sim \mu$ and $Z \sim \gamma_{\sigma}$ are independent. Finally, we observe that if \eqref{ASSUM:KL_null} is violated then $\liminf_{n \rightarrow \infty}n\EE\big[\kl{\hat\mu_n*\gamma_\sigma}{\mu*\gamma_\sigma}\big]= \infty$. This is a consequence of Lemma \ref{lem:fdivnullonesamplb}, by noting that $f_{\mathsf{KL}}$ is continuously twice differentiable  with  positive second derivative. 
\medskip

\noindent\textbf{Part $\bm{(ii)}$:} To prove \eqref{eq:KL-GS-onesamp-alt}, we use Theorem \ref{Thm:KLdiv-limdist}$(ii)$ 
with $\rho=\nu*\gamma_{\sigma}$.  Note  that $\mu*\gamma_{\sigma}$,  $\hat {\mu}_n*\gamma_\sigma$, $\nu*\gamma_{\sigma} $ are all mutually absolutely continuous, and $p_{\mu*\gamma_{\sigma}}=\mu*\varphi_{\sigma}/\nu*\varphi_{\sigma}>0$. Moreover,  $\log \big(\mu*\varphi_{\sigma}/\nu*\varphi_{\sigma}\big) \in L^2(\nu*\varphi_{\sigma})$ by assumption, and using similar steps as in Part $(i)$, we have $ \kl{\mu*\gamma_{\sigma}}{\nu*\gamma_{\sigma}}<\infty$ and $\kl{\hat \mu_n*\gamma_{\sigma}}{\nu*\gamma_{\sigma}}<\infty$ a.s. Hence, we have from \eqref{eq:KL-onesample-alt} with $r_n=n^{1/2}$  that
\begin{align}
   n^{\frac 12}\big(\kl{\hat \mu_n*\gamma_\sigma}{\nu*\gamma_{\sigma}}-\kl{\mu*\gamma_{\sigma}}{\nu*\gamma_{\sigma}}\big) \trightarrow{d}
   \int_{\RR^d} G_{\mu,\sigma}(x) \log \left(\frac{\mu*\varphi_{\sigma}(x)}{\nu*\varphi_{\sigma}(x)}\right) d x,
 \label{eq:onesampaltKLinterm} 
   \end{align}
provided 
\begin{align}
 n^{1/2} \left(\frac{\hat\mu_n*\varphi_{\sigma}}{\nu*\varphi_{\sigma}}-\frac{\mu*\varphi_{\sigma}}{\nu*\varphi_{\sigma}}\right) \trightarrow{w} \frac{G_{\mu,\sigma}}{\nu*\varphi_{\sigma}}\quad\mbox{ in }L^2(\eta), \label{eq:weakconvreqalt1}
 \end{align}
where $\eta$ has relative density $p_{\eta}=1+(\nu*\varphi_{\sigma}/\mu*\varphi_{\sigma})$. To prove  \eqref{eq:KL-GS-onesamp-alt}, it remains to show that \eqref{eq:weakconvreqalt1} holds and that the RHS of \eqref{eq:onesampaltKLinterm} is $N\big(0,v_{1,f_{\mathsf{KL}}}^2(\mu,\nu,\sigma)\big)$. 

\medskip

We first show that \eqref{eq:weakconvreqalt1} holds given \eqref{ASSUM:KL_null}, \eqref{ASSUM:KL_alt}, and $\norm{(\nu*\varphi_{\sigma})^2/\mu*\varphi_{\sigma}}_{\infty}<\infty$. The latter along with the fact that $\nu*\gamma_{\sigma}$ is a probability measure implies that $\eta$ is $\sigma$-finite. Next, observe that 
$n^{1/2} (\hat \mu_{n}*\varphi_{\sigma}-\mu*\varphi_{\sigma})/\nu*\varphi_{\sigma}=n^{-1/2} \sum_{i=1}^n Z_i$, where $Z_i=\big(\varphi(\cdot-X_i)-\mu*\varphi_{\sigma}\big)/ \nu*\varphi_{\sigma}$, for $i=1,\ldots, n$, are centered and i.i.d. with the same distribution as $Z=\big(\varphi(\cdot-X)-\mu*\varphi_{\sigma}\big)/ \nu*\varphi_{\sigma}$. The claim then follows from Theorem~\ref{Thm:CLT-in-Lp}$(ii)$, given that the conditions therein are satisfied. Akin to the proof of Part~$(i)$ above, the process $Z_i$ is jointly measurable with paths a.s. in $L^2(\eta)$. Indeed
\begin{align}
\EE\big[\norm{Z_i}_{2,{\eta}}^2 \big]&=\int_{\RR^d}\EE_{\mu} \left[\abs{Z_i(x)}^2\right] \left(\nu*\varphi_{\sigma}(x)+ \frac{\big(\nu*\varphi_{\sigma}(x)\big)^2}{\mu*\varphi_{\sigma}(x)}\right) dx\notag \\
&=\int_{\RR^d}\frac{\mathsf{Var}_{\mu}\big(\varphi_{\sigma}(x-\cdot)\big)}{\nu*\varphi_{\sigma}(x)}\, dx+\int_{\RR^d}\frac{\mathsf{Var}_{\mu}\big(\varphi_{\sigma}(x-\cdot)\big)}{\mu*\varphi_{\sigma}(x)}\, dx <\infty, \label{eq:bndsecmomexp}
\end{align}
where the first equality is by Fubini's theorem and last inequality is due to \eqref{ASSUM:KL_null} and \eqref{ASSUM:KL_alt}. This finiteness implies the conditions in Theorem \ref{Thm:CLT-in-Lp}(ii): the first follows trivially, while the second follows via similar arguments as in Part $(i)$. Invoking Theorem \ref{Thm:CLT-in-Lp}, there exists a centered Gaussian process $G_{\mu,\sigma}/\nu*\varphi_{\sigma}$ such that \eqref{eq:weakconvreqalt1} holds.

Next, we claim that the RHS of \eqref{eq:onesampaltKLinterm} is zero mean Gaussian random variable with variance $v_{1,f_{\mathsf{KL}}}^2(\mu,\nu,\sigma)$. This follows from the dual characterization of a random variable $B$ taking values in a Banach space $\mathfrak{B}$ as Gaussian if and only if $f(B)$ is a real-valued Gaussian random variable for every $f$ in the topological dual space of continuous linear functionals on $\mathfrak{B}$ (see \cite[Page 55]{LT-1991}). Hence,  the  RHS of \eqref{eq:onesampaltKLinterm} is a real Gaussian random variable because $G_{\mu,\sigma}/\sqrt{\nu*\varphi_{\sigma}}$ is $L^2(\RR^d)$-valued Gaussian random variable  and $\sqrt{\nu*\varphi_{\sigma}} \log \big(\mu*\varphi_{\sigma}/\nu*\varphi_{\sigma}\big) \in L^2(\RR^d)$ (by assumption).  
Computing the mean and variance explicitly leads to the claim above with  
\begin{align}
    v_{1,f_{\mathsf{KL}}}^2(\mu,\nu,\sigma)=\int_{\RR^d}\int_{\RR^d} \Sigma^{(1,1)}_{\mu,\nu,\sigma}(x,y) \log \left(\frac{\mu*\varphi_{\sigma}(x)}{\nu*\varphi_{\sigma}(x)}\right) \log \left(\frac{\mu*\varphi_{\sigma}(y)}{\nu*\varphi_{\sigma}(y)}\right) d x\, d y, \label{eq:var-smoothedKL1samp}
\end{align}
where we used   $\int_{\RR^d}\int_{\RR^d}\Sigma^{(1,1)}_{\mu,\nu,\sigma}(x,y)dx dy=\int_{\RR^d}\int_{\RR^d} \mathsf{cov}\big(\varphi_{\sigma}(x-X),\varphi_{\sigma}(y-X)\big)dx dy=0$, which in itself follows from  Fubini's theorem and $\int_{\RR^d}\varphi_{\sigma}(x-\cdot)dx=1$.
   \medskip
   
   To prove the final claim in Part $(ii)$, we note that for every $x \in \RR^d$, 
   \begin{align}
       \abs{\frac{\mu *\varphi_{\sigma}(x)}{\nu *\varphi_{\sigma}(x)}}=   \abs{\frac{\EE_{\mu}[ \varphi_{\sigma}(x-\cdot)]}{\EE_{\nu}[ \varphi_{\sigma}(x-\cdot)}} \leq \norm{\frac{d \mu}{d \nu}}_{\infty}<\infty, \label{eq:ratiodensbnd}
   \end{align}
which implies $\norm{\mu*\varphi_{\sigma}/\nu*\varphi_{\sigma}}_{\infty}<\infty$. Similarly,  $\norm{\nu*\varphi_{\sigma}/\mu*\varphi_{\sigma}}_{\infty}<\infty$. 
   This  leads to $\norm{(\nu*\varphi_{\sigma})^2/\mu*\varphi_{\sigma}}_{\infty}<\infty$ and  $\log \big(\mu*\varphi_{\sigma}/\nu*\varphi_{\sigma}\big) \in L^2(\nu*\varphi_{\sigma})$. On the other hand, for $\beta$-sub-Gaussian $\mu$ with $\beta<\sigma$, \eqref{ASSUM:KL_null} holds using the same argument as in Part~$(i)$, while \eqref{ASSUM:KL_alt} follows since
    \begin{equation}
      \int_{\RR^d}\frac{\mathsf{Var}_{\mu}\big(\varphi_{\sigma}(x-\cdot)\big)}{\nu*\varphi_{\sigma}(x)}\, dx \leq      \norm{\frac{\mu*\varphi_{\sigma}}{\nu*\varphi_{\sigma}}}_{\infty}\int_{\RR^d}\frac{\mathsf{Var}_{\mu}\big(\varphi_{\sigma}(x-\cdot)\big)}{\mu*\varphi_{\sigma}}\, dx<\infty.\label{eq:varcondbnddensrat}
    \end{equation} 
    Hence, all the conditions needed for \eqref{eq:KL-GS-onesamp-alt} are satisfied.

\medskip

\noindent\textbf{Part $\bm{(iii)}$:}
We apply Theorem \ref{Thm:KLdiv-limdist}$(iii)$ 
with $\rho=\mu*\gamma_{\sigma}$. We have  that $\mu*\gamma_{\sigma}$,  $\hat {\mu}_n*\gamma_\sigma$, $\hat {\nu}_n*\gamma_\sigma$ are all mutually absolutely continuous, and $p_{\hat {\nu}_n*\gamma_\sigma}=\hat {\nu}_n*\gamma_\sigma/\mu*\varphi_{\sigma}>0$.  Moreover, we have $ \kl{\hat \mu_n*\gamma_{\sigma}}{\hat \nu_n*\gamma_{\sigma}} <\infty$ a.s. via steps similar to those in Part~$(i)$. Also, since $\mu$ has compact support and $\supp(\hat \mu_n),\supp(\hat \nu_n)\subset \supp(\mu)$, we have
\begin{align}
 & \abs{\frac{\hat \mu_n *\varphi_{\sigma}(x)}{\hat \nu_n *\varphi_{\sigma}(x)}}=\frac{\sum_{i=1}^n e^{-\frac{\norm{x-X_i}^2}{2\sigma^2}}}{\sum_{i=1}^n e^{-\frac{\norm{x-Y_i}^2}{2\sigma^2}}} \leq \max_{1 \leq i \leq n}e^{-\frac{\norm{x-X_i}^2}{2\sigma^2}+\frac{\norm{x-Y_i}^2}{2\sigma^2}} 
 \leq  c'e^{c\norm{x}},\label{eq:ubnddenratio}
\end{align}
for some constants $c,c'>0$ that depend only on $\sigma$ and $\supp(\mu)$. Taking $q(x)=c'e^{c\norm{x}}$, we have
\begin{align}
   \frac{d\eta_1}{dx}(x)&=\mu*\varphi_{\sigma}(x),\notag \\
    \frac{d\eta_2}{dx}(x)&=\mu*\varphi_{\sigma}(x)+\mu*\varphi_{\sigma}(x)\,c'e^{c\norm{x}}.\notag
\end{align}

Recall that $n^{1/2} \big((\hat \mu_n*\varphi_{\sigma}/ \mu*\varphi_{\sigma})-1\big)=n^{-1/2} \sum_{i=1}^n Z_i$, where $Z_i$ is given in \eqref{eq:defZivars}. Let  $\tilde Z_i$ be defined similarly with $X_i$ replaced by $Y_i$, so that $n^{1/2} \big((\hat \nu_n*\varphi_{\sigma}/ \mu*\varphi_{\sigma})-1\big)=n^{-1/2} \sum_{i=1}^n \tilde Z_i$. Setting $r_n=n^{1/2}$, it follows from \eqref{eq:KL-twosample-null} that 
    \begin{align}
  n \kl{\hat \mu_n *\varphi_{\sigma}}{\hat \nu_n *\varphi_{\sigma}} \trightarrow{d} \frac{1}{2} \int_{\mathfrak{S}} \frac{\big(G_{\mu,\sigma}(x)-\tilde G_{\mu,\sigma}(x)\big)}{\mu*\varphi_{\sigma}(x)}^2 d x \notag
\end{align}
so long that
\begin{align}
  & \left( n^{-1/2} \sum_{i=1}^n  Z_i,n^{-1/2} \sum_{i=1}^n \tilde Z_i \right) \trightarrow{w} \left(\frac{G_{\mu,\sigma}}{\mu*\varphi_{\sigma}},\frac{\tilde G_{\mu,\sigma}}{\mu*\varphi_{\sigma}}\right)~~~\mbox{ in }L^2(\eta_1) \times L^2(\eta_2),\notag
    \end{align}
where $\big(G_{\mu,\sigma}, \tilde G_{\mu,\sigma}\big):=\big(G_{\mu,\sigma}(x), \tilde G_{\mu,\sigma}(y)\big)_{(x,y) \in \RR^d \times \RR^d}$ is a 2-dimensional Gaussian process. 

A sufficient condition for the joint weak convergence above is  
\begin{align}
  & n^{-1/2} \sum_{i=1}^n  Z_i \trightarrow{w} \frac{G_{\mu,\sigma}}{\mu*\varphi_{\sigma}} \mbox{ in }L^2(\eta_1)  \quad\mbox{ and } \quad n^{-1/2} \sum_{i=1}^n  \tilde Z_i \trightarrow{w} \frac{\tilde  G_{\mu,\sigma}}{\mu*\varphi_{\sigma}}\quad\mbox{ in } L^2(\eta_2).\label{eq:weakconvmarg}
    \end{align}
    To see this, first observe that both $\eta_1$ and $\eta_2$ are finite measures on $\big(\RR^d,\cB(\RR^d)\big)$, and hence $L^2(\eta_1)$ and $L^2(\eta_2)$ are Polish. Furthermore, given \eqref{eq:weakconvmarg}, $\big(n^{-1/2} \sum_{i=1}^n \tilde Z_i\big)_{n \in \NN}$ and $\big(n^{-1/2} \sum_{i=1}^n \tilde Z_i\big)_{n \in \NN}$ are both asymptotically tight and asymptotically measurable in $L^2(\eta_1)$ and $L^2(\eta_2)$, respectively \cite[Lemma 1.3.8]{AVDV-book}. This implies that  $\big(n^{-1/2} \sum_{i=1}^n \tilde Z_i,$ $n^{-1/2} \sum_{i=1}^n \tilde Z_i\big)_{n \in \NN}$  are jointly asymptotically tight and jointly measurable in $L^2(\eta_1) \times L^2(\eta_2)$ \cite[Lemma 1.4.3 and 1.4.4]{AVDV-book}. Then, by Polishness of $L^2(\eta_1) \times L^2(\eta_2)$,  the desired joint weak convergence holds if the finite-dimensional marginals of the joint process converge weakly (see \cite[Lemma 16]{GKRS-2022-entropicOT}), i.e., for every $x_1,\ldots, x_m, \tilde x_1,\ldots, \tilde x_l \in \RR^d$,
    \begin{align}
      \left(n^{-1/2} \sum_{i=1}^n  Z_i(x_1),\ldots, n^{-1/2} \sum_{i=1}^n  Z_i(x_m),n^{-1/2} \sum_{i=1}^n \tilde Z_i(\tilde x_1),\ldots,\sum_{i=1}^n \tilde Z_i(\tilde x_l)\right)  \label{eq:weakconfindimmarg}
    \end{align}
    converges weakly. The latter follows from the multivariate CLT.

Next, we apply Theorem \ref{Thm:CLT-in-Lp} to show that the sufficient condition given in \eqref{eq:weakconvmarg} is satisfied when $\mu,\nu$ have compact supports.
As in Part $(i)$, the process $Z_i$ above is jointly measurable,
has paths a.s. in $L^2(\eta_1)$, and satisfies the conditions in Theorem \ref{Thm:CLT-in-Lp}(ii). This follows by a similar argument as in Part $(i)$, under the assumption in \eqref{ASSUM:KL_null} holds. Hence, by Theorem \ref{Thm:CLT-in-Lp}, there exists a centered Gaussian process $G_{\mu,\sigma}/\mu*\varphi_{\sigma}$ such that the first weak convergence claim in the equation above holds. The second weak convergence claim holds via a similar argument with $Z_i$ and $\eta_1$ replaced by $\tilde Z_i$ and $\eta_2$, respectively, provided that
\[
\int_{\RR^d}\frac{\mathsf{Var}_{\mu}\big(\varphi_{\sigma}(x-\cdot)\big)}{\mu*\varphi_{\sigma}(x)}\big(1+e^{c\norm{x}}\big)\,dx<\infty.
\]
Thus, to prove \eqref{eq:KL-GS-twosamp-null}, it remains to show that the above equation holds when $\mu$ has compact support. This follows via a direct computation: 
\begin{flalign}
   &  \int_{\RR^d}\frac{\mathsf{Var}_{\mu}\big(\varphi_{\sigma}(x-\cdot)\big)}{\mu*\varphi_{\sigma}(x)}\, dx + \int_{\RR^d}\frac{\mathsf{Var}_{\mu}\big(\varphi_{\sigma}(x-\cdot)\big)}{\mu*\varphi_{\sigma}(x)}\, e^{c\norm{x}}dx \notag \\
     & \qquad  \qquad \leq \int_{\RR^d}\frac{\EE_{\mu}\big[\varphi^2_{\sigma}(x-\cdot)\big]}{\mu*\varphi_{\sigma}(x)}\, dx + \int_{\RR^d}\frac{\EE_{\mu}\big[\varphi^2_{\sigma}(x-\cdot)\big]}{\mu*\varphi_{\sigma}(x)}\, e^{c\norm{x}}dx \notag \\
     & \qquad \qquad\leq \int_{\RR^d}\frac{ \int_{\RR^d} e^{-\frac{\norm{x-y}^2}{\sigma^2}} d \mu(y)}{ \int_{\RR^d}e^{-\frac{\norm{x-y}^2}{2\sigma^2}}d \mu(y)}dx +\int_{\RR^d}\frac{ \int_{\RR^d} e^{-\frac{\norm{x-y}^2}{\sigma^2}} d \mu(y)}{ \int_{\RR^d}e^{-\frac{\norm{x-y}^2}{2\sigma^2}}d \mu(y)} e^{c\norm{x}} dx \notag\\
       & \qquad \qquad\leq \int_{\RR^d}e^{-\frac{\norm{x}^2}{2\sigma^2}}\frac{ \int_{\RR^d} e^{\frac{2 x\cdot y-\norm{y}^2}{\sigma^2}} d \mu(y)}{ \int_{\RR^d}e^{\frac{2 x\cdot y-\norm{y}^2}{2\sigma^2}}d \mu(y)}dx +\int_{\RR^d}e^{-\frac{\norm{x}^2}{2\sigma^2}}\frac{ \int_{\RR^d} e^{\frac{2 x\cdot y-\norm{y}^2}{\sigma^2}} d \mu(y)}{ \int_{\RR^d}e^{\frac{2 x\cdot y-\norm{y}^2}{2\sigma^2}}d \mu(y)} e^{c\norm{x}} dx \notag\\
          & \qquad \qquad\leq \int_{\RR^d}e^{-\frac{\norm{x}^2}{2\sigma^2}}\frac{ \int_{\RR^d} e^{\frac{2 \norm{x}\norm{y}-\norm{y}^2}{\sigma^2}} d \mu(y)}{ \int_{\RR^d}e^{\frac{-2 \norm{x}\norm{ y}-\norm{y}^2}{2\sigma^2}}d \mu(y)}dx +\int_{\RR^d}e^{-\frac{\norm{x}^2}{2\sigma^2}}\frac{ \int_{\RR^d} e^{\frac{2 \norm{x}\norm{y}-\norm{y}^2}{\sigma^2}} d \mu(y)}{ \int_{\RR^d}e^{\frac{-2 \norm{x}\norm{ y}-\norm{y}^2}{2\sigma^2}}d \mu(y)} e^{c\norm{x}} dx \notag\\
     &\qquad \qquad\leq  c' \int_{\RR^d} e^{-\frac{\norm{x}^2}{2\sigma^2}}e^{c\norm{x}} dx <\infty, \label{eq:finitevarint} &&
\end{flalign}
for some  constants $c,c'>0$  which depends on $\sigma$ and  $\supp(\mu)$.

\medskip

\noindent\textbf{Part $\bm{(iv)}$:} To prove \eqref{eq:KL-GS-twosamp-alt}, we utilize Theorem \ref{Thm:KLdiv-limdist}$(iv)$ 
with $\rho=\nu*\gamma_{\sigma}$. The positivity and absolute continuity of the probability measures as well as finiteness of KL divergences required in Theorem \ref{Thm:KLdiv-limdist}$(iv)$ follow from similar steps as above, which are thus omitted. Likewise, \eqref{eq:ubnddenratio} holds (possibly with different constants $c,c'$) since $\mu$ and $\nu$ have compact supports. Moreover, using Jensen's inequality and steps similar to those leading to \eqref{eq:finitevarint}, we have
\begin{align}
&\frac{\nu*\varphi_{\sigma}(x)^2}{\mu*\varphi_{\sigma}(x) } \leq \frac{\EE_{\nu}\big[\varphi^2_{\sigma}(x-\cdot)\big]}{\mu*\varphi_{\sigma}(x)} = e^{-\frac{\norm{x}^2}{2\sigma^2}} \frac{ \int_{\RR^d} e^{\frac{2 \norm{x}\norm{y}-\norm{y}^2}{\sigma^2}} d \nu(y)}{ \int_{\RR^d}e^{\frac{-2 \norm{x}\norm{ y}-\norm{y}^2}{2\sigma^2}}d \mu(y)} \leq  c'  e^{-\frac{\norm{x}^2}{2\sigma^2}}e^{c\norm{x}} \leq \tilde c, \notag \\
&\frac{\nu*\varphi_{\sigma}(x)}{\mu*\varphi_{\sigma}(x) }  \leq \frac{ \int_{\RR^d} e^{\frac{2 \norm{x}\norm{y}-\norm{y}^2}{2\sigma^2}} d \nu(y)}{ \int_{\RR^d}e^{\frac{-2 \norm{x}\norm{ y}-\norm{y}^2}{2\sigma^2}}d \mu(y)} \leq  c'  e^{c\norm{x}}, \label{eq:boundratdens}
\end{align}
 for some  constants $c,c', \tilde c>0$  which depends only on $\sigma$ and the supports of $\mu,\nu$. Hence,  $\norm{\nu*\varphi_{\sigma}^2/\mu*\varphi_{\sigma}}_{\infty}<\infty$ and  $\mu*\varphi_{\sigma},~\log \big(\mu*\varphi_{\sigma}/\nu*\varphi_{\sigma}\big) \in L^2(\nu*\varphi_{\sigma})$ (note the above inequalities also hold with $\mu$ and $\nu$ interchanged). Next, observe that the measures $\eta_1$ and $\eta_2$ given in Theorem \ref{Thm:KLdiv-limdist}$(iv)$ with  
$q(x)=c'e^{c\norm{x}}$ have Lebesgue densities
\begin{align}
   \frac{d\eta_1}{dx}(x)&=\nu*\varphi_{\sigma}(x)+\frac{\big(\nu*\varphi_{\sigma}(x)\big)^2}{\mu*\varphi_{\sigma}(x)},\notag \\
    \frac{d\eta_2}{dx}(x)&=\nu*\varphi_{\sigma}(x)+c' \nu*\varphi_{\sigma}(x)\,e^{c\norm{x}}+\mu*\varphi_{\sigma}(x),\label{eq:normspmeas}
\end{align}
and hence are finite measures based on the inequalities above. Then, the Polishness of  $L^2(\eta_1) \times L^2(\eta_2)$ and the discussion in Part $(iii)$ implies via \eqref{eq:KL-twosample-alt} with $r_n=n^{1/2}$ that 
\begin{align}
  & n^{\frac 12}\big(\kl{\hat \mu_n*\gamma_\sigma}{\hat \nu_n*\gamma_{\sigma}}-\kl{\mu*\gamma_{\sigma}}{\nu*\gamma_{\sigma}}\big)\notag \\
  &\qquad \qquad \qquad\qquad\trightarrow{d} \int_{\RR^d} G_{\mu,\sigma}(x) \log \left(\frac{\mu*\varphi_{\sigma}(x)}{\nu*\varphi_{\sigma}(x)}\right) d x-\int_{\RR^d} G_{\nu,\sigma}(x) \frac{\mu*\varphi_{\sigma}(x)}{\nu*\varphi_{\sigma}(x)} d x, \label{eq:KL-GS-twosamp-alt-interm}
\end{align} 
provided that
\begin{align} \label{eq:wcden2sampalt}
\begin{split}
 &  n^{1/2} \left(\frac{\hat\mu_n*\varphi_{\sigma}}{\nu*\varphi_{\sigma}}-\frac{\mu*\varphi_{\sigma}}{\nu*\varphi_{\sigma}}\right) \trightarrow{w} \frac{G_{\mu,\sigma}}{\nu*\varphi_{\sigma}}\mbox{ in }L^2(\eta_1), \\
   & n^{1/2} \left(\frac{\hat\nu_n*\varphi_{\sigma}}{\nu*\varphi_{\sigma}}-1\right) \trightarrow{w} \frac{G_{\nu,\sigma}}{\nu*\varphi_{\sigma}}\mbox{ in }L^2(\eta_2).
\end{split}
 \end{align}
Resorting to Theorem \ref{Thm:CLT-in-Lp} once more, we next show that the weak convergence requirements in \eqref{eq:wcden2sampalt} hold if 
\begin{align}
      & \int_{\RR^d}\frac{\mathsf{Var}_{\mu}\big(\varphi_{\sigma}(x-\cdot)\big)}{\nu*\varphi_{\sigma}(x)}\, dx +\int_{\RR^d}\frac{\mathsf{Var}_{\mu}\big(\varphi_{\sigma}(x-\cdot)\big)}{\mu*\varphi_{\sigma}(x)}\, dx<\infty, \notag 
      \end{align}
      and
      \begin{align}
       & \int_{\RR^d}\frac{\mathsf{Var}_{\nu}\big(\varphi_{\sigma}(x-\cdot)\big)}{\nu*\varphi_{\sigma}(x)}\, \big(1+e^{c\norm{x}}\big)dx+\int_{\RR^d}\frac{\mathsf{Var}_{\nu}\big(\varphi_{\sigma}(x-\cdot)\big)\mu*\varphi_{\sigma}(x)}{\big(\nu*\varphi_{\sigma}(x)\big)^2}\, dx <\infty,\notag
\end{align}
 respectively. To see this, note that the second term in the penultimate equation and the first term in the last equation can be bounded as shown in Part $(iii)$. For the remaining terms, we have from \eqref{eq:ratiodensbnd}  via steps leading to \eqref{eq:finitevarint}-\eqref{eq:boundratdens} that 
 \begin{align}
     \int_{\RR^d}\mspace{-8mu}\frac{\mathsf{Var}_{\nu}\big(\varphi_{\sigma}(x-\cdot)\big)\mu\mspace{-2mu}*\mspace{-2mu}\varphi_{\sigma}(x)}{\big(\nu*\varphi_{\sigma}(x)\big)^2}\, dx \leq c'\mspace{-5mu}\int_{\RR^d}\mspace{-10mu} e^{\mspace{-3mu}-\frac{\norm{x}^2}{2\sigma^2}}\mspace{-2mu}e^{c\norm{x}}\frac{\mu\mspace{-2mu}*\mspace{-2mu}\varphi_{\sigma}(x)}{\nu\mspace{-2mu}*\mspace{-2mu}\varphi_{\sigma}(x)} dx \leq    c'^2\mspace{-5mu}\int_{\RR^d}\mspace{-10mu} e^{-\frac{\norm{x}^2}{2\sigma^2}} e^{2c\norm{x}} dx  \mspace{-1.5mu}<\mspace{-1.5mu}\infty, \notag 
 \end{align}
 and
 \begin{align} \int_{\RR^d}\frac{\mathsf{Var}_{\mu}\big(\varphi_{\sigma}(x-\cdot)\big)}{\nu*\varphi_{\sigma}(x)}\, dx & \leq  \int_{\RR^d}\frac{\EE_{\mu}\big[\varphi^2_{\sigma}(x-\cdot)\big]}{\nu*\varphi_{\sigma}(x)}\, dx \leq  c' \int_{\RR^d} e^{-\frac{\norm{x}^2}{2\sigma^2}}e^{c\norm{x}} \, dx <\infty. \notag
 \end{align}
Finally, the proof of the claim is completed  by noting that the RHS of \eqref{eq:KL-GS-twosamp-alt-interm} is a zero mean Gaussian random variable with variance
\begin{align}
v_{2,f_{\mathsf{KL}}}^2(\mu,\nu,\sigma)&:=\sum_{1 \leq i,j \leq 2} \int_{\RR^d}\int_{\RR^d}  \Sigma^{(i,j)}_{\mu,\nu,\sigma}(x,y) L_{i,f_{\mathsf{KL}}}(x)L_{j,f_{\mathsf{KL}}}(y) d x\, d y\notag \\
    &=\sum_{1 \leq i,j \leq 2}\int_{\RR^d}\int_{\RR^d}  \Sigma^{(i,j)}_{\mu,\nu,\sigma}(x,y) \tilde L_{i,f_{\mathsf{KL}}}(x) \tilde L_{j,f_{\mathsf{KL}}}(y) d x\, d y, \label{eq:var-smoothedKL2samp} 
\end{align}
where $\tilde L_{1,f_{\mathsf{KL}}}:=\log \left(\mu*\varphi_{\sigma}/\nu*\varphi_{\sigma}\right)$ and $\tilde L_{2,f_{\mathsf{KL}}}:= -\mu*\varphi_{\sigma}/\nu*\varphi_{\sigma}$. The second equality above uses 
\begin{align}
\int_{\RR^d}\int_{\RR^d}\Sigma^{(1,1)}_{\mu,\nu,\sigma}(x,y)dx dy=\int_{\RR^d}\int_{\RR^d} \mathsf{cov}\big(\varphi_{\sigma}(x-X),\varphi_{\sigma}(y-X)\big)dx dy=0.   \notag
\end{align}

\subsubsection{Proof of Lemma \ref{lem:fdivnullonesamplb}} \label{lem:fdivnullonesamplb-proof}
Without loss of generality, we may assume that there exists a subsequence $(n_k)_{k \in \NN}$ such that $\EE\big[\fdiv{\hat \mu_{n_k}*\gamma_\sigma }{\mu*\gamma_\sigma}\big]<\infty$ for all $k \in \NN$, since otherwise the LHS of \eqref{eq:lowerbndfdiv} is infinite and there is nothing to prove. Henceforth, we take $n$ within such a subsequence. Note that since $\hat \mu_n*\gamma_\sigma \ll \mu*\gamma_\sigma$, $\fdiv{\hat \mu_n*\gamma_\sigma}{\mu*\gamma_\sigma} =\EE_{\mu*\gamma_\sigma}\big[f \circ (\hat \mu_n*\varphi_\sigma/ \mu*\varphi_\sigma)\big]$. Applying Taylor's expansion of  $f(x)$ at $x=1$ and observing that $f(1)=0$, we have
\begin{align}
  &f\left(\frac{\hat \mu_n*\varphi_\sigma(x)} {\mu*\varphi_\sigma(x)}\right)=f'(1)\mspace{-3mu}\left(\frac{\hat \mu_n*\varphi_\sigma(x)} {\mu*\varphi_\sigma(x)}\mspace{-2mu}-\mspace{-2mu}1\right)\mspace{-2mu}+\mspace{-2mu}\int_{0}^1 (1-\tau)f''\mspace{-2mu}\left((1\mspace{-2mu}-\mspace{-2mu}\tau)\mspace{-2mu}+\mspace{-2mu}\tau \frac{\hat \mu_n*\varphi_\sigma(x)}{ \mu*\varphi_\sigma(x)}\right)\mspace{-2mu}\left(\frac{\hat \mu_n*\varphi_\sigma(x)} {\mu*\varphi_\sigma(x)}\mspace{-2mu}-\mspace{-2mu}1\right)^2 \mspace{-2mu}d \tau. \label{eq:Taylorexpfdiv}
\end{align}
Let $\vartheta_n(x):=n^{1/2}\big(\hat \mu_n*\varphi_\sigma(x)-\mu*\varphi_\sigma(x)\big)$. Taking expectation first w.r.t. to $\mu*\gamma_\sigma$ and then w.r.t. $\mu^{\otimes n}$ in the above equation, and observing that the first term on the RHS integrates to zero, we obtain
\begin{align}
 n \EE\big[\fdiv{\hat \mu_n*\gamma_\sigma} {\mu*\gamma_\sigma}\big] & =\EE\left[\int_{\RR^d}\int_{0}^1 (1-\tau)f''\left(\frac{\mu*\varphi_\sigma(x)+\tau n^{-\frac 12} \vartheta_n(x)}{ \mu*\varphi_\sigma(x)}\right)\frac{\vartheta_n^2(x)}{\mu*\varphi_\sigma(x)}d \tau dx\right] \notag \\
 & =\int_{\RR^d}\int_{0}^1 (1-\tau)\,\EE\left[f''\left(\frac{\mu*\varphi_\sigma(x)+\tau n^{-\frac 12} \vartheta_n(x)}{ \mu*\varphi_\sigma(x)}\right) \frac{\vartheta_n^2(x)}{\mu*\varphi_\sigma(x)}\right]d \tau dx,\notag
\end{align}
where the final equality uses Fubini's theorem. Noting that the integrand is nonnegative (by non-negativity of $f''$), Fatou's lemma implies
\begin{align}
 &\liminf_{n \rightarrow \infty} n \EE\big[\fdiv{\hat \mu_n*\gamma_\sigma} {\mu*\gamma_\sigma}\big]  \geq  \int_{\RR^d}\int_{0}^1  (1-\tau)\liminf_{n \rightarrow \infty}\EE\left[f''\left(\frac{\mu*\varphi_\sigma(x)+\tau n^{-\frac 12} \vartheta_n(x)}{ \mu*\varphi_\sigma(x)}\right) \frac{\vartheta_n^2(x)}{\mu*\varphi_\sigma(x)}\right] d \tau dx.  \label{eq:applfatlemmahelsq}
\end{align}
Next, observe that since $\vartheta_n(x)=\sum_{i=1}^n Z_i(x)/\sqrt{n}$ with $Z_i(x)=\varphi_{\sigma}(x-X_{i}) - \mu*\varphi_{\sigma}(x)$ and $\abs{Z_i(x)} \leq  (2\pi \sigma^{2})^{-d/2}$, the CLT implies that for every $k \in \NN$, we have
\begin{align}
 \lim_{n \rightarrow \infty} \EE\left[f''\left(\frac{\mu*\varphi_\sigma(x)+\tau n^{-\frac 12} \vartheta_n(x)}{ \mu*\varphi_\sigma(x)}\right) \vartheta_n^2(x)\wedge k\right]= \E\left[f''(1)G_{\mu,\sigma}^2(x) \wedge k\right]. \label{eq:convexpclttrunc}
\end{align}
Indeed, this will follow by the definition of weak convergence applied to the bounded continuous map $y \mapsto y \wedge k$, provided that
\begin{align}
  f''\left(\frac{\mu*\varphi_\sigma(x)+\tau n^{-\frac 12} \vartheta_n(x)}{ \mu*\varphi_\sigma(x)}\right) \vartheta_n^2(x) \trightarrow{d }f''(1)\,G_{\mu,\sigma}^2(x),\quad \forall \,x\in\RR^d. \notag
\end{align}
However, the above come from the CLT, whereby $\vartheta_n(x)\trightarrow{d } G_{\mu,\sigma}(x)$, and the extended continuous mapping theorem (see \cite[Theorem 1.11.1]{AVDV-book}) applied to the sequence of functions $g_{n,x}(y)=f''\big((\mu*\varphi_\sigma(x)+\tau n^{-\frac 12} y)/ \mu*\varphi_\sigma(x)\big)  y^2$  satisfying  $\lim_{n \rightarrow \infty} g_{n,x}(y_n)=g_x(y):=y^2 f''(1) $ for $y_n \rightarrow y$ (the latter follows by continuity of $f''$). Then, taking  $k \rightarrow \infty$ in \eqref{eq:convexpclttrunc} and using monotone convergence theorem yields
\begin{equation}
  \lim_{n \rightarrow \infty} \EE\left[f''\left(\frac{\mu*\varphi_\sigma(x)+\tau n^{-\frac 12} \vartheta_n(x)}{ \mu*\varphi_\sigma(x)}\right) \vartheta_n^2(x)\right]= f''(1)\E\left[G_{\mu,\sigma}^2(x) \right]. \notag
\end{equation}
Substituting the above in \eqref{eq:applfatlemmahelsq} and evaluating the integral, we conclude that
\begin{align}
 \liminf_{n \rightarrow \infty} n \EE\big[\fdiv{\hat \mu_n*\gamma_\sigma} {\mu*\gamma_\sigma}\big] &\geq \frac{f''(1)}{2} \int_{\RR^d}  \frac{\E\big[G_{\mu,\sigma}^2(x)\big]}{\mu*\varphi_\sigma(x)} dx= \frac{f''(1)}{2}\int_{\RR^d}\frac{\mathsf{Var}_{\mu}\big(\varphi_{\sigma}(x-\cdot)\big)}{\mu*\varphi_{\sigma}(x)}\, dx. \notag
\end{align}

\subsubsection{Proof of Corollary \ref{Cor:GS-KL-bs}}\label{Sec:bsconsist-proof}
As already noted, the claim follows from  \cite[Theorem 23.9]{Va1998} and the linearity of the first order Hadamard derivative of the KL divergence functional, provided the tangent cone $\mathfrak{T}_{\Theta}(\theta^{\star})$ contains a non-trivial linear subspace. We show this below by considering the one-sample case. The proof for the two-sample case uses identical arguments, and hence is omitted.

Recall  that  $\kl{\hat \mu_n*\gamma_{\sigma}}{\nu*\gamma_{\sigma}}=\Phi\big(p_{\hat \mu_n*\gamma_{\sigma}}-p_{\mu*\gamma_{\sigma}},0)$ with $\phi(x,y)= \phi_{\mathsf{KL,1}}(x,y)=x \log x$ in \eqref{eq:defPhifunc}, where $p_{\hat \mu_n*\gamma_{\sigma}}=\hat \mu_n*\varphi_{\sigma}/ \nu*\varphi_{\sigma}$ and  $p_{\mu*\gamma_{\sigma}}= \mu*\varphi_{\sigma}/ \nu*\varphi_{\sigma}$  (see the proof of Theorem \ref{Thm:KLdiv-limdist}$(ii)$). The first order Hadamard derivative is given by 
\[\Phi_{\theta^{\star}}'(h_1,h_2)=\int_{\RR^d}  h_1(x) \log \big(p_{\mu*\gamma_{\sigma}}(x)\big)\nu*\varphi_{\sigma}(x)\, dx,\]
for all  $(h_1,h_2) \in \mathfrak{T}_{\Theta}(\theta^{\star})=\mathrm{cl}\big(\big\{\big((g_1-p_{\mu*\gamma_{\sigma}})/t,0\big):(g_1-p_{\mu*\gamma_{\sigma}},0) \in \Theta, ~t>0\big\}\big)$, where 
\begin{align}
     \Theta &=\left\{(g_1-p_{\mu*\gamma_{\sigma}},0) \in \mathfrak{D}: g_1\geq 0,  \norm{g_1}_{1,\nu*\gamma_{\sigma}}=1\right\},\notag \\
     \mathfrak{D}&=\left\{(g_1-p_{\mu*\gamma_{\sigma}},g_2-1):\, g_1,g_2 \in  L^1(\nu*\gamma_{\sigma}), \norm{g_1-p_{\mu*\gamma_{\sigma}}}_{2,\eta}+\norm{g_2-1)}_{2,\nu*\gamma_{\sigma}}<\infty \right\}, \notag 
\end{align}
and $\eta$ that has Lebesgue density $\nu*\varphi_{\sigma}+(\nu*\varphi_{\sigma})^2/\mu*\varphi_{\sigma}$. Further, the proof of Proposition~\ref{Prop:GS-KL-limdist}$(ii)$ shows that under the conditions therein, $\eta$ is a $\sigma$-finite measure and
\begin{align}
 n^{1/2} \left(p_{\hat \mu_n*\gamma_{\sigma}}-p_{\mu*\gamma_{\sigma}}\right) \trightarrow{w} \frac{G_{\mu,\sigma}}{\nu*\varphi_{\sigma}}\quad\mbox{ in }L^2(\eta). \label{eq:wcempmeas}
 \end{align}
 As $\big(n^{1/2} \left(p_{\hat \mu_n*\gamma_{\sigma}}-p_{\mu*\gamma_{\sigma}}\right),0\big)$
 takes values in $\mathfrak{T}_{\Theta}(\theta^{\star})$, the Portmanteau theorem guarantees that 
$\supp\big(G_{\mu,\sigma}/\nu*\varphi_{\sigma}\big)$ is contained in $\mathfrak{T}_{\Theta}(\theta^{\star})$. As $L^2(\eta)$ is a separable Banach space and $G_{\mu,\sigma}/\nu*\varphi_{\sigma}$ is a centered  $L^2(\eta)$-valued Gaussian random variable, \cite[Lemma 5.1]{VanderVaart-2008} then guarantees that $\mathfrak{T}_{\Theta}(\theta^{\star})$ contains the reproducing kernel Hilbert space corresponding to $G_{\mu,\sigma}/\nu*\varphi_{\sigma}$, and hence, in particular, a linear subspace. 

Given the above,  \cite[Theorem 23.9]{Va1998} applies and we obtain 
\begin{align}
    n^{\frac 12}\left(\mathsf{D}_{\mathsf{KL}}\big(\hat \mu_n^B*\gamma_\sigma \|\nu*\gamma_{\sigma}\big)-\kl{\hat \mu_n*\gamma_{\sigma}}{\nu*\gamma_{\sigma}}\right) &\trightarrow{d} \int_{\RR^d} G_{\mu,\sigma}(x) \log \left(\frac{\mu*\varphi_{\sigma}(x)}{\nu*\varphi_{\sigma}(x)}\right) d x  \sim N\big(0, v_{1,f_{\mathsf{KL}}}^2(\mu,\nu,\sigma)\big), \label{eq:condconv-hadder-bs}
\end{align}
 so long that
\begin{align}
 n^{1/2} \left(\frac{\hat\mu_n^B*\varphi_{\sigma}}{\nu*\varphi_{\sigma}}-\frac{\hat \mu_n*\varphi_{\sigma}}{\nu*\varphi_{\sigma}}\right) \trightarrow{w} \frac{G_{\mu,\sigma}}{\nu*\varphi_{\sigma}}\quad\mbox{ in }L^2(\eta),\notag
 \end{align}
where both convergences are conditionally in probability.  However, the last weak convergence above holds via \cite[Remark 2.5, Page 860]{Gine-Zinn-1990} provided that \eqref{eq:wcempmeas} is satisfied.

Next, note by \cite[Theorem 1.12.4]{AVDV-book} (and the discussion preceding it) that \eqref{eq:condconv-hadder-bs} is equivalent to
 \begin{align}
 \sup_{f \in \mathsf{BL}_1(\RR) }\EE_{B}\left[\abs{f\left(n^{\frac 12}\left(\mathsf{D}_{\mathsf{KL}}\big(\hat \mu_n^B*\gamma_\sigma \|\nu*\gamma_{\sigma}\big)-\kl{\hat \mu_n*\gamma_{\sigma}}{\nu*\gamma_{\sigma}}\right)\right)-f(W_1)}\right] =o_{\PP}(1),   \notag
\end{align}
where $\EE_{B}$ denotes the conditional expectation  given $(X_1,\ldots,X_n,Y_1,\ldots,Y_n)$,  $\mathsf{BL}_1(\RR)$ is the set of bounded 1-Lipschitz continuous functions, and $W_1 \sim N\big(0,v_{1,f_{\mathsf{KL}}}^2(\mu,\nu,\sigma)\big)$. Conclude by noting that the last equation  implies the claim in Corollary \ref{Cor:GS-KL-bs}$(i)$, which is equivalent to convergence in distribution conditionally in probability, since the so-called Dudley metric $d_{\mathsf{BL}_1}(\mu,\nu):=\sup_{f \in \mathsf{BL}_1(\RR)} \abs{\EE_{\mu}[f]-\EE_{\nu}[f]}$ metrizes weak convergence in $\RR$.

\subsection{Proof of Proposition \ref{prop:HTperf}}
We will use the following reformulation of a result given in
\cite{Bickel-98}.
\begin{lemma}[Corollary 2 in Appendix A.9 of \cite{Bickel-98}] \label{lem:bickel98}
Suppose $(\pi_n)_{n \in \NN_0}$ is such that there exists $h \in L^2(\pi_0)$ with $\lim_{n \rightarrow \infty}n \helsq{\pi_n}{\pi_0} = \norm{h/2}^2_{2,\pi_0}$ and $\int_{\RR^d \times \RR^d}h\, d\pi_0=0$. For $(X^n,Y^n) \sim \pi_0^{\otimes n}$, define $\Lambda_{n,\pi_n,\pi_0}(X^n,Y^n):=  \sum_{i=1}^n \log\big(d\pi_n(X_i,Y_i)/d\pi_0(X_i,Y_i)\big) $ and $\sigma_h:=\norm{h}_{2,\pi_0}$. Then 
\begin{align}  \begin{split}\label{eq:weaklimloglik}
   & \Lambda_{n,\pi_n,\pi_0}(X^n,Y^n)-n^{-\frac 12} \sum_{i=1}^n h(X_i)+0.5\sigma_h^2=o_{\PP}(1), \\
  &  \Lambda_{n,\pi_n,\pi_0}(X^n,Y^n) \trightarrow{d}  N\big(-0.5\sigma_h^2,\sigma_h^2\big).
    \end{split}     
    \end{align}
\end{lemma}
A short proof of the above claim is given in Appendix \ref{lem:bickel98-proof} for completeness.

\medskip

For analyzing the type II error probability of the hypothesis test in \eqref{HT-smoothed-KL-DP}, we require a refined version of the limit distribution result in \eqref{eq:KL-GS-twosamp-alt} that accounts for the dependence of $\pi_n$  on $n$. More specifically, we will show that for $(X^n,Y^n) \sim \pi_n^{\otimes n}$,
\begin{align}
  & n^{\frac 12}\big(\kl{\hat \mu_n*\gamma_\sigma}{\hat \nu_n*\gamma_{\sigma}}-\kl{\mu_n*\gamma_{\sigma}}{\nu_n*\gamma_{\sigma}}\big)\notag \\
  &\qquad \qquad \qquad \qquad\qquad \qquad\trightarrow{d} \int_{\RR^d}  G_{\mu_0,\sigma}(x) \log \left(\frac{\mu_0*\varphi_{\sigma}(x)}{\nu_0*\varphi_{\sigma}(x)}\right) d x-\int_{\RR^d}  G_{\nu_0,\sigma}(x) \frac{\mu_0*\varphi_{\sigma}(x)}{\nu_0*\varphi_{\sigma}(x)} d x\notag \\
  & \qquad \qquad\qquad \qquad\qquad \qquad\sim N\big(0,v_{2,f_{\mathsf{KL}}}^2(\mu_0,\nu_0,\sigma)\big),\label{eq:localaltlimdist}
\end{align}
where $(G_{\mu_0,\sigma},G_{\nu_0,\sigma})$ is a centered Gaussian process with covariance function  $\Sigma_{\mu_0,\nu_0,\sigma}$ as given in \eqref{eq:covfuncGP} and $v_{2,f_{\mathsf{KL}}}^2(\mu_0,\nu_0,\sigma)$ is specified in \eqref{eq:var-smoothedKL2samp}. We shall establish the above under Assumption \ref{Assump-localalt}$(i)$ using Lemma \ref{lem:bickel98} and an application of Le Cam's third lemma \cite[Theorem 3.10.7]{AVDV-book}.  
Note that $(\pi_n)_{n \in \NN_0}$ satisfying Assumption \ref{Assump-localalt} also fulfills the conditions in Lemma \ref{lem:bickel98}, and hence \eqref{eq:weaklimloglik} holds for $(X^n,Y^n) \sim \pi_0^{\otimes n}$. Consequently, so long that \eqref{eq:wcden2sampalt} holds with $(\mu,\nu)=(\mu_0,\nu_0)$, we have
\begin{align}
 \bigg(  n^{1/2} \left(\frac{\hat\mu_n*\varphi_{\sigma}}{\nu_0*\varphi_{\sigma}}-\frac{\mu_0*\varphi_{\sigma}}{\nu_0*\varphi_{\sigma}}\right), n^{1/2} \left(\frac{\hat\nu_n*\varphi_{\sigma}}{\nu_0*\varphi_{\sigma}}-1\right),&\Lambda_{n,\pi_n,\pi_0}\bigg)\trightarrow{w} \left(\frac{G_{\mu_0,\sigma}}{\nu_0*\varphi_{\sigma}},\frac{G_{\nu_0,\sigma}}{\nu_0*\varphi_{\sigma}},W \right) , \notag 
\end{align}
in $L^2(\eta_1)\times L^2(\eta_2) \times \RR$, where $\eta_1$ and $\eta_2$ as given in \eqref{eq:normspmeas} are finite measures for compactly supported $(\mu_0,\nu_0)$, $W \sim  N\big(-0.5\sigma_h^2,\sigma_h^2\big)$, and the centered Gaussian process  
$\big(G_{\mu_0,\sigma}(x), G_{\nu_0,\sigma}(y),W\big)_{(x,y) \in \RR^d \times \RR^d}$ has covariance function  $\tilde \Sigma_{\mu,\nu,\sigma}: (\RR^d \times \RR^d) \times (\RR^d \times \RR^d) \rightarrow \RR^{3 \times 3}$ given by 
\begin{align}
\tilde \Sigma_{\mu_0,\nu_0,\sigma}\big((x,y),(\tilde x,\tilde y)\big)
&:=\mspace{-2 mu}
 \left[ {\begin{array}{ccc}
   \EE\big[G_{\mu_0,\sigma}(x)G_{\mu_0,\sigma}(\tilde x)\big] & \EE\big[G_{\mu_0,\sigma}(x)G_{\nu_0,\sigma}(\tilde y)\big] & \EE\big[G_{\mu_0,\sigma}(x)h(X,Y)\big] \\[2 pt]
   \EE\big[G_{\nu_0,\sigma}(y)G_{\mu_0,\sigma}(\tilde x)\big] & \EE\big[G_{\nu_0,\sigma}(y)G_{\nu_0,\sigma}(\tilde y)\big]& \EE\big[G_{\nu_0,\sigma}(y)h(X,Y)\big]\\
     \EE\big[h(X,Y)G_{\mu_0,\sigma}(\tilde x)\big] & \EE\big[h(X,Y)G_{\nu_0,\sigma}(\tilde y)\big] & \sigma_h^2 
  \end{array} } \right] \notag \\
  &=\mspace{-2 mu}\left[ \arraycolsep=1.3 pt\def\arraystretch{0.8}{\begin{array}{ccc}
\mathsf{cov}\big(\varphi_{\sigma}(x-X),\varphi_{\sigma}(\tilde x-X)\big) & \mathsf{cov}\big(\varphi_{\sigma}(x-X),\varphi_{\sigma}(\tilde y-Y)\big) &\EE\big[\varphi_{\sigma}(x-X)h(X,Y)\big] \\[2 pt]
\mathsf{cov}\big(\varphi_{\sigma}(y-Y),\varphi_{\sigma}(\tilde x-X)\big) & \mathsf{cov}\big(\varphi_{\sigma}(y-Y),\varphi_{\sigma}(\tilde y-Y)\big) &\EE\big[\varphi_{\sigma}(y-Y)h(X,Y)\big]  \\[2 pt]
\EE\big[h(X,Y)\varphi_{\sigma}(\tilde x-X)\big] & \EE\big[h(X,Y)\varphi_{\sigma}(\tilde y-Y)\big] & \sigma_h^2 
  \end{array} } \right], \notag
\end{align}
for $(X,Y) \sim \pi_0$. This follows by \eqref{eq:weaklimloglik}, the convergence of finite-dimensional marginals of the joint process (via the multivariate CLT), and the separability of $L^2(\eta_1)\times L^2(\eta_2) \times \RR$; see arguments leading to \eqref{eq:weakconfindimmarg}.  
\medskip

Having the above, Le Cam's third lemma \cite[Theorem 3.10.7 and Example 3.10.8]{AVDV-book} implies that for $(X^n,Y^n) \sim \pi_n^{\otimes n}$, we have
\begin{align}
 \bigg(  n^{1/2} \left(\frac{\hat\mu_n*\varphi_{\sigma}}{\nu_0*\varphi_{\sigma}}-\frac{\mu_0*\varphi_{\sigma}}{\nu_0*\varphi_{\sigma}}\right), n^{1/2} \left(\frac{\hat\nu_n*\varphi_{\sigma}}{\nu_0*\varphi_{\sigma}}-1\right)\bigg)\trightarrow{w} &\left(\frac{\bar G_{\mu_0,\sigma,h}}{\nu_0*\varphi_{\sigma}},\frac{\bar G_{\nu_0,\sigma,h}}{\nu_0*\varphi_{\sigma}}\right)\notag 
\end{align}
in $L^2(\eta_1)\times L^2(\eta_2)$, where the Gaussian process
$\big(\bar G_{\mu_0,\sigma,h}, \bar G_{\nu_0,\sigma,h}\big)$ has mean function  $(m_{1,h},m_{2,h}):=\big(\EE_{\pi_0}\big[h(X,Y)$ $\varphi_{\sigma}( \cdot-X)\big],\EE_{\pi_0}\big[h(X,Y)\varphi_{\sigma}(\cdot-Y)\big]\big)$ and  covariance function  $\Sigma_{\mu,\nu,\sigma}$ as given in \eqref{eq:covfuncGP}. Since $\mu_n$ and $\nu_n$ are compactly supported (on $[-b,b]^d$), the proof of Proposition \ref{Prop:GS-KL-limdist}$(iv)$ applies and results in 
\begin{align}
  & n^{\frac 12}\big(\kl{\hat \mu_n*\gamma_\sigma}{\hat \nu_n*\gamma_{\sigma}}-\kl{\mu_0*\gamma_{\sigma}}{\nu_0*\gamma_{\sigma}}\big)\notag \\
  &\qquad \qquad \qquad \qquad\qquad \qquad\trightarrow{d} \int_{\RR^d}  \bar G_{\mu_0,\sigma,h}(x) \log \left(\frac{\mu_0*\varphi_{\sigma}(x)}{\nu_0*\varphi_{\sigma}(x)}\right) d x-\int_{\RR^d} \bar G_{\nu_0,\sigma,h} (x)\frac{\mu_0*\varphi_{\sigma}(x)}{\nu_0*\varphi_{\sigma}(x)} d x. \notag
\end{align} 
Next, note that the finite measures $\eta_1$ and $\eta_2$ from \eqref{eq:normspmeas} have bounded Lebesgue densities. Consequently, the convergence given in \eqref{eq:convmarglocalt} in $L^{\infty}(\lambda) \times L^{\infty}(\lambda)$  implies the same in $L^{\infty}(\eta_1) \times L^{\infty}(\eta_2)$ and $L^{2}(\eta_1) \times L^{2}(\eta_2)$. It then follows via the definition of the Hadamard first derivative (see  \eqref{eq:hadderklgen}), similarly to \eqref{eq:KL-GS-twosamp-alt-interm}, that
\begin{align}
  & n^{\frac 12}\big(\kl{\mu_n*\gamma_\sigma}{ \nu_n*\gamma_{\sigma}}-\kl{\mu_0*\gamma_{\sigma}}{\nu_0*\gamma_{\sigma}}\big)\notag \\
  &\qquad \qquad \qquad \qquad\qquad\qquad\trightarrow{d} \int_{\RR^d}  m_{1,h}(x) \log \left(\frac{\mu_0*\varphi_{\sigma}(x)}{\nu_0*\varphi_{\sigma}(x)}\right) d x-\int_{\RR^d}m_{2,h}(x)\frac{\mu_0*\varphi_{\sigma}(x)}{\nu_0*\varphi_{\sigma}(x)} d x. \notag
\end{align} 
Subtracting the last equation from the penultimate one  leads to \eqref{eq:localaltlimdist}.

\medskip

Armed with \eqref{eq:localaltlimdist}, we proceed to analyze the  type I and type II error probabilities, i.e., $e_{1,n}(T_n):=\PP(T_n>t_n|H_0)$ and $e_{2,n}(T_n):=\PP(T_n\leq t_n|H_{1,n})$,  respectively.
Consider the test statistic $T_n=\kl{\hat \mu_{n}*\gamma_{\sigma}}{\hat \nu_{n}*\gamma_{\sigma}}$ with critical value $t_n= \epsilon+c n^{-1/2}$ 
for some constant $c$ that will be specified later. Define the event  $\cE_{n,\epsilon,c}:=\big\{\kl{\hat \mu_{n}*\gamma_{\sigma}}{\hat \nu_{n}*\gamma_{\sigma}} > \epsilon+cn^{-1/2}\big\}$.  Note that   $(-\infty,c)$ and $[c,\infty)$ are  continuity sets for the Gaussian measure on $\RR$, i.e., with boundary measure zero. Then, with $Z_{n,0}:=n^{1/2}\big(\kl{\hat \mu_{n}*\gamma_{\sigma}}{\hat \nu_{n}*\gamma_{\sigma}}-\kl{ \mu_0*\gamma_{\sigma}}{\nu_0*\gamma_{\sigma}}\big)$, we have
\begin{align}
\limsup_{n \rightarrow \infty} e_{1,n}(T_n)
 &=\limsup_{n \rightarrow \infty}\PP\big(\cE_{n,\epsilon,c}|H_0\big)\leq  \limsup_{n \rightarrow \infty} \PP\big(Z_{n,0} > c|H_0\big) =Q\big(c/v_{2,f_{\mathsf{KL}}}(\mu_0,\nu_0,\sigma)  \big), \label{eq:t1error}
\end{align}
where  $v_{2,f_{\mathsf{KL}}}(\mu_0,\nu_0,\sigma)$ is as given in  \eqref{eq:var-smoothedKL2samp} and $Q$ is the $Q$-function (or complementary error function) given by $Q(x)=(2 \pi)^{-1/2} \int_{x}^{\infty} e^{-z^2/2}dz$. The  inequality above is due to $\kl{ \mu_0*\gamma_{\sigma}}{\nu_0*\gamma_{\sigma}} \leq \epsilon$, while the final equality  uses the Portmanteau theorem and \eqref{eq:KL-GS-twosamp-alt} (the latter applies since $\mu_{0},\nu_0$ have compact supports). Similarly, defining $Z_{n,1}:=n^{1/2}\big(\kl{\hat \mu_{n}*\gamma_{\sigma}}{\hat \nu_{n}*\gamma_{\sigma}}-\kl{ \mu_n*\gamma_{\sigma}}{\nu_n*\gamma_{\sigma}}\big)$ and 
 $\bar \cE_{n,\epsilon,c}:=\big\{\kl{\hat \mu_{n}*\gamma_{\sigma}}{\hat \nu_{n}*\gamma_{\sigma}}$ $ \leq  \epsilon+cn^{-1/2}\big\}$, for the type II error probability we have
 \begin{flalign}
\limsup_{n \rightarrow \infty} e_{2,n}(T_n)
 &=\limsup_{n \rightarrow \infty}\PP\big(\bar \cE_{n,\epsilon,c}|H_{1,n}\big)\notag \\
 &= \limsup_{n \rightarrow \infty} \PP\Big(Z_{n,1}  \leq  -n^{1/2}\big(\kl{ \mu_n*\gamma_{\sigma}}{\nu_n*\gamma_{\sigma}}-\epsilon\big)+c\,|H_{1,n}\Big) \notag \\
  &\leq  \limsup_{n \rightarrow \infty} \PP\big(Z_{n,1} 
\leq c-C|H_{1,n}\big) \notag \\
 &=1-Q\big((c-C)/v_{2,f_{\mathsf{KL}}}(\mu_0,\nu_0,\sigma)\big),  \label{eq:t2error} &&
\end{flalign}
where the last inequality uses $\kl{ \mu_n*\gamma_{\sigma}}{\nu_n*\gamma_{\sigma}} \geq \epsilon+n^{-1/2}C$ and the final equality uses the Portmanteau theorem together with \eqref{eq:localaltlimdist}.
\medskip

To arrive at the result in Proposition \ref{prop:HTperf}, it remains to appropriately upper bound $v_{2,f_{\mathsf{KL}}}(\mu_0,\nu_0,\sigma)$. By using Cauchy-Schwarz inequality and the fact that second moment upper bounds variance, we have
\begin{flalign}
 \mathsf{cov}\big(\varphi_{\sigma}(x-X),\varphi_{\sigma}(y-X)\big)^2
&\leq \EE_{\mu_0}\big[\varphi^2_{\sigma}(x-X)\big]\EE_{\mu_0}\big[\varphi^2_{\sigma}(y-X)\big] \notag \\
  & =(2\pi \sigma^2)^{-2d}\int_{\RR^d}e^{-\frac{\norm{x-z}^2}{\sigma^2}}d\mu_0(z)\int_{\RR^d}e^{-\frac{\norm{y-z}^2}{\sigma^2}} d\mu_0(z)\notag \\
  & \leq (2\pi \sigma^2)^{-2d}e^{\frac{2b^2d}{\sigma^2}}e^{-\frac{(\norm{x}^2+\norm{y}^2)}{2\sigma^2}},\notag &&
\end{flalign}
where the last inequality uses $\norm{x-z}^2 \geq \norm{x}^2/2-\norm{z}^2$.
Hence 
\begin{align}
\Sigma^{(1,1)}_{\mu_0,\nu_0,\sigma}(x,y)=\mathsf{cov}\big(\varphi_{\sigma}(x-X),\varphi_{\sigma}(y-X)\big) \leq (2\pi \sigma^2)^{-d}e^{\frac{b^2d}{\sigma^2}}e^{-\frac{(\norm{x}^2+\norm{y}^2)}{4\sigma^2}}. \notag
\end{align}
 By following similar steps,   for all $1 \leq i,j \leq 2$, we have 
\begin{align}
 \Sigma^{(i,j)}_{\mu_0,\nu_0,\sigma}(x,y)& \leq (2\pi \sigma^2)^{-d}e^{\frac{b^2d}{\sigma^2}}e^{-\frac{(\norm{x}^2+\norm{y}^2)}{4\sigma^2}}.\notag
\end{align}
Also,
\begin{align}
e^{\frac{-b^2d}{2\sigma^2}}e^{\frac{-2b\sqrt{d}\norm{x}}{\sigma^2}} \leq \frac{ \int_{\RR^d} e^{\frac{-2 \norm{x}\norm{z}-\norm{z}^2}{2\sigma^2}} d \mu_0(z)}{ \int_{\RR^d}e^{\frac{2 \norm{x}\norm{ z}-\norm{z}^2}{2\sigma^2}}d \nu_0(z)} &\leq \frac{\mu_0*\varphi_{\sigma}(x)}{\nu_0*\varphi_{\sigma}(x)} \leq \frac{ \int_{\RR^d} e^{\frac{2 \norm{x}\norm{z}-\norm{z}^2}{2\sigma^2}} d \mu_0(z)}{ \int_{\RR^d}e^{\frac{-2 \norm{x}\norm{ z}-\norm{z}^2}{2\sigma^2}}d \nu_0(z)} \leq e^{\frac{b^2d}{2\sigma^2}}e^{\frac{2b\sqrt{d}\norm{x}}{\sigma^2}}. \notag
\end{align}
Combining the above bounds, we have from \eqref{eq:var-smoothedKL2samp} that 
\begin{flalign}
v_{2,f_{\mathsf{KL}}}^2(\mu_0,\nu_0,\sigma)
    & = \sum_{1 \leq i,j \leq 2}\int_{\RR^d}\int_{\RR^d}  \Sigma^{(i,j)}_{\mu_0,\nu_0,\sigma}(x,y) \tilde L_{i,f_{\mathsf{KL}}}(x) \tilde L_{j,f_{\mathsf{KL}}}(y) d x\, d y \notag \\
    &\leq  (2 \pi \sigma^2)^{-d}e^{\frac{b^2d}{\sigma^2}}\int_{\RR^d}\int_{\RR^d}e^{-\frac{(\norm{x}^2+\norm{y}^2)}{4\sigma^2}} \bigg(\frac{(b^2d+4b\sqrt{d}\norm{x}) (b^2d+4b\sqrt{d}\norm{y})}{4\sigma^4}\notag \\
    &\quad + e^{\frac{2b^2d+4b\sqrt{d}(\norm{x}+\norm{y})}{2 \sigma^2}}+ e^{\frac{b^2d+4b\sqrt{d}\norm{x}}{2 \sigma^2}}\frac{(b^2d+4b\sqrt{d}\norm{y}) }{2\sigma^2} \notag \\
    &\quad+ e^{\frac{b^2d+4b\sqrt{d}\norm{y}}{2 \sigma^2}}\frac{(b^2d+4b\sqrt{d}\norm{x}) }{2\sigma^2}\bigg)dx\, dy \notag \\
   & =:c_{b,d,\sigma}.\label{eq:threshold-HT} &&
\end{flalign}
For $0<\tau,\tau' \leq 1 $, setting $c=c_{b,d,\sigma} Q^{-1}(\tau)$, the claim of Proposition \ref{prop:HTperf} follows from \eqref{eq:t1error} and \eqref{eq:t2error} with $C_{b,d,\sigma,\tau,\tau'}=c-c_{b,d,\sigma}Q^{-1}(1-\tau')$.

\subsection{Proof of Proposition \ref{prop:pertconstr}}\label{prop:pertconstr-proof}
\ \\ 
\noindent\textbf{Part $\bm{(i)}$:} Since $\norm{d(\nu_0 \otimes \mu_0)/d\pi_0}_{\infty} <\infty$, there exists $n_0 \in \NN$ such that the RHS of \eqref{eq:pertconst} is non-negative $\pi_0$-a.s., and henceforth, we take $n \geq n_0$. Further, observe that $\pi_n$ specified by \eqref{eq:pertconst} is a valid joint probability measure as
\begin{align}
\int_{\RR^d \times \RR^d} d\pi_n=\int_{\RR^d \times \RR^d}d\pi_0+\bar c n^{-\frac12}(d (\mu_0 \otimes \nu_0)-d (\nu_0 \otimes \mu_0)=1.\notag
\end{align}
Via Taylor's theorem applied to $f_{\mathsf{H}^2}(x)=(\sqrt{x}-1)^2$ around $x=1$ and using $f'(1)=f''(1)=0$, we have
\begin{align}
f_{\mathsf{H}^2}\circ\frac{d\pi_n} {d\pi_0}&=\frac 12 \int_{0}^1 \frac{(1-\tau)}{\left((1-\tau)+\tau \frac{d\pi_n} {d\pi_0}\right)^{\frac 32}} \left(\frac{d\pi_n} {d\pi_0}-1\right)^2 d \tau=\frac 12 \int_{0}^1 \frac{(1-\tau)}{\left((1-\tau)+\tau \frac{d\pi_n} {d\pi_0}\right)^{\frac 32}} h_{\pi_0,\bar c}^2 d \tau.\notag 
\end{align}
Multiplying by $n$ and taking expectation w.r.t. $\pi_0$ leads to
\begin{align}
n \helsq{\pi_n}{\pi_0}&= \int_{\RR^d \times \RR^d} \int_{0}^1 \frac{(1-\tau)}{2\left((1-\tau)+\tau \frac{d\pi_n} {d\pi_0}\right)^{3/2}} h_{\pi_0,\bar c}^2 d \tau d\pi_0. \notag
\end{align}
Note that the integrand in the RHS above is dominated by $0.5(1-\tau)^{-1/2}h_{\pi_0,\bar c}^2$, which is integrable w.r.t. the product measure   $\tau \otimes \pi_0$ under the assumption $\norm{h_{\pi_0,\bar c}}_{2,\pi_0}<\infty$. Taking limit, the dominated convergence theorem and the fact that $d\pi_n/d\pi_0$ converges pointwise to 1, imply that 
\begin{align}
\lim_{n \rightarrow \infty}n \helsq{\pi_n}{\pi_0}&=\frac 14\norm{h_{\pi_0,\bar c}}_{2,\pi_0}^2. \notag
\end{align}
Moreover, it is readily verified that $\int_{\RR^d \times \RR^d} h_{\pi_0,\bar c} d\pi_0=0$. Finally, the marginals of $\pi_n$ satisfy 
$ d\mu_n=d\mu_0+n^{-1/2}\bar c \big(d\mu_0-d\nu_0\big)$ and $ d\nu_n=d\nu_0+n^{-1/2}\bar c\big(d\nu_0-d\mu_0\big)$, which implies
\begin{align} \label{eq:margpert}
\begin{split}\mu_n*\varphi_{\sigma}&=\mu_0*\varphi_{\sigma}+ n^{-1/2}\bar c\big(\mu_0*\varphi_{\sigma}-\nu_0*\varphi_{\sigma}\big),  \\
\nu_n*\varphi_{\sigma}&=\nu_0*\varphi_{\sigma}+n^{-1/2}\bar c\big(\nu_0*\varphi_{\sigma}-\mu_0*\varphi_{\sigma}\big),
\end{split}
\end{align}
and thus
\begin{align}
&n^{1/2}\left(\frac{\mu_n*\varphi_{\sigma}-\mu_0*\varphi_{\sigma}}{\nu_0*\varphi_{\sigma}}\right)= \frac{\bar c(\mu_0*\varphi_{\sigma}-\nu_0*\varphi_{\sigma})}{\nu_0*\varphi_{\sigma}}=\frac{\EE_{\pi_0}\left[h_{\pi_0,\bar c}(X,Y)\varphi_{\sigma}(\cdot-X)\right]}{\nu_0*\varphi_{\sigma}},\notag \\
& n^{1/2}\left(\frac{\nu_n*\varphi_{\sigma}-\nu_0*\varphi_{\sigma}}{\nu_0*\varphi_{\sigma}}\right)= \frac{\bar c(\nu_0*\varphi_{\sigma}-\mu_0*\varphi_{\sigma})}{\nu_0*\varphi_{\sigma}}=\frac{\EE_{\pi_0}\left[h_{\pi_0,\bar c}(X,Y)\varphi_{\sigma}(\cdot-Y)\right]}{\nu_0*\varphi_{\sigma}}. \notag
\end{align}
Hence, \eqref{eq:convmarglocalt} is satisfied (with pointwise equality), verifying all the requirements in Assumption \ref{Assump-localalt}$(i)$ and concluding the proof of Part $(i)$. 

\medskip

\noindent\textbf{Part $\bm{(ii)}$:} Note that $\norm{h_{\pi_0,\bar c}}_{2,\pi_0}<\infty$ as $\norm{d(\nu_0 \otimes \mu_0)/d\pi_0}_{\infty,\pi_0}, \norm{d(\mu_0 \otimes \nu_0)/d\pi_0}_{2,\pi_0} 
<\infty$. Hence all the assumptions in Part $(i)$ are fulfilled thereby implying via the proof above that Assumption \ref{Assump-localalt}$(i)$ is satisfied. So, we only need to verify that  Assumption \ref{Assump-localalt}$(ii)$ holds. By Taylor's theorem applied to $\phi_{\mathsf{KL}}(x,y)=x \log (x/y)$, we have
\begin{align}
&\mu_n*\varphi_{\sigma} \log\left(\frac{\mu_n*\varphi_{\sigma}}{\nu_n*\varphi_{\sigma}}\right)-\mu_0*\varphi_{\sigma} \log\left(\frac{\mu_0*\varphi_{\sigma}}{\nu_0*\varphi_{\sigma}}\right)\notag \\
&\qquad \qquad\qquad=\left(1+\log \left(\frac{\mu_0*\varphi_{\sigma}}{\nu_0*\varphi_{\sigma}}\right)\right) 
 \left(\mu_n*\varphi_{\sigma}-\mu_0*\varphi_{\sigma}\right) -\frac{\mu_0*\varphi_{\sigma}}{\nu_0*\varphi_{\sigma}}\left(\nu_n*\varphi_{\sigma}-\nu_0*\varphi_{\sigma}\right) \notag \\
 &\quad \qquad \qquad\qquad + \left(\mu_n*\varphi_{\sigma}-\mu_0*\varphi_{\sigma}\right)^2\int_{0}^1 \frac{(1-\tau)d\tau}{(1-\tau)\mu_0*\varphi_{\sigma}+\tau \mu_n*\varphi_{\sigma}} \notag \\
 &\quad \qquad \qquad\qquad + \left(\nu_n*\varphi_{\sigma}-\nu_0*\varphi_{\sigma}\right)^2\int_{0}^1 \frac{(1-\tau)\big((1-\tau)\mu_0*\varphi_{\sigma}+\tau \mu_n*\varphi_{\sigma}\big)d\tau}{(1-\tau)\nu_0*\varphi_{\sigma}+\tau \nu_n*\varphi_{\sigma}} \notag \\
 &\quad \qquad \qquad\qquad-2 \left(\mu_n*\varphi_{\sigma}-\mu_0*\varphi_{\sigma}\right)\left(\nu_n*\varphi_{\sigma}-\nu_0*\varphi_{\sigma}\right)\int_{0}^1 \frac{(1-\tau)d\tau}{(1-\tau)\nu_0*\varphi_{\sigma}+\tau \nu_n*\varphi_{\sigma}}. \notag 
\end{align}
Taking integral w.r.t. Lebesgue measure and simplifying using \eqref{eq:margpert}, we obtain
\begin{flalign}
&\kl{\mu_n*\gamma_{\sigma}}{\nu_n*\gamma_{\sigma}} 
-\kl{\mu_0*\gamma_{\sigma}}{\nu_0*\gamma_{\sigma}} \notag \\
&\qquad \qquad\qquad=n^{-\frac{1}{2}}\bar c\big(\kl{\mu_0*\gamma_{\sigma}}{\nu_0*\gamma_{\sigma}} +\kl{\nu_0*\gamma_{\sigma}}{\mu_0*\gamma_{\sigma}}+\chisq{\mu_0*\gamma_{\sigma}}{\nu_0*\gamma_{\sigma}}\big) \notag \\
 &\quad \qquad \qquad\qquad+ \frac{\bar c^2}{n}\int_{\RR^d}\left(\mu_0*\varphi_{\sigma}-\nu_0*\varphi_{\sigma}\right)^2\int_{0}^1 \frac{(1-\tau)}{(1-\tau)\mu_0*\varphi_{\sigma}+\tau \mu_n*\varphi_{\sigma}} d\tau dx\notag \\
 & \quad \qquad \qquad\qquad+\frac{\bar c^2}{n} \int_{\RR^d}\left(\nu_0*\varphi_{\sigma}-\mu_0*\varphi_{\sigma}\right)^2\int_{0}^1 \frac{(1-\tau)\big((1-\tau)\mu_0*\varphi_{\sigma}+\tau \mu_n*\varphi_{\sigma}\big)}{(1-\tau)\nu_0*\varphi_{\sigma}+\tau \nu_n*\varphi_{\sigma}}d\tau dx \notag \\
 &\quad \qquad \qquad\qquad+\frac{2\bar c^2}{n} \int_{\RR^d}\big(\nu_0*\varphi_{\sigma}-\mu_0*\varphi_{\sigma}\big)^2\int_{0}^1 \frac{(1-\tau)}{(1-\tau)\nu_0*\varphi_{\sigma}+\tau \nu_n*\varphi_{\sigma}} d\tau dx\notag\\
 & \qquad \qquad\qquad> n^{-\frac{1}{2}}\bar c\,\big(\kl{\mu_0*\gamma_{\sigma}}{\nu_0*\gamma_{\sigma}} +\kl{\nu_0*\gamma_{\sigma}}{\mu_0*\gamma_{\sigma}}+\chisq{\mu_0*\gamma_{\sigma}}{\nu_0*\gamma_{\sigma}}\big), \notag &&
\end{flalign}
where the final inequality is due to the omitted terms being positive. The proof is concluded by noting that the RHS is larger than $Cn^{-1/2}$ for  $\bar c$ sufficiently large. 

\subsection{Proof of Proposition \ref{prop:HTperfKLDP}}
We consider the test statistic $T_n=\kl{\hat \mu_{n}*\gamma_{\sigma}}{\hat \nu_{n}*\gamma_{\sigma}}$ with critical value $t_n= \epsilon+c n^{-1/2}$  for an appropriately chosen $\sigma$ (small enough) and $c$, and analyze the resulting asymptotic error probabilities of the hypothesis test in \eqref{HT-KL-DP}. The choice of $\sigma$ relies on the stability lemma (see Lemma \ref{lem:smoothKL-stability}) which quantifies the deviation of KL divergence from its smoothed version as a function of $\sigma$.

Recall that $\bar M,\bar \epsilon,\underaccent{\bar}{s}, \bar s $ are known constants such that $M \leq \bar M<\infty$, $\epsilon<\bar \epsilon\leq \tilde \epsilon$, and $0 <\underaccent{\bar}{s} \leq s \leq \bar s \leq 1$.  Let  $\sigma_{ \epsilon,\bar \epsilon,\underaccent{\bar}{s},\bar s,d,\bar M}$ equal $x$ such that \begin{align}
 c_{d,\bar s} \bar M\big(\bar M+1+\log \bar M\big) (x^{\underaccent{\bar}{s}}\vee x^{\bar s})=\bar \epsilon-\epsilon. \label{eq:sigmachoice}
\end{align}
Choose any  $\sigma<\sigma_{ \epsilon,\bar \epsilon,\underaccent{\bar}{s},\bar s,d,\bar M}$, whereby $c_{d,s} M\big(M+1+\log M\big) \sigma^{s} <\bar \epsilon-\epsilon $. Hence, \eqref{eq:KL-stability-ub} and $\kl{ \mu_1}{\nu_{1}} \geq \tilde \epsilon$ imply  $\kl{\mu_1*\gamma_{\sigma}}{\nu_1*\gamma_{\sigma}} >\epsilon$. On the other hand, the data processing inequality along with the fact that $\kl{ \mu_0}{\nu_0} \leq \epsilon$ imply $\kl{\mu_0*\gamma_{\sigma}}{\nu_0*\gamma_{\sigma}} \leq \epsilon$. Then, for the test statistic  $ T_n$, 
we obtain via steps leading to \eqref{eq:t1error} that
\begin{align}
\limsup_{n \rightarrow \infty} e_{1,n}(T_n)
\leq Q\big(c/v_{2,f_{\mathsf{KL}}}(\mu_0,\nu_0,\sigma)  \big).\notag
\end{align}
Setting $Z_{n,1}:=n^{1/2}\big(\kl{\hat \mu_{n}*\gamma_{\sigma}}{\hat \nu_{n}*\gamma_{\sigma}}-\kl{ \mu_1*\gamma_{\sigma}}{\nu_1*\gamma_{\sigma}}\big)$, we further have 
\begin{align}
\limsup_{n \rightarrow \infty} e_{2,n}(T_n)&\leq \limsup_{n \rightarrow \infty} \PP\big(Z_{n,1} \leq -n^{1/2}(\kl{ \mu_1*\gamma_{\sigma}}{\nu_1*\gamma_{\sigma}}-\epsilon)+c|H_1\big) =1-Q(-\infty)=0,  \notag
\end{align}
where the penultimate equality uses the Portmanteau theorem applied to $Z_{n,1}$ converging weakly to a Gaussian random variable (via \eqref{eq:KL-GS-twosamp-alt} with $v_{2,f_{\mathsf{KL}}}^2(\mu_1,\nu_1,\sigma) \leq c_{b,d,\sigma}$ similar to \eqref{eq:threshold-HT}),  and $c-n^{1/2}(\kl{ \mu_1*\gamma_{\sigma}}{\nu_1*\gamma_{\sigma}}$ $-\epsilon)$ diverging to $-\infty$ due to $\kl{ \mu_1*\gamma_{\sigma}}{\nu_1*\gamma_{\sigma}}$ $>\epsilon$. Choosing $c=c_{b,d,\sigma} Q^{-1}(\tau)$ (see \eqref{eq:threshold-HT}) completes the proof.

\subsection{Proof of Lemma \ref{lem:smoothKL-stability}}
For $\tau \in [0,1]$, let $z_{\mu,\tau}(x):=(1-\tau)p_{\mu}(x)+\tau \mu*\varphi_{\sigma}(x) $ and $z_{\nu,\tau}(x):=(1-\tau)p_{\nu}(x)+\tau\nu*\varphi_{\sigma}(x)$. By Taylor's theorem, we have 
\begin{align}
\mu*\varphi_{\sigma}(x) \log \left(\frac{\mu*\varphi_{\sigma}(x)}{\nu*\varphi_{\sigma}(x)}\right)&=p_{\mu}(x) \log \left(\frac{p_{\mu}(x)}{p_{\nu}(x)}\right)+ \big(\mu*\varphi_{\sigma}(x)-p_{\mu}(x)\big)\int_{0}^1\left(1+\log \frac{z_{\mu,\tau}(x)}{z_{\nu,\tau}(x)}\right)d\tau \notag \\
&\quad \qquad  -\big(\nu*\varphi_{\sigma}(x)-p_{\nu}(x)\big)\int_{0}^1 \frac{z_{\mu,\tau}(x)}{z_{\nu,\tau}(x)}d\tau.\notag 
\end{align}
Note that $1/M \leq p_{\mu}(x)/p_{\nu}(x) \leq M$ by assumption, which implies $1/M \leq \mu*\varphi_{\sigma}(x)/\nu*\varphi_{\sigma}(x) \leq M$, and hence  
\[\frac 1M \leq \frac{p_{\mu}(x)}{p_{\nu}(x)} \wedge \frac{\mu*\varphi_{\sigma}(x)}{\nu*\varphi_{\sigma}(x)} \leq \frac{z_{\mu,\tau}(x)}{z_{\nu,\tau}(x)} \leq \frac{p_{\mu}(x)}{p_{\nu}(x)} \vee \frac{\mu*\varphi_{\sigma}(x)}{\nu*\varphi_{\sigma}(x)}\leq M,\quad \forall \,x \in \RR^d.\] 
Integrating w.r.t. Lebesgue measure in the above equation, we then obtain (note that $M \geq 1$)
\begin{align}
\abs{\kl{\mu}{\nu}-\kl{\mu*\gamma_{\sigma}}{\nu*\gamma_{\sigma}}}  &\leq \left(1+\log M\right)\int_{\RR^d}\abs{\mu*\varphi_{\sigma}(x)-p_{\mu}(x)} dx +M \int_{\RR^d}\abs{\nu*\varphi_{\sigma}(x)-p_{\nu}(x)} dx.\label{eq:bnddiffKL}
\end{align}
The first integral is bounded as follows:
\begin{flalign}
   \int_{\RR^d} \abs{p_{\mu}(x)-\mu*\varphi_{\sigma}(x)}d x&=  \int_{\RR^d} \abs{p_{\mu}(x)-\sigma^{-d}\int_{\RR^d} p_{\mu}(y) \varphi_{1}\left((x-y)\sigma^{-1}\right)d y }d x \notag\\
&=\int_{\RR^d} \abs{p_{\mu}(x)-\int_{\RR^d}  p_{\mu}(x-\sigma z) \varphi_{1}(z) d z}d x \notag\\
&=\int_{\RR^d} \abs{\int_{\RR^d}\big( p_{\mu}(x)\varphi_1(z)- p_{\mu}(x-\sigma z) \varphi_1(z)\big) d z } d x \notag \\
&  \stackrel{(a)}{\leq} \int_{\RR^d} \left(\int_{\RR^d}\abs{ p_{\mu}(x+\sigma z)- p_{\mu}(x) }d x \right) \varphi_1(z) d z \notag \\
& \stackrel{(b)}{\leq}   \int_{\RR^d}  \kappa_{1,1}\left(p_{\mu},\sigma \norm{z}\right)\varphi_1(z) d z \notag \\
& \stackrel{(c)}{\leq}  M  \sigma^s \int_{\RR^d} \norm{z}^s \varphi_1(z) d z, \notag &&
\end{flalign}
where $(a)$ uses Fubini's theorem, $(b)$ is by the definition of the modulus of smoothness in Definition \ref{def:Lipschitzclass}, while $(c)$ is because $p_{\mu} \in \mathrm{Lip}_{s,1}(M,\cX)$. Following similar steps with $\mu$ replaced by $\nu$, the same bound holds for the second integral term in \eqref{eq:bnddiffKL}, which concludes the proof.

\section{Concluding Remarks} \label{Sec:conclusion}
This paper provided a unified methodology for deriving one- and two-sample limit distributions for $f$-divergences, under both the null (i.e., when the population distributions are equal) and the alternative. We focused on four prominent examples, namely, KL divergence, $\chi^2$ divergence, squared Hellinger distance, and total variation distance. The general limit theorems were stated under technical conditions on the distributions which guarantee Hadamard differentiability of the relevant $f$-divergence functional. The framework allows arbitrary estimators of the population measures and accounts for dependent data, which renders it rather flexible and broadly applicable. We instantiate the general limit distribution theory to the setting of Gaussian-smoothed $f$-divergence, showing that the limit (except for TV distance) is Gaussian under the alternative, or can be represented as a weighted sum of i.i.d. $\chi^2$ random variables under the null. In contrast,  the limit distribution for smoothed TV distance in both the above scenarios is non-Gaussian. We also proposed a hypothesis testing pipeline for auditing DP and characterized its asymptotic level and power by utilizing our distributional limits.

While this work focuses on $f$-divergences, a limit distribution theory for other classes of divergences, such as integral probability metrics and Bregman divergences, are largely unexplored and worth pursuing. We believe that our approach based on the functional delta method and Hadamard directional differentiability can be extended to cover those cases as well. Another appealing research avenue is to explore asymptotic distributions of more sophisticated (non plug-in) estimators of $f$-divergences. This includes approaches based on bias correction mechanisms (e.g., \cite{Yanjun-2020,kandasamy2015nonparametric,berrett2019-efficientfunctional}) or variational methods (e.g., \cite{belghazi2018,sreekumar2021neural}). While such estimators are known to achieve better minimax rates over appropriate classes of distributions, the main challenge here is to establish regularity conditions under which the Hadamard derivative of the relevant functional exists. Towards that end, recent results on Hadamard differentiability of supremum-type functionals \cite{CCR-2020} may be useful for deriving limit theorems of variational estimators. Yet another intriguing question pertains to establishing Berry-Esseen type results, which would reveal the convergence rate of the empirical $f$-divergence to its distributional limit. However, deriving such results is highly non-trivial and would require different tools and technical approaches than the empirical process theory-based techniques employed in this work.

\begin{appendices}
      
\section{Proofs of  Part $(iii)$ of Lemma \ref{Lem:extfuncdelta}, Theorem \ref{thm:extcontmapthm}, and the claim in Remark \ref{rem:relaxeddeltmethod}}

\subsubsection{Proof of Part $(iii)$ of Lemma \ref{Lem:extfuncdelta}} \label{Lem:extfuncdelta-proof}
 We will use  similar steps as in  \cite[Theorem 1]{Romisch-2004} and \cite[Theorem 2]{Romisch-2004} with  the extended continuous mapping theorem replaced by Theorem \ref{thm:extcontmapthm}. We highlight the relevant steps for the claim that Part $(i)$ of Lemma \ref{Lem:extfuncdelta} holds. The proof that  Part $(ii)$ holds is similar and omitted. 

Recall that since $\Theta$ is convex, the tangent cone is given by (see \cite{Romisch-2004})
\begin{align}
     \mathfrak{T}_{\Theta}(\theta)&:=\bigg\{h \in \mathfrak{D}:h=\lim_{n \rightarrow \infty} \frac{\theta_n-\theta}{t_n} \mbox{ for some }\theta_n \rightarrow \theta \mbox{ with } \theta_n \in \Theta \mbox{ and }t_n \downarrow 0^+\bigg\}\notag \\
     &=\mathrm{cl}\big(\big\{(\tilde \theta-\theta)/t:\tilde \theta \in \Theta, ~t>0\big\}\big).\notag
\end{align}
Since $ \mathfrak{T}_{\Theta}(\theta)$ is closed, its boundary is a Borel set. Set $\mathfrak{D}_n=\{h \in \mathfrak D: \theta+(h/r_n) \in \mathfrak D_{\Phi}\}$,  $\mathfrak{D}_n^{\rightarrow}=\{h \in \mathfrak D: \theta+(h/r_n) \in \Theta\}$, $\mathfrak{D}_g=\mathfrak{D}_{\infty} =\mathfrak{T}_{\Theta}(\theta)$, $g=\Phi_{\theta}' $,  $g_n(h)=r_n \big(\Phi(\theta+(h/r_n))-\Phi(\theta) \big)$, $H_n=r_n(Z_n-\theta)$ and $H=Z$. Then, the claim will follow from Theorem \ref{thm:extcontmapthm}, provided that $\PP\big(H \in \mathsf{int}\big(\mathfrak{T}_{\Theta}(\theta)\big)\big)=1$ and $\mathfrak{D}_* \subseteq \mathfrak{D} \setminus \mathsf{int}\big(\mathfrak{T}_{\Theta}(\theta)\big)$, since these imply $\PP\left(H \in \mathfrak{D}_{\infty}\right)=1$ and $\PP\left(H \in \mathfrak{D}_*\right)=0$.

The first condition above follows from
\begin{align}
    \PP\left(H \in \mathsf{int}\big(\mathfrak{T}_{\Theta}(\theta)\big)\right) 
    \stackrel{(a)}{=}1-\PP\left(H \notin \mathfrak{T}_{\Theta}(\theta)\right)
    \stackrel{(b)}{=}1-\lim_{n \rightarrow \infty} \PP\left(H_n \notin \mathfrak{T}_{\Theta}(\theta)\right)
      \stackrel{(c)}{\geq} 1-\lim_{n \rightarrow \infty} \PP^*\left(Z_n \notin \Theta\right)=1, \notag
\end{align}
where $(a)$ is because $ \PP\big(H \in \partial \mathfrak{T}_{\Theta}(\theta)\big)=0$, $(b)$ is since $H_n \trightarrow{w} H$ and Portmanteau theorem, $(c)$ is due to $Z_n \in \Theta$ implies $H_n \in \mathfrak{T}_{\Theta}(\theta)$,  and the final equality follows from $\PP^*\left(Z_n \notin \Theta\right) \rightarrow 0$. To show the second condition, suppose $h \in \mathsf{int}\big(\mathfrak{T}_{\Theta}(\theta)\big) $.  We claim that there exists $m \in \NN$ such that $h \in \mathsf{int}\big(\mathfrak{D}_n^{\rightarrow}\big)$ for all $n \geq m$ which  implies that $h \notin \cup_{n \geq m} \overline{\mathfrak{D} \setminus \mathfrak{D}_n^{\rightarrow}} \supseteq \mathfrak{D}_*$ as desired. To see this, note that $h=(\tilde \theta-\theta)/t$ for some $t>0$ and $\tilde \theta \in \mathsf{int}\big(\Theta\big)$. Let $\hat{\theta}_{\alpha}:=\alpha \tilde \theta+ (1-\alpha) \theta $ for  $\alpha \in [0,1]$. We have  $\hat{\theta}_{\alpha} \in \Theta$ since $\Theta$ is convex. Then, by definition of $\mathfrak{D}_n^{\rightarrow}$,  $\alpha_n r_n (\tilde \theta-\theta)=r_n(\hat{\theta}_{\alpha_n}-\theta) \in \mathfrak{D}_n^{\rightarrow}$. Hence,  for $n \geq m=\min\{n:r_n t \geq  1\}$,   we have $h \in \mathsf{int}\big(\mathfrak{D}_n^{\rightarrow}\big)$ by choosing $\alpha_n=1/(tr_n) $. Hence, all the conditions in Theorem \ref{thm:extcontmapthm}  are satisfied under the assumptions in Part $(iii)$ of Lemma \ref{Lem:extfuncdelta}, and consequently  $r_n\big(\Phi(Z_n)-\Phi(\theta)\big) \trightarrow{d} \Phi_{\theta}'(Z)$ follows. The proof of 
    $r_n\big(\Phi(Z_n)-\Phi(\theta)\big)=\phi'_\theta\big(r_n(Z_n-\theta)\big) +o_\PP(1)$ and Part $(ii)$ follows similarly to \cite{Romisch-2004} with the relevant steps adapted as above, and hence omitted. 

\subsubsection{Proof of Theorem \ref{thm:extcontmapthm}} \label{thm:extcontmapthm-proof}

Since the proof follows closely to that of \cite[Theorem 1.11.1]{AVDV-book}, we only  highlight the differences. Following similar steps as therein, we have that $g$ restricted to $\mathfrak{D}_g \cap \mathfrak{D}_{\infty}$ is continuous, and hence $g(H)$ is Borel measurable.

Consider a closed set $\cC \subseteq \mathfrak{E}$. Then, we have
\begin{align}
\cap_{n=1}^{\infty} \overline{\cup_{m=n}^{\infty} g_m^{-1}(\cC)} \subseteq g^{-1}(\cC) \cup \big(\mathfrak{D} \setminus(\mathfrak{D}_{\infty} \cap \mathfrak{D}_g) \big)\cup \mathfrak{D}_* , \label{eq:setinclu}
\end{align}
where $\bar \cC$ denotes the closure of the set $\cC$.
Indeed, if $h$ is in the set on the left,  then one of the following condition must hold: $(i)$ there exists a divergent sequence $(n_{k})_{k \in \NN}$ such that $h_{n_k} \in g_{n_k}^{-1}(\cC) \cap \mathfrak{D}_{\rightarrow}$ and $d(h_{n_k},h) \rightarrow 0$, where $d(\cdot,\cdot)$ denotes the metric of $\mathfrak D$; or $(ii)$ there exists a divergent sequence $(n_{k})_{k \in \NN}$ such that $h_{n_k} \in g_{n_k}^{-1}(\cC) \cap (\mathfrak{D} \setminus \mathfrak{D}_n^{\rightarrow})$ and $d(h_{n_k},h) \rightarrow 0$. If $(i)$ holds, then by the assumption in the theorem, either $g_{n_k}(h_{n_k}) \rightarrow g(h)$ (if $h \in \mathfrak{D}_{\infty} \cap \mathfrak{D}_g$) or $h \in \mathfrak{D} \setminus(\mathfrak{D}_{\infty} \cap \mathfrak{D}_g)$. If $g_{n_k}(h_{n_k}) \rightarrow g(h)$, then $g(h) \in \cC$ since $g_{n_k}(h_{n_k}) \in \cC$ and $\cC$ is a closed set. On the other hand,  if $(ii)$ holds, then this implies that $h \in \cap_{m=1}^{\infty} \overline{\cup_{k=m}^{\infty} g_{n_k}^{-1}(\cC) \cap (\mathfrak{D} \setminus \mathfrak{D}_{n_k}^{\rightarrow})} \subseteq \cap_{m=1}^{\infty} \overline{\cup_{k=m}^{\infty}  (\mathfrak{D} \setminus \mathfrak{D}_{n_k}^{\rightarrow})} \subseteq \mathfrak{D}_*$.  Hence, \eqref{eq:setinclu} holds. Then,  we have for any fixed $k$ that
\begin{align}
   \limsup_{n \rightarrow \infty} \PP\left(g_n(H_n) \in \cC \right) \leq \limsup_{n \rightarrow \infty} \PP\left(H_n \in \overline{ \cup_{m=k}^{\infty} g_m^{-1}(\cC)} \right) \leq \PP\left(H \in \overline{ \cup_{m=k}^{\infty} g_m^{-1}(\cC)} \right), 
\end{align}
where the final inequality follows from $H_n \trightarrow{w} H $ and Portmanteau theorem \cite[Theorem 1.3.4]{AVDV-book}. Taking limit $k \rightarrow \infty$, we have
\begin{align}
    \limsup_{n \rightarrow \infty} \PP\left(g_n(H_n) \in \cC \right) &\leq \PP\left(H \in \cap_{k=1}^{\infty}\overline{ \cup_{m=k}^{\infty} g_m^{-1}(\cC)} \right) \notag \\
    &\stackrel{(a)}{\leq} \PP(H \in g^{-1}(\cC) )+\PP\big(H \in \mathfrak{D} \setminus(\mathfrak{D}_{\infty} \cap \mathfrak{D}_g) \big)+\PP\big(H \in \mathfrak{D}_*\big) \notag \\
   &\stackrel{(b)}{=} \PP(g(H) \in \cC ), \label{eq:weakconvextc}
\end{align}
where $(a)$ is due to \eqref{eq:setinclu}, and $(b)$ is because   $\PP(H \in \mathfrak{D}_*)=0$ by assumption 
    and 
    \begin{align}
        \PP\big(H \in \mathfrak{D} \setminus(\mathfrak{D}_{\infty} \cap \mathfrak{D}_g) \big)=1-\PP\big(H \in \mathfrak{D}_{\infty} \cap \mathfrak{D}_g \big)=1-\PP\big(H \in \mathfrak{D}_{\infty} \big)=0. \notag
    \end{align} 
Since $\cC$ is an arbitrary closed set, $g_n(H_n) \trightarrow{w} g(H)$ again follows from \eqref{eq:weakconvextc} by Portmanteau theorem.

    \subsubsection{Proof of the claim in Remark \ref{rem:relaxeddeltmethod}} \label{Sec:remarkproof}
The proof follows by using Part $(iii)$ of Lemma \ref{Lem:extfuncdelta}  in the proof of Part $(iii)$ and $(iv)$ of Theorem \ref{Thm:KLdiv-limdist}. 
Recall that $\cQ_n:=\{\omega \in \Omega: p_{\mu_n}(\omega,\cdot)/p_{\nu_n}(\omega,\cdot) \leq q(\cdot) \}$, and let $\tilde \cQ_n:= \{\omega \in \Omega:Z_n(\omega,\cdot)=(p_{\mu_n}(\omega,\cdot)-p_{\mu},p_{\nu_n}(\omega,\cdot)-1) \in \Theta\}$. Note that under the conditions in Part $(iii)$ and $(iv)$ of Theorem \ref{Thm:KLdiv-limdist} sans the requirement  $p_{\mu_n}/p_{\nu_n} \leq q$ a.s., $\Omega \setminus \tilde \cQ_n = (\Omega \setminus  \cQ_n) \cup \cN_0 $ for some null set $\cN_0$ of outer probability zero. Hence, $\PP^*(Z_n \notin \Theta) =\PP^*(\Omega \setminus \cQ_n) \rightarrow 0$ under the assumption in Remark \ref{rem:relaxeddeltmethod}. Also, $Z=(B_1,B_2)$ being continuous implies $\PP\big(Z \in \partial \mathfrak{T}_{\Theta}(\theta)\big)=0$. Hence, all the conditions in Part $(iii)$ of Lemma \ref{Lem:extfuncdelta} are satisfied, and the result  follows.

\section{Proof of Lemma \ref{Lem:contmapLp}} \label{Lem:contmapLp-proof}
Since $g_n \rightarrow g$ in $L^{r'}(\rho)$, for every $\epsilon>0$,  there exists $n_0(\epsilon) \in \NN$ such that   $\norm{g_n}_{r',\rho}\leq \norm{g}_{r',\rho}+\epsilon$ for all $n \geq n_0(\epsilon)$. Thus, for  every $n \geq n_0(\epsilon)$, we have by Minkowski's and H\"{o}lder's inequality that
\begin{align}
    \norm{f_ng_n-fg}_{1,\rho} &\leq \norm{f_ng_n-fg_n}_{1,\rho}+\norm{fg_n-fg}_{1,\rho} \notag \\
    &\leq \norm{f_n-f}_{r,\rho}\norm{g_n}_{r',\rho}+\norm{f}_{r,\rho}\norm{g_n-g}_{r',\rho}. \notag \\
     &\leq \norm{f_n-f}_{r,\rho}(\norm{g}_{r',\rho}+\epsilon)+\norm{f}_{r,\rho}\norm{g_n-g}_{r',\rho}.\notag
\end{align}
The RHS converges to zero since $f_n \rightarrow f$ in $L^r(\rho)$ and $g_n \rightarrow g$ in $L^{r'}(\rho)$, thus proving Part $(i)$. The claim in  Part $(ii)$ follows from \cite[Example 1.4.7 (Slutsky's lemma)]{AVDV-book} by the separability of $L^r(\rho)$ spaces for $1 \leq r <\infty$, and the continuous mapping theorem applied to the continuous map $(f,g) \mapsto fg$ from  $L^{r}(\rho) \times L^{r'}(\rho) \rightarrow L^1(\rho)$ by Part $(i)$. 

\section{Proof of Lemma \ref{Lem:assumpfdiv}}
\label{Lem:assumpfdiv-proof}
 
\noindent\textbf{Part $\bm{(i)}$:} 
The convexity of  $\phi_{\mathsf{KL}}$, $\phi_{\chi^2}$ and $\phi_{\mathsf{H}^2}$ follows since these are perspectives\footnote{The perspective of a convex function $f:\RR \rightarrow \RR$  is the function $f_p:\RR \times \RR_+ \rightarrow \RR$ given by $f_p(x,t)=t f(x/t)$.} of the convex functions $f_{\mathsf{KL}}$,  $f_{\chi^2}$ and $f_{\mathsf{H}^2}$, respectively (see \cite{Boyd}). 
Computing the partial derivatives yields
\begin{subequations}
  \begin{align}
        \begin{split} \label{eq:KL2dervs}
      &  D^{(1,0)}\phi_{\mathsf{KL}}(x,y)=1+\log \left(\frac xy\right), \quad  D^{(0,1)}\phi_{\mathsf{KL}}(x,y)=-\frac{x}{y}, \\ 
  &  D^{(1,1)}\phi_{\mathsf{KL}}(x,y)=-\frac{1}{y}, \quad D^{(2,0)}\phi_{\mathsf{KL}}(x,y)=\frac 1x, \quad D^{(0,2)}\phi_{\mathsf{KL}}(x,y)=\frac{x}{y^2}, 
        \end{split}\\
                 \begin{split} \label{eq:chisq2dervs}
     & D^{(1,0)}\phi_{\chi^2}(x,y)=\frac{2(x-y)}{y},\quad  D^{(0,1)}\phi_{\chi^2}(x,y)=1-\frac{x^2}{y^2},  \\
   & D^{(1,1)}\phi_{\chi^2}(x,y)=-\frac{2x}{y^2}, \quad D^{(2,0)}\phi_{\chi^2}(x,y)=\frac 2y, \quad D^{(0,2)}\phi_{\chi^2}(x,y)=\frac{2x^2}{y^3},
            \end{split}\\
            \begin{split}\label{eq:helsq2dervs}
     & D^{(1,0)}\phi_{\mathsf{H}^2}(x,y)=1-y^{\frac 12}x^{-\frac 12},\quad  D^{(0,1)}\phi_{\mathsf{H}^2}(x,y)=1-x^{\frac 12}y^{-\frac 12}, \\
  & D^{(1,1)}\phi_{\mathsf{H}^2}(x,y)=-\frac 12x^{-\frac 12}y^{-\frac 12}, \quad D^{(2,0)}\phi_{\mathsf{H}^2}(x,y)=\frac 12x^{-\frac 32}y^{\frac 12}, \quad   D^{(0,2)}\phi_{\mathsf{H}^2}(x,y)=\frac 12y^{-\frac 32}x^{\frac 12}. 
      \end{split}
        \end{align}
\end{subequations}
    
These partial derivatives obviously satisfy the continuity conditions in  Assumption~\ref{Assump1} and those mentioned in the statement of Lemma \ref{Lem:assumpfdiv} (i), thus completing the proof of Part $(i)$.

\medskip

 \noindent\textbf{Part $\bm{(ii)}$:}
For $(ii)(a)$, first note that for $g_1,g_2 \geq 0$ and $g_1^{\star},g_2^{\star}>0$, we have
\begin{align}
\abs{D^{(2,0)}\phi_{\mathsf{KL}}\circ\big((1-\tau)(g_1^{\star},g_2^{\star})+\tau (g_1,g_2)\big)}(1-\tau)& =\frac{(1-\tau)}{(1-\tau)g_1^{\star}+\tau g_1} \leq \frac{1}{g_1^{\star}}. \notag 
\end{align}

The case of $D^{(1,1)}\phi_{\mathsf{KL}}$ is similar, while for $D^{(0,2)}\phi_{\mathsf{KL}}$, we have for $(g_1,g_2) \in \bar\Theta(q) $ that
\begin{align}
\abs{D^{(0,2)}\phi_{\mathsf{KL}}\circ\big((1-\tau)(g_1^{\star},g_2^{\star})+\tau (g_1,g_2)\big)}(1-\tau)& = \frac{(1-\tau)\big((1-\tau)g_1^{\star}+\tau g_1\big)}{\big((1-\tau)g_2^{\star}+\tau g_2\big)^2} \leq \frac{g_1^{\star}}{g_2^{\star2}}+\frac{g_1}{2g_2^{\star}g_2}\leq\frac{g_1^{\star}}{g_2^{\star2}}+\frac{q}{g_2^{\star}}.\notag
\end{align}

\medskip

For $(ii)(b)$, the case of $D^{(1,1)}\phi_{\chi^2}$ and $D^{(2,0)}\phi_{\chi^2}$ again can be shown as above, while
\begin{align}
\abs{D^{(0,2)}\phi_{\chi^2}\circ\big((1-\tau)(g_1^{\star},g_2^{\star})+\tau (g_1,g_2)\big)}(1-\tau)& = \frac{2(1-\tau)\big((1-\tau)g_1^{\star}+\tau g_1\big)^2}{\big((1-\tau)g_2^{\star}+\tau g_2\big)^3} \lesssim \frac{g_1^{\star2}}{g_2^{\star3}}+\frac{g_1^2}{g_2^{\star}g_2^2}\leq\frac{g_1^{\star2}}{g_2^{\star3}}+\frac{q^2}{g_2^{\star}}.\notag
\end{align}

Finally, for $(ii)(c)$, let $(g_1,g_2),(g_1^{\star},g_2^{\star}) \in \check \Theta(q_1,q_2)$. We have 
\begin{align}
\abs{D^{(1,1)}\phi_{\mathsf{H}^2}\mspace{-2 mu}\circ\mspace{-2 mu}\big((1-\tau)(g_1^{\star},g_2^{\star})\mspace{-2 mu}+\mspace{-2 mu}\tau (g_1,g_2)\big)}\mspace{-2 mu}\mspace{-2 mu}(1-\tau)\mspace{-2 mu}&\lesssim \mspace{-2 mu}\frac{(1-\tau)}{\big((1-\tau)g_2^{\star}+\tau g_2\big)^{\frac 12}\mspace{-2 mu}\big((1-\tau)g_1^{\star}+\tau g_1\big)^{\frac 12}}  \leq  \frac{1}{\big(g_1^{\star}g_2^{\star}\big)^{\frac 12}},\notag \\
\abs{D^{(2,0)}\phi_{\mathsf{H}^2}\circ\big((1-\tau)(g_1^{\star},g_2^{\star})+\tau (g_1,g_2)\big)}(1-\tau)&\lesssim \frac{(1-\tau)\big((1-\tau)g_2^{\star}+\tau g_2\big)^{\frac 12}}{\big((1-\tau)g_1^{\star}+\tau g_1\big)^{\frac 32}} \notag \\
& \leq \frac{g_2^{\star\frac 12}}{g_1^{\star\frac 32}}+\frac{g_2^{\frac 12}\tau^{1/2}}{g_1^{\star\frac 32}(1-\tau)^{\frac 12}}\leq \frac{g_2^{\star\frac 12}}{g_1^{\star\frac 32}}+\frac{q_2^{\frac 12}\tau^{\frac 12}}{g_1^{\star\frac 32}(1-\tau)^{\frac 12}}, \notag
\end{align}
where  $\tau^{1/2}(1-\tau)^{-\frac 12} \in L^1([0,1],\lambda)$. 
The bound for  $D^{(0,2)}\phi_{\mathsf{H}^2}$ can be shown similarly, thus completing the proof.

\section{Proof of Lemma \ref{lem:bickel98}} \label{lem:bickel98-proof}
    The claim follows from the proof of \cite[Appendix A.9, Corollary 2]{Bickel-98} (see Page 500-501) by noting that the condition $(a)$, $(b)$ and $(c)$ given therein which suffices for the proof to hold are satisfied under the conditions  here. Specifically, condition $(a)$ is equivalent to $\lim_{n \rightarrow \infty}n \helsq{\pi_n}{\pi_0} = \norm{h/2}^2_{2,\pi_0}$ (with $\delta=h/2$ and $\mu=\pi_0$ in the notation therein) which in turn also implies $(b)$, while $(c)$ is equivalent to  $\int_{\RR^d \times \RR^d}h\, d\pi_0=0$.

 \section{Limit distribution for Gaussian-smoothed $\chi^2$ divergence, $\mathsf{H}^2$ distance and TV distance} \label{App:GSfdiv}

\subsection{$\chi^2$ divergence} \label{App:GSfdiv-chisq} We consider distributional limits for the Gaussian-smoothed $\chi^2$ divergence.
\begin{prop}[Limit distribution for Gaussian-smoothed $\chi^2$ divergence] \label{Prop:GS-chisq-limdist} The following hold:
\begin{enumerate}[(i)]
    \item  (One-sample null) If \eqref{ASSUM:KL_null} is satisfied,
then there exists a version of $G_{\mu,\sigma}$ such that $G_{\mu,\sigma}/\sqrt{\mu*\varphi_{\sigma}}$ is $L^2(\RR^d)$-valued, and 
\begin{align}
  n \chisq{\hat \mu_n*\gamma_\sigma}{\mu*\gamma_\sigma} \trightarrow{d}   \int_{\RR^d} \frac{G_{\mu,\sigma}^2(x)}{\mu*\varphi_{\sigma}(x)}d x, \label{eq:chisq-GS-onesamp-null}
\end{align}
where the limit can be represented as a weighted sum of i.i.d. $\chi^2$ random variables with 1 degree of freedom. In particular, \eqref{eq:chisq-GS-onesamp-null} holds for $\beta$-sub-Gaussian $\mu$  with $\beta<\sigma$. Conversely,  if \eqref{ASSUM:KL_null} is violated, then $\EE\big[\chisq{\hat \mu_n*\gamma_\sigma}{\mu*\gamma_\sigma}\big]=\infty$ for every $n \in \NN$.
\medskip
\item  (One-sample alternative)  If \eqref{ASSUM:KL_alt} holds and $\chisq{\mu*\varphi_{\sigma}}{\nu*\varphi_{\sigma}}<\infty$,  then 
\begin{align}
  n^{\frac 12}\big(\chisq{\hat \mu_n*\gamma_\sigma}{\nu*\gamma_{\sigma}}-\chisq{\mu*\gamma_{\sigma}}{\nu*\gamma_{\sigma}}\big) \trightarrow{d} N\big(0,v_{1,f_{\chi^2}}^2(\mu,\nu,\sigma)\big), \label{eq:chisq-GS-onesamp-alt}
\end{align}
where $v_{1,f_{\chi^2}}^2(\mu,\nu,\sigma)$ is given in \eqref{eq:chisqvaronesamp}.
In particular,  \eqref{eq:chisq-GS-onesamp-alt} holds for $\beta$-sub-Gaussian $\mu$  with $\beta<\sigma$ such that  $\mu \ll \nu$ and $\norm{d \mu/d \nu}_{\infty}  <\infty$.
\medskip
\item (Two-sample null) If 
$\mu$ has compact support, then   there exists a version of $G_{\mu,\sigma}$,  $\tilde G_{\mu,\sigma}$ such that $G_{\mu,\sigma}/\sqrt{\mu*\varphi_{\sigma}}$ and $\tilde G_{\mu,\sigma}/\sqrt{\mu*\varphi_{\sigma}}$ are $L^2(\RR^d)$-valued, and
\begin{align}
  n\, \chisq{\hat \mu_n*\gamma_\sigma}{\hat \nu_n*\gamma_\sigma} \trightarrow{d}   \int_{\RR^d} \frac{\big(G_{\mu,\sigma}(x)-\tilde G_{\mu,\sigma}(x)\big)^2}{\mu*\varphi_{\sigma}(x)}d x, \label{eq:chisq-GS-twosamp-null}
\end{align}
where the limit can be represented as a weighted sum of i.i.d.  $\chi^2$ random variables with 1 degree of freedom. 
\medskip
 \item (Two-sample alternative) If  $\mu,\nu$ have compact supports, then 
 \begin{align}
   & n^{\frac 12}\big(\chisq{\hat \mu_n*\gamma_\sigma}{\hat \nu_n*\gamma_{\sigma}}-\chisq{\mu*\gamma_{\sigma}}{\nu*\gamma_{\sigma}}\big)\trightarrow{d} N\big(0,v_{2,f_{\chi^2}}^2(\mu,\nu,\sigma)\big), \label{eq:chisq-GS-twosamp-alt}
\end{align}
where $v_{2,f_{\chi^2}}^2(\mu,\nu,\sigma)$  is given in \eqref{eq:var-smoothedchisq2samp}.
\end{enumerate}
\end{prop}
\begin{proof}
    Since the general idea of the proof is similar to that of Proposition \ref{Prop:GS-KL-limdist}, we only provide a sketch of the proof, while highlighting  the differences.
\medskip

\noindent\textbf{Part $\bm{(i)}$:} We apply Theorem \ref{Thm:chisqdiv-limdist}(i) 
with $\rho=\mu*\gamma_{\sigma}$. We have
   \begin{flalign}
    \EE \big[ \chisq{\hat {\mu}_n*\gamma_{\sigma}}{\mu*\gamma_{\sigma}}\big] & \stackrel{(a)}{=}      \frac{1}{n^2} \int_{\RR^d} \EE \left[     \frac{\Big( \sum_{i=1}^n \big(\varphi_{\sigma}(x-X_i)-\mu*\varphi_{\sigma}(x)\big)\Big)^2}{\mu*\varphi_{\sigma}(x)} \right]\,dx\notag \\
    &\stackrel{(b)}{=}      \frac{1}{n} \int_{\RR^d} \EE \left[     \frac{\big(  \varphi_{\sigma}(x-X)-\mu*\varphi_{\sigma}(x)\big)^2}{\mu*\varphi_{\sigma}(x)} \right]\,dx\notag \\
        &=     \frac{1}{n} \int_{\RR^d}\frac{\mathsf{Var}_{\mu}\big(\varphi_{\sigma}(x-\cdot)\big)}{\mu*\varphi_{\sigma}(x)}\,dx, \label{eq:chisqexpbnd} &&
 \end{flalign} 
 where $(a)$ is by Fubini's theorem and  $(b)$ uses that $X_1,\ldots,X_n$ are i.i.d. Consequently, under \eqref{ASSUM:KL_null}, we obtain $\chisq{\hat {\mu}_n*\gamma_{\sigma}}{\mu*\gamma_{\sigma}}<\infty$ a.s. The rest of the proof leading to \eqref{eq:chisq-GS-onesamp-null} via Theorem \ref{Thm:chisqdiv-limdist}(i) is similar to that of Proposition \ref{Prop:GS-KL-limdist}$(i)$ and proceeds by showing that the conditions in Theorem \ref{Thm:CLT-in-Lp}$(ii)$ are satisfied; we omit the details to avoid repetition. The claim that \eqref{eq:chisq-GS-onesamp-null} holds for $\beta$-sub-Gaussian $\mu$  with $\beta<\sigma$ is a consequence of \cite[Proposition~1]{AZYA-2022} by noting that the LHS of \eqref{ASSUM:KL_null} equals $I_{\chi^2}(V;V+Z)$, with $V \sim \mu$ and $Z \sim \gamma_{\sigma}$ independent of each other. The final claim  is obvious from \eqref{eq:chisqexpbnd}.
 
 \medskip

\noindent\textbf{Part $\bm{(ii)}$:}
Note that
\begin{flalign}
    \EE \big[ \chisq{\hat {\mu}_n*\gamma_{\sigma}}{\nu*\gamma_{\sigma}}\big]
    & =       \int_{\RR^d} \frac{1}{{n^2\nu*\varphi_{\sigma}(x)}}\EE \left[\left( \sum_{i=1}^n \big(\varphi_{\sigma}(x-X_i)-\nu*\varphi_{\sigma}(x)\big)\right)^2 \right]\,dx\notag \\
    &\stackrel{(a)}{\leq }      \int_{\RR^d} \frac{1}{{\nu*\varphi_{\sigma}(x)}}\EE_{\mu} \left[      \big(\varphi_{\sigma}(x-X)-\nu*\varphi_{\sigma}(x)\big)^2 \right]\,dx\notag \\
        &\stackrel{(b)}{= }     \int_{\RR^d}\frac{\mathsf{Var}_{\mu}\big(\varphi_{\sigma}(x-\cdot)\big)}{\nu*\varphi_{\sigma}(x)}\,dx+\chisq{\mu*\varphi_{\sigma}}{\nu*\varphi_{\sigma}}, \notag &&
 \end{flalign} 
where $(a)$ uses the convexity of the map $x \mapsto x^2$ while $(b)$ is because $X_1,\ldots,X_n$ are i.i.d. Since the RHS of the above equation is finite by assumption, we conclude that 
$\chisq{\hat {\mu}_n*\gamma_{\sigma}}{\mu*\gamma_{\sigma}}<\infty$ a.s. Also, under \eqref{ASSUM:KL_alt}, it follows that
\begin{align}
 n^{1/2} \left(\frac{\hat\mu_n*\varphi_{\sigma}}{\nu*\varphi_{\sigma}}-\frac{\mu*\varphi_{\sigma}}{\nu*\varphi_{\sigma}}\right) \trightarrow{w} \frac{G_{\mu,\sigma}}{\nu*\varphi_{\sigma}}\quad\mbox{ in }L^2(\nu*\varphi_{\sigma}), \notag
 \end{align}
 via similar arguments to those in proof of Proposition~\ref{Prop:GS-KL-limdist}$(ii)$. Then, \eqref{eq:chisq-GS-onesamp-alt} is a direct consequence of Theorem~\ref{Thm:chisqdiv-limdist}$(ii)$ with $\rho=\nu*\gamma_{\sigma}$, which implies 
\begin{align}
      n^{\frac 12}\big(\chisq{\hat \mu_n*\gamma_\sigma}{\nu*\gamma_{\sigma}}-\chisq{\mu*\gamma_{\sigma}}{\nu*\gamma_{\sigma}}\big) &\trightarrow{d} 2\int_{\RR^d} G_{\mu,\sigma}(x) \frac{\mu*\varphi_{\sigma}(x)}{\nu*\varphi_{\sigma}(x)} d x \sim  N\big(0,v_{1,f_{\chi^2}}^2(\mu,\nu,\sigma)\big), \notag
\end{align}
where 
\begin{align}
 v_{1,f_{\chi^2}}^2(\mu,\nu,\sigma)= 4 \int_{\RR^d}\int_{\RR^d} \Sigma^{(1,1)}_{\mu,\nu,\sigma}(x,y) \frac{\mu*\varphi_{\sigma}(x)}{\nu*\varphi_{\sigma}(x)}\frac{\mu*\varphi_{\sigma}(y)}{\nu*\varphi_{\sigma}(y)}d x\, d y. \label{eq:chisqvaronesamp}
\end{align}

Finally, the last claim follows from the first since \eqref{ASSUM:KL_alt} is satisfied for $\beta$-sub-Gaussian $\mu$ with $\beta<\sigma$ and $ \norm{d \mu/d \nu}_{\infty}<\infty$ due to  \eqref{eq:varcondbnddensrat},  and by \eqref{eq:ratiodensbnd}
\begin{align}
    \chisq{\mu*\varphi_{\sigma}}{\nu*\varphi_{\sigma}}=\int_{\RR^d} \left(\frac{\mu*\varphi_{\sigma}(x)}{\nu*\varphi_{\sigma}(x)}-1\right)^2 \nu*\varphi_{\sigma}(x)\, dx\leq 2\left(\norm{\frac{\mu*\varphi_{\sigma}}{\nu*\varphi_{\sigma}}}_{\infty}^2+1\right) <\infty.\notag 
\end{align}

 \medskip

\noindent\textbf{Part $\bm{(iii)}$:}
We specialize Theorem \ref{Thm:chisqdiv-limdist}$(iii)$ with $\rho=\mu*\gamma_{\sigma}$. First, note that the absolute continuity requirements on the distributions are satisfied, and $p_{\hat {\nu}_n*\gamma_\sigma}>0$. Also,  $\chisq{\hat {\mu}_n*\gamma_\sigma}{\hat {\nu}_n*\gamma_\sigma}<\infty$ a.s. since with $\hat X \sim \hat \mu_n$, $\hat Y \sim \hat \nu_n$, $X \sim \mu $ independent of $Y \sim \mu$, we have
  \begin{flalign}
    \EE \big[ \chisq{\hat {\mu}_n*\gamma_{\sigma}}{\hat \nu_n*\gamma_{\sigma}}\big] &= \EE \big[ \chisq{N(\hat X, \sigma^2I_d)}{N(\hat Y, \sigma^2 I_d)}\big] \notag \\
    &\stackrel{(a)}{\leq }    \EE \big[ \chisq{N(X, \sigma^2I_d)}{N(Y, \sigma^2 I_d)}\big] \notag \\
    &\stackrel{(b)}{=}    \EE \left[e^{\frac{\norm{X-Y}^2}{\sigma^2}}\right]   \stackrel{(c)}{<}\infty. \label{eq:expchisqbnd} &&
 \end{flalign} 
 Here, $(a)$ follows by  convexity of $\chi^2$ divergence,  $(b)$ uses the closed form expression for $\chi^2$ divergence between multi-variate isotropic Gaussians \cite{Nielsen-2013}, and $(c)$ is because $\mu$ has compact support. On the other hand, taking $q(x)=c'e^{c\norm{x}}$ such that $ \hat \mu_n *\varphi_{\sigma}(x)/\hat \nu_n *\varphi_{\sigma}(x) \leq q(x)$ a.s. based on  \eqref{eq:ubnddenratio}, we have
 \begin{align}
   \frac{d\eta_1}{dx}(x)&=\big(1+c'e^{c\norm{x}}\big)\mu*\varphi_{\sigma}(x),\notag \\
    \frac{d\eta_2}{dx}(x)&=\big(1+c'e^{c\norm{x}}+c'^2 e^{2c\norm{x}}\big)\mu*\varphi_{\sigma}(x).\notag
\end{align}
Setting $r_n=n^{1/2}$, it follows from \eqref{eq:chisq-twosample-null} and the arguments in the proof of Part $(iii)$ that 
    \begin{align}
  n \chisq{\hat \mu_n *\varphi_{\sigma}}{\hat \nu_n *\varphi_{\sigma}} \trightarrow{d}  \int_{\mathfrak{S}} \frac{\big(G_{\mu,\sigma}(x)-\tilde G_{\mu,\sigma}(x)\big)}{\mu*\varphi_{\sigma}(x)}^2 d x, \notag
\end{align}
provided that \eqref{eq:weakconvmarg} holds. By the same arguments as in the proof of Proposition \ref{Prop:GS-KL-limdist}$(iii)$, this holds if
\[
\int_{\RR^d}\frac{\mathsf{Var}_{\mu}\big(\varphi_{\sigma}(x-\cdot)\big)}{\mu*\varphi_{\sigma}(x)}\big(1+e^{2c\norm{x}}\big)\,dx<\infty,
\]
which in turn is satisfied because $\supp(\mu)$ is compact (see the derivation leading to \eqref{eq:finitevarint}). 

 \medskip

\noindent\textbf{Part $\bm{(iv)}$:} 
We employ Theorem \ref{Thm:chisqdiv-limdist}$(iv)$ 
with $\rho=\nu*\gamma_{\sigma}$. The required positivity and absolute continuity requirements of the probability measures are readily verified. Also, using Jensen's inequality and \eqref{eq:expchisqbnd}, we have
 \begin{align}
 \chisq{\mu*\gamma_{\sigma}}{\nu*\gamma_{\sigma}} \leq   \EE \big[ \chisq{\hat {\mu}_n*\gamma_{\sigma}}{\hat \nu_n*\gamma_{\sigma}}\big]    <\infty, \notag
 \end{align}
 which implies $ \chisq{\hat {\mu}_n*\gamma_{\sigma}}{\hat \nu_n*\gamma_{\sigma}} <\infty$ a.s.  Moreover,
 \begin{flalign}
\int_{\RR^d}\frac{\big(\mu*\varphi_{\sigma}(x)\big)^4}{\big(\nu*\varphi_{\sigma}(x)\big)^3}\, dx  &= \int_{\RR^d}\frac{ \left(\int_{\RR^d} e^{-\frac{\norm{x-y}^2}{2\sigma^2}} d \mu(y)\right)^4}{ \left(\int_{\RR^d}e^{-\frac{\norm{x-y}^2}{2\sigma^2}}d \nu(y)\right)^3}dx \notag \\
&= \int_{\RR^d} e^{-\frac{\norm{x}^2}{2\sigma^2}}  \frac{ \left(\int_{\RR^d}e^{\frac{2 x\cdot y-\norm{y}^2}{2\sigma^2}}d \mu(y)\right)^4}{ \left(\int_{\RR^d}e^{\frac{2 x\cdot y-\norm{y}^2}{2\sigma^2}}d \nu(y)\right)^3}dx \notag \\
& \leq  \int_{\RR^d} e^{-\frac{\norm{x}^2}{2\sigma^2}}  e^{c \norm{x}}dx<\infty. \notag &&
 \end{flalign}
Taking $q(x)=c'e^{c\norm{x}}$ with $ \hat \mu_n *\varphi_{\sigma}(x)/\hat \nu_n *\varphi_{\sigma}(x) \leq q(x)$ a.s. (see \eqref{eq:ubnddenratio}), we have
 \begin{align}
   \frac{d\eta_1}{dx}(x)&=\nu*\varphi_{\sigma}(x)+\mu*\varphi_{\sigma}(x)+\nu*\varphi_{\sigma}(x)\,c'e^{c\norm{x}},\notag \\
    \frac{d\eta_2}{dx}(x)&=\frac{d\eta_1}{dx}(x)+\nu*\varphi_{\sigma}(x)\,c'^2e^{2c\norm{x}}+\frac{\big(\mu*\varphi_{\sigma}(x)\big)^2}{\nu*\varphi_{\sigma}(x)}.\notag
\end{align}
Similarly to the proof of Proposition \ref{Prop:GS-chisq-limdist}$(iv)$, 
 setting $r_n=n^{1/2}$ in \eqref{eq:chisq-twosample-alt}, yields
 \begin{align}
       n^{\frac 12}\big(\chisq{\hat \mu_n*\gamma_\sigma}{\hat \nu_n*\gamma_{\sigma}}\mspace{-2 mu}-\mspace{-2 mu}\chisq{\mu*\gamma_{\sigma}}{\nu*\gamma_{\sigma}}\mspace{-2 mu}\big)&\mspace{-2 mu}\trightarrow{d}\mspace{-2 mu} 2\mspace{-2 mu}\int_{\RR^d} \mspace{-2 mu}G_{\mu,\sigma}(x) \frac{\mu*\varphi_{\sigma}(x)}{\nu*\varphi_{\sigma}(x)} d x \mspace{-2 mu}+ \mspace{-2 mu}\int_{\RR^d} \mspace{-2 mu} G_{\nu,\sigma}(x) \left(\frac{\mu*\varphi_{\sigma}(x)}{\nu*\varphi_{\sigma}(x)}\right)^2 d x \notag \\
   & \qquad  \sim N\big(0,v_{2,f_{\chi^2}}^2(\mu,\nu,\sigma)\big), \notag
 \end{align}
 provided that \eqref{eq:wcden2sampalt} holds. Here, the variance is given by
 \begin{align}
  v_{2,f_{\chi^2}}^2(\mu,\nu,\sigma)&:= \sum_{1 \leq i,j \leq 2} \int_{\RR^d}\int_{\RR^d}  \Sigma^{(i,j)}_{\mu,\nu,\sigma}(x,y) L_{i,f_{\chi^2}}(x)L_{j,f_{\chi^2}}(y) d x\, d y\notag \\
    &=\sum_{1 \leq i,j \leq 2}\int_{\RR^d}\int_{\RR^d}  \Sigma^{(i,j)}_{\mu,\nu,\sigma}(x,y) \tilde L_{i,f_{\chi^2}}(x) \tilde L_{j,f_{\chi^2}}(y) d x\, d y, \label{eq:var-smoothedchisq2samp} 
\end{align}
with $\tilde L_{1,f_{\chi^2}}\mspace{-4 mu}:=\mspace{-4 mu}2 \left(\mu*\varphi_{\sigma}/\nu*\varphi_{\sigma}\right)$ and $\tilde L_{2,f_{\chi^2}}\mspace{-4 mu}:=\mspace{-4 mu} (\mu*\varphi_{\sigma}/\nu*\varphi_{\sigma})^2$.
Finally, Theorem \ref{Thm:CLT-in-Lp} implies that \eqref{eq:wcden2sampalt} is satisfied if
\begin{align}
    & \int_{\RR^d}\frac{\mathsf{Var}_{\mu}\big(\varphi_{\sigma}(x-\cdot)\big)}{\nu*\varphi_{\sigma}(x)}\big(1+e^{c\norm{x}}\big)\, dx +\int_{\RR^d}\frac{\mathsf{Var}_{\mu}\big(\varphi_{\sigma}(x-\cdot)\big)\mu*\varphi_{\sigma}(x)}{\big(\nu*\varphi_{\sigma}(x)\big)^2}\, dx<\infty, \notag\\
      &  \int_{\RR^d}\frac{\mathsf{Var}_{\nu}\big(\varphi_{\sigma}(x-\cdot)\big)}{\nu*\varphi_{\sigma}(x)}\big(1+e^{2c\norm{x}}\big)\, dx +\sum_{j=1,2}\int_{\RR^d}\frac{\mathsf{Var}_{\nu}\big(\varphi_{\sigma}(x-\cdot)\big)\big(\mu*\varphi_{\sigma}(x)\big)^j}{\big(\nu*\varphi_{\sigma}(x)\big)^{j+1}}\, dx<\infty,\notag 
\end{align}
which holds for compactly supported $\mu,\nu$, akin to \eqref{eq:finitevarint}. This completes the proof.
\end{proof}

\subsection{$\mathsf{H}^2$ distance} \label{App:GSfdiv-helsq} The following proposition obtains limit distributions for  Gaussian-smoothed squared Hellinger distance. 
\begin{prop}[Limit distribution for Gaussian-smoothed $\mathsf{H}^2$ distance] \label{Prop:GS-helsq-limdist} The following hold:
\begin{enumerate}[(i)]
    \item  (One-sample null)  If \eqref{ASSUM:KL_null} holds,
then there exists a version of $G_{\mu,\sigma}$ such that $G_{\mu,\sigma}/\sqrt{\mu*\varphi_{\sigma}}$ is $L^2(\RR^d)$-valued, and 
\begin{align}
  n \helsq{\hat \mu_n*\gamma_\sigma}{\mu*\gamma_\sigma} \trightarrow{d}   \frac 14 \int_{\mathfrak{S}} \frac{G_{\mu,\sigma}^2(x)}{\mu*\varphi_{\sigma}(x)}d x, \label{eq:helsq-GS-onesamp-null}
\end{align}
where the limit can be represented as a weighted sum of i.i.d.  $\chi^2$ random variables with 1 degree of freedom.
In particular,  \eqref{eq:helsq-GS-onesamp-null} holds for $\beta$-sub-Gaussian $\mu$  with $\beta<\sigma$. Conversely, if~\eqref{ASSUM:KL_null} is violated, then  $\liminf_{n \rightarrow \infty} n \EE\big[\helsq{\hat \mu_n*\gamma_\sigma}{\mu*\gamma_\sigma} \big]=\infty$. 
\medskip
\item  (One-sample alternative)  If  \begin{equation}
      \int_{\RR^d}\frac{\mathsf{Var}_{\mu}\big(\varphi_{\sigma}(x-\cdot)\big)}{\mu*\varphi_{\sigma}(x)}\, dx+ \int_{\RR^d}\frac{\mathsf{Var}_{\mu}\big(\varphi_{\sigma}(x-\cdot)\big)}{\mu*\varphi_{\sigma}(x)}\left(\frac{\nu*\varphi_{\sigma}(x)}{\mu*\varphi_{\sigma}(x)}\right)^{\frac 12}\, dx <\infty,\label{ASSUM:Hel_alt}
    \end{equation}    then
\begin{align}
    n^{\frac 12}\big(\helsq{\hat \mu_n*\gamma_\sigma}{\nu*\gamma_{\sigma}}-\helsq{\mu*\gamma_{\sigma}}{\nu*\gamma_{\sigma}}\big) \trightarrow{d}  N\big(0,v_{1,f_{\mathsf{H}^2}}^2(\mu,\nu,\sigma)\big), \label{eq:helsq-GS-onesamp-alt}
\end{align}
where $v_{1,f_{\mathsf{H}^2}}^2(\mu,\nu,\sigma)$ is given in \eqref{eq:helsqvaronesamp}. In particular,  \eqref{eq:helsq-GS-onesamp-alt} holds for $\beta$-sub-Gaussian $\mu$ with $\beta<\sigma$ and $\nu \ll \mu$ such that  $\norm{d \nu/d \mu}_{\infty}  <\infty$.
\medskip
\item (Two-sample null) If 
\begin{equation}
      \int_{\RR^d}\frac{\mathsf{Var}_{\mu}\big(\varphi_{\sigma}(x-\cdot)\big)}{\big(\mu*\varphi_{\sigma}(x)\big)^{3/2}}\, dx <\infty,\label{ASSUM:Hel_null_2samp}
    \end{equation} 
  then   there exists a version of $G_{\mu,\sigma}$, $\tilde G_{\mu,\sigma}$ such that $G_{\mu,\sigma}/\sqrt{\mu*\varphi_{\sigma}}$ and $\tilde G_{\mu,\sigma}/\sqrt{\mu*\varphi_{\sigma}}$ are $L^2(\RR^d)$-valued, and  
\begin{align}
  n\, \helsq{\hat \mu_n*\gamma_\sigma}{\hat \nu_n*\gamma_\sigma} \trightarrow{d}   \frac 14 \int_{\RR^d} \frac{\big(G_{\mu,\sigma}(x)-\tilde G_{\mu,\sigma}(x)\big)^2}{\mu*\varphi_{\sigma}(x)} d x, \label{eq:helsq-GS-twosamp-null}
\end{align}
where the limit can be represented as a weighted sum of i.i.d. $\chi^2$ random variables with 1 degree of freedom. 
In particular,~\eqref{ASSUM:Hel_null_2samp} and \eqref{eq:helsq-GS-twosamp-null} holds for $\beta$-sub-Gaussian $\mu$  with $\beta<\sigma/\sqrt{6}$. 
\medskip
 \item (Two-sample alternative) If either of the following conditions hold
 \begin{enumerate}
    \item $\mu,\nu$ have compact supports;
    \item $\norm{d \mu/d\nu}_{\infty}<\infty$, and 
  \begin{align}
      \qquad\int_{\RR^d}\frac{\mathsf{Var}_{\mu}\big(\varphi_{\sigma}(x-\cdot)\big)}{\big(\mu*\varphi_{\sigma}(x)\big)^{3/2}}\, dx +
     \int_{\RR^d}\frac{\mathsf{Var}_{\nu}\big(\varphi_{\sigma}(x-\cdot)\big)}{\mu*\varphi_{\sigma}(x)}\, dx+  \int_{\RR^d}\frac{\mathsf{Var}_{\nu}\big(\varphi_{\sigma}(x-\cdot)\big)}{\big(\nu*\varphi_{\sigma}(x)\big)^{3/2}}\, dx <\infty\,; \label{ASSUM:Hel_alt_2samp}
    \end{align}
 \item The conditions in $(b)$ hold with the role of $\mu$ and $\nu$ interchanged;
 \end{enumerate}
 \vspace{5 pt}
 then 
 \begin{align}
   & n^{\frac 12}\big(\helsq{\hat \mu_n*\gamma_\sigma}{\hat \nu_n*\gamma_{\sigma}}-\helsq{\mu*\gamma_{\sigma}}{\nu*\gamma_{\sigma}}\big)\trightarrow{d} N\big(0,v_{2,f_{\mathsf{H}^2}}^2(\mu,\nu,\sigma)\big),\label{eq:helsq-GS-twosamp-alt}
\end{align}
where $v_{2,f_{\mathsf{H}^2}}^2(\mu,\nu,\sigma)$ is specified in \eqref{eq:var-smoothedhelsq2samp}. In particular,  \eqref{eq:helsq-GS-twosamp-alt} holds for $\beta$-sub-Gaussian $\mu$, $\nu$  with $\beta<\sigma/\sqrt{6}$ such that $\mu \ll \gg \nu $ and $\norm{d \mu/d\nu}_{\infty} \vee \norm{d \nu/d\mu}_{\infty}<\infty$.
\end{enumerate}
\end{prop}
\begin{proof}
    The proof relies on Theorem \ref{Thm:Helsqdiv-limdist} by ensuring that the relevant conditions therein are fulfilled.

\medskip
\noindent\textbf{Part $\bm{(i)}$:}
We apply Theorem \ref{Thm:Helsqdiv-limdist}$(i)$ with $\rho=\mu*\gamma_{\sigma}$.
Observe that $p_{\eta}=2$ and that the positivity and absolute continuity assumptions are satisfied. Then, \eqref{eq:helsq-GS-onesamp-null} will follow from \eqref{eq:Hel-onesample-null}, provided    \eqref{eq:weakconv-emprat-null1} holds, which in turn is fulfilled under \eqref{ASSUM:KL_null}  (see \eqref{eq:pathsasL2} and the following discussion). The penultimate claim is because $\beta$-sub-Gaussian $\mu$  with $\beta<\sigma$ satisfies~\eqref{ASSUM:KL_null}, as mentioned in the proof of Proposition \ref{Prop:GS-KL-limdist}$(i)$. The final claim follows by instantiating Lemma~\ref{lem:fdivnullonesamplb} to $\mathsf{H}^2$ distance and noting that $f_{\mathsf{H}^2}$ satisfy the regularity conditions therein.

\medskip
\noindent\textbf{Part $\bm{(ii)}$:}
We apply Theorem \ref{Thm:Helsqdiv-limdist}$(ii)$ with 
$\rho=\mu*\gamma_{\sigma}$. Note that $\norm{\nu*\varphi_{\sigma}/\mu*\varphi_{\sigma}}_{1,\mu*\varphi_{\sigma}}=1 $ and 
$d\eta/dx=\mu*\varphi_{\sigma}+\big(\mu*\varphi_{\sigma}\big)^{1/2} \big( \nu*\varphi_{\sigma}\big)^{1/2}$, which implies that $\eta$ is a finite measure. Thus, provided that
\begin{align}
 n^{1/2} \left(\frac{\hat\mu_n*\varphi_{\sigma}}{\mu*\varphi_{\sigma}}-1\right) \trightarrow{w} \frac{G_{\mu,\sigma}}{\mu*\varphi_{\sigma}}\quad\mbox{ in }L^2(\eta), \notag
 \end{align}
we obtain from \eqref{eq:Helsq-onesample-alt} the limit distribution
\begin{align}
        n^{\frac 12}\big(\helsq{\hat \mu_n*\gamma_\sigma}{\nu*\gamma_{\sigma}}-\helsq{\mu*\gamma_{\sigma}}{\nu*\gamma_{\sigma}}\big) &\trightarrow{d} \int_{\RR^d}  \left(\frac{\nu*\varphi_{\sigma}(x)}{\mu*\varphi_{\sigma}(x)}\right)^{\frac 12}  G_{\mu,\sigma}(x)\, dx  \sim N\big(0,v_{1,f_{\mathsf{H}^2}}^2(\mu,\nu,\sigma)\big), \notag
\end{align}
where 
 \begin{align}
 v_{1,f_{\mathsf{H}^2}}^2(\mu,\nu,\sigma)=  \int_{\RR^d}\int_{\RR^d} \Sigma^{(1,1)}_{\mu,\nu,\sigma}(x,y) \left(\frac{\nu*\varphi_{\sigma}(x)}{\mu*\varphi_{\sigma}(x)}\right)^{\frac 12} \left(\frac{\nu*\varphi_{\sigma}(y)}{\mu*\varphi_{\sigma}(y)}\right)^{\frac 12} d x\, d y. \label{eq:helsqvaronesamp}
\end{align}
To establish the weak convergence condition we again verify Condition $(ii)$ of Theorem \ref{Thm:CLT-in-Lp}, which uses the assumption from \eqref{ASSUM:Hel_alt}, and then invoke Theorem \ref{Thm:CLT-in-Lp}$(i)$ (see the proof of Proposition \ref{Prop:GS-KL-limdist}). The final claim follows as the LHS of \eqref{ASSUM:Hel_alt} is finite for $\beta$-sub-Gaussian $\mu$ with $\beta<\sigma$ and $\norm{d \nu/d \mu}_{\infty}  <\infty$, based on \eqref{eq:ratiodensbnd} and the arguments at the end of proof of Proposition \ref{Prop:GS-KL-limdist}$(i)$.  

\medskip
\noindent\textbf{Part $\bm{(iii)}$:}
We specialize Theorem \ref{Thm:Helsqdiv-limdist}$(iii)$ with 
$\rho=\mu*\gamma_{\sigma}$. Setting $q_1 =q_2=c_{d,\sigma}/\mu*\gamma_{\sigma}$ where $c_{d,\sigma}=(2\pi \sigma^2)^{-d/2}$, we have that $\eta_1=\eta_2=\eta$ is a $\sigma$-finite measure specified by the Lebesgue density
$d \eta/dx=2\mu*\varphi_{\sigma}+\big(c_{d,\sigma}\,\mu*\varphi_{\sigma}\big)^{1/2}$ (note that both terms are bounded on $\RR^d$). Then, from \eqref{eq:Helsq-twosample-null}, we obtain
 \begin{align}
  n \helsq{\hat \mu_n *\varphi_{\sigma}}{\hat \nu_n *\varphi_{\sigma}} \trightarrow{d}  \frac 14 \int_{\mathfrak{S}} \frac{\big(G_{\mu,\sigma}(x)-\tilde G_{\mu,\sigma}(x)\big)}{\mu*\varphi_{\sigma}(x)}^2 d x, \notag
\end{align}
provided that
\begin{align} 
  n^{1/2} \left(\frac{\hat\mu_n*\varphi_{\sigma}}{\mu*\varphi_{\sigma}}-1\right) \trightarrow{w} \frac{G_{\mu,\sigma}}{\mu*\varphi_{\sigma}}\quad  \mbox{ and }\quad n^{1/2} \left(\frac{\hat\nu_n*\varphi_{\sigma}}{\mu*\varphi_{\sigma}}-1\right) \trightarrow{w} \frac{\tilde G_{\mu,\sigma}}{\mu*\varphi_{\sigma}}\quad\mbox{ in }L^2(\eta). \notag
 \end{align}
Applying Theorem \ref{Thm:CLT-in-Lp}$(ii)$, the above weak convergence holds once \eqref{ASSUM:Hel_null_2samp} is satisfied.

\medskip

To prove the final claim in Part $(iii)$,
we first note that
\begin{flalign}
\int_{\RR^{d}} \frac{\mathsf{Var}_{\mu}\big(\varphi_{\sigma}(x-\cdot)\big)}{\big(\mu*\varphi_{\sigma}(x)\big)^{\frac 32}} dx &\le \int_{\RR^{d}} \int_{\RR^{d}} \frac{\varphi_{\sigma}^{2}(x-y)}{\big(\mu*\varphi_{ \sigma}(x))^{\frac 32}} dx\,  d\mu(y) \notag \\
&= \int_{\RR^{d}} e^{-\| y \|^{2}/\sigma^{2}} \frac{1}{(2\pi \sigma^{2})^{d}}\int_{\RR^{d}} \frac{e^{2y \cdot x/\sigma^{2}-\|x\|^{2}/\sigma^{2}}}{\big(\mu*\varphi_{ \sigma}(x))^{\frac 32}} dx \, d\mu(y). \notag &&
\end{flalign}
Since $\|x-y\|^{2} \le (1+\tau)\|x\|^{2}+(1+\tau^{-1})\|y\|^{2}, \, \forall\, \tau  \in (0,1)$, Jensen's inequality implies
\begin{align}
\big(\mu*\varphi_{\sigma}(x)\big)^{\frac 32} \ge \EE_{\mu}\left[\varphi_{\sigma}^{\frac 32}(x-\cdot)\right] \geq  \frac{e^{-3(1+\tau)\|x\|^{2}/(4\sigma^{2})}}{(2\pi \sigma^{2})^{\frac{3d}{4}}}  \underbrace{\int_{\RR^{d}} e^{-3(1+\tau^{-1})\|y\|^{2}/(4\sigma^{2})} d \mu(y)}_{=: c_{\mu,\tau}}. \notag   
\end{align}
Then
\begin{align}
\frac{1}{(2\pi \sigma^{2})^{d}} \int_{\RR^{d}} \frac{e^{2y \cdot x/\sigma^{2}-\|x\|^{2}/\sigma^{2}}}{\big(\mu*\varphi_{ \sigma}(x))^{\frac 32}} dx 
&\le \frac{2^{\frac{3d}{4}}\sigma^{\frac d2}\pi^{\frac d4}}{(1-3\tau)^{\frac d2} c_{\mu,\tau}} \int_{\RR^{d}} e^{2y \cdot x/\sigma^{2}} \varphi_{\sqrt{2}\sigma/\sqrt{1-3\tau}}(x) d x =\frac{2^{\frac{3d}{4}}\sigma^{\frac d2}\pi^{\frac d4}}{(1-3\tau)^{\frac d2} c_{\mu,\tau}} e^{ \frac{4 \| y \|^{2}}{(1-3\tau)
\sigma^{2}}}, \notag
\end{align}
where in the last equation we used $\EE\big[e^{\alpha \cdot W}\big] = e^{\norm{\alpha}^2 \sigma^2/2}$ for $W \sim N(0,\sigma^2 I_d)$.
Conclude that 
\begin{align}
\int_{\RR^{d}} \frac{\mathsf{Var}_{\mu}\big(\varphi_{\sigma}(x-\cdot)\big)}{\big(\mu*\varphi_{ \sigma}(x))^{\frac 32}}dx
&\le \frac{2^{\frac{3d}{4}}\sigma^{\frac d2}\pi^{\frac d4}}{(1-3\tau)^{\frac d2} c_{\mu,\tau}} \int_{\RR^{d}}e^{ \frac{3(1+\tau)\| y \|^{2}}{(1-3\tau)\sigma^{2}}}  d\mu(y) \leq \frac{2^{\frac{3d}{4}}\sigma^{\frac d2}\pi^{\frac d 4}}{(1-3\tau)^{\frac d 2} c_{\mu,\tau}} \left ( 1-\frac{6(1+\tau)\beta^{2}}{(1-3\tau)\sigma^{2}} \right)^{-\frac d2}, \notag
\end{align}
provided that
$3(1+\tau)/\big((1-3\tau)\sigma^{2}\big) < 1/(2\beta^{2})$, i.e., $\beta < \sigma \sqrt{1-3\tau}/\sqrt{6(1+\tau)}$. Here, the final inequality uses that for any $\beta$-sub-Gaussian $X$ with mean zero and $0 \le \tau' < 1/(2\beta^{2})$, we have 
$\E\big [e^{\tau' \| X \|^{2}} \big] \leq (1-2\beta^{2} \tau')^{-d/2}$ (see \cite[Equation 7]{Goldfeld-Kato-2020}). Since $\tau \in (0,1)$ is arbitrary and $c_{\mu,\tau}>0$ for $\tau>0$, the desired result follows.

\medskip
\noindent\textbf{Part $\bm{(iv)}$:} We use Theorem~\ref{Thm:Helsqdiv-limdist}$(iv)$ with 
$\rho=\mu*\gamma_{\sigma}$. Observe that $\norm{\nu*\varphi_{\sigma}/\mu*\varphi_{\sigma}}_{1,\mu*\varphi_{\sigma}}=1 $, and since $\norm{d \mu/d \nu}_{\infty} <\infty$ by assumption, we further have $\mu*\varphi_{\sigma}/\nu*\varphi_{\sigma} \in L^1(\mu*\gamma_{\sigma})$. With $q_1 =q_2=c_{d,\sigma}/\mu*\gamma_{\sigma}$, the measures $\eta_1$ and $\eta_2$ in Theorem \ref{Thm:Helsqdiv-limdist}$(iv)$ are specified by the Lebesgue densities
\begin{align}
    \frac{d \eta_1}{dx}&=\mu*\varphi_{\sigma}+\big(\mu*\varphi_{\sigma}\,\nu*\varphi_{\sigma}\big)^{\frac 12}+\big(c_{d,\sigma}\,\mu*\varphi_{\sigma}\big)^{\frac 12}+\big(\mu*\varphi_{\sigma}\big)^{\frac 32}\big(\nu*\varphi_{\sigma}\big)^{-\frac 12},\label{eq:lebdenssqhelalt1} \\
     \frac{d \eta_2}{dx}&=\mu*\varphi_{\sigma}+\big(\mu*\varphi_{\sigma}\big)^{\frac 52}\big(\nu*\varphi_{\sigma}\big)^{-\frac 32}+c_{d,\sigma}^{\frac 12}\big(\mu*\varphi_{\sigma}\big)^2\big(\nu*\varphi_{\sigma}\big)^{-\frac 32}+\big(\mu*\varphi_{\sigma}\big)^{\frac 32}\big(\nu*\varphi_{\sigma}\big)^{-\frac 12}. \label{eq:lebdenssqhelalt2}
\end{align}
Thus, if the weak convergence conditions
\begin{align} \notag
\begin{split}
 &  n^{1/2} \left(\frac{\hat\mu_n*\varphi_{\sigma}}{\mu*\varphi_{\sigma}}-1\right) \trightarrow{w} \frac{G_{\mu,\sigma}}{\mu*\varphi_{\sigma}}\quad\mbox{ in }L^2(\eta_1), \\
   & n^{1/2} \left(\frac{\hat\nu_n*\varphi_{\sigma}}{\mu*\varphi_{\sigma}}-\frac{\nu*\varphi_{\sigma}}{\mu*\varphi_{\sigma}}\right) \trightarrow{w} \frac{G_{\nu,\sigma}}{\mu*\varphi_{\sigma}}\quad\mbox{ in }L^2(\eta_2).
\end{split}
 \end{align}
 hold, then \eqref{eq:Helsq-twosample-alt} yields 
\begin{align}
      & n^{\frac 12}\big(\helsq{\hat \mu_n*\gamma_\sigma}{\hat \nu_n*\gamma_{\sigma}}-\helsq{\mu*\gamma_{\sigma}}{\nu*\gamma_{\sigma}}\big)\notag \\
   &\qquad \qquad \qquad \qquad\trightarrow{d} \int_{\RR^d}  \left(\frac{\nu*\varphi_{\sigma}(x)}{\mu*\varphi_{\sigma}(x)}\right)^{\frac 12}  G_{\mu,\sigma}(x)\, dx+ \int_{\RR^d}  \left(\frac{\mu*\varphi_{\sigma}(x)}{\nu*\varphi_{\sigma}(x)}\right)^{\frac 12}  G_{\nu,\sigma}(x)\, dx \notag \\
   &\qquad \qquad \qquad \qquad \qquad \sim  N\big(0,v_{2,f_{\mathsf{H}^2}}^2(\mu,\nu,\sigma)\big), \notag
\end{align}
where
 \begin{align}
  v_{2,f_{\mathsf{H}^2}}^2(\mu,\nu,\sigma)&:= \sum_{1 \leq i,j \leq 2} \int_{\RR^d}\int_{\RR^d}  \Sigma^{(i,j)}_{\mu,\nu,\sigma}(x,y) L_{i,f_{\mathsf{H}^2}}(x)L_{j,f_{\mathsf{H}^2}}(y) d x\, d y\notag \\
    &=\sum_{1 \leq i,j \leq 2}\int_{\RR^d}\int_{\RR^d}  \Sigma^{(i,j)}_{\mu,\nu,\sigma}(x,y) \tilde L_{i,f_{\mathsf{H}^2}}(x) \tilde L_{j,f_{\mathsf{H}^2}}(y) d x\, d y, \label{eq:var-smoothedhelsq2samp} 
\end{align}
with $\tilde L_{1,f_{\mathsf{H}^2}}:=(\nu*\varphi_{\sigma}/\mu*\varphi_{\sigma})^{1/2} $ and $\tilde L_{2,f_{\mathsf{H}^2}}:= (\mu*\varphi_{\sigma}/\nu*\varphi_{\sigma})^{1/2}$.
Since $\norm{d \mu/d \nu}_{\infty}$ $ <\infty$ by assumption,   $\eta_1$ and $\eta_2$ are $\sigma$-finite measures because the terms in the RHS of \eqref{eq:lebdenssqhelalt1} and \eqref{eq:lebdenssqhelalt2} either have finite Lebesgue integrals or are bounded on $\RR^d$. Hence, the above weak convergences hold if the conditions in Theorem \ref{Thm:CLT-in-Lp}$(ii)$ are satisfied. This in turn happens if
  \begin{equation}
      \int_{\RR^d}\frac{\mathsf{Var}_{\mu}\big(\varphi_{\sigma}(x-\cdot)\big)}{\big(\mu*\varphi_{\sigma}(x)\big)^{3/2}}\, dx+\int_{\RR^d}\frac{\mathsf{Var}_{\mu}\big(\varphi_{\sigma}(x-\cdot)\big)}{\big(\mu*\varphi_{\sigma}(x)\nu*\varphi_{\sigma}(x)\big)^{1/2}}\, dx <\infty, \notag
    \end{equation}
    and
    \begin{equation}
     \int_{\RR^d}\frac{\mathsf{Var}_{\nu}\big(\varphi_{\sigma}(x-\cdot)\big)}{\mu*\varphi_{\sigma}(x)}\, dx+  \int_{\RR^d}\frac{\mathsf{Var}_{\nu}\big(\varphi_{\sigma}(x-\cdot)\big)}{\big(\nu*\varphi_{\sigma}(x)\big)^{3/2}}\, dx+\int_{\RR^d}\frac{\mathsf{Var}_{\nu}\big(\varphi_{\sigma}(x-\cdot)\big)}{\big(\mu*\varphi_{\sigma}(x)\nu*\varphi_{\sigma}(x)\big)^{1/2}}\, dx <\infty.\notag
    \end{equation}
The above conditions simplify to that in  \eqref{ASSUM:Hel_alt_2samp} when $\norm{d \mu/d \nu}_{\infty} <\infty$, thus proving the desired claim  under the conditions in $(b)$. The validity of \eqref{eq:helsq-GS-twosamp-alt} under conditions in $(c)$ is a consequence of the  symmetry of $\mathsf{H}^2$ distance in its arguments. 
For Condition (a), we need to show that \eqref{ASSUM:Hel_alt_2samp} holds when $\mu,\nu$ have compact supports. Using steps similar to those in \eqref{eq:boundratdens}, one readily verifies \eqref{ASSUM:Hel_alt_2samp}, as well as that $\eta_1$ and $\eta_2$ are $\sigma$-finite, and $\mu*\varphi_{\sigma}/\nu*\varphi_{\sigma} \in L^1(\mu*\gamma_{\sigma})$. Hence, all the sufficient conditions required for \eqref{eq:helsq-GS-twosamp-alt} to hold are verified.

\medskip
The final claim in Part $(iv)$ follows from the proof of the final claim in Part $(iii)$ which shows that the first and last terms in the LHS of \eqref{ASSUM:Hel_alt_2samp} are finite for $\beta$-sub-Gaussian $\mu,\nu$  with $\beta<\sigma/\sqrt{6}$, and 
\begin{flalign}
     \int_{\RR^d}\frac{\mathsf{Var}_{\nu}\big(\varphi_{\sigma}(x-\cdot)\big)}{\mu*\varphi_{\sigma}(x)}\, dx &= \int_{\RR^d}\frac{\mathsf{Var}_{\nu}\big(\varphi_{\sigma}(x-\cdot)\big)}{\nu*\varphi_{\sigma}(x)}\frac{\nu*\varphi_{\sigma}(x)}{\mu*\varphi_{\sigma}(x)}\, dx \notag \\
     &\leq \norm{\frac{d\nu}{d\mu}}_{\infty}\int_{\RR^d}\frac{\mathsf{Var}_{\nu}\big(\varphi_{\sigma}(x-\cdot)\big)}{\nu*\varphi_{\sigma}(x)} \, dx \notag \\
       &\leq (2\pi \sigma^{2})^{-d/4} \norm{\frac{d\nu}{d\mu}}_{\infty}\int_{\RR^d}\frac{\mathsf{Var}_{\nu}\big(\varphi_{\sigma}(x-\cdot)\big)}{\big(\nu*\varphi_{\sigma}(x)\big)^{3/2}} \, dx <\infty. \notag &&
\end{flalign}
 This completes the proof.
\end{proof}

\subsection{TV distance}\label{App:GSfdiv-TV}
Before stating limit distributions for Gaussian-smoothed TV distance, we recall the definition of tightness or stochastic boundedness of a sequence of random variables. A   sequence of real-valued random variables $Z_n,n \in \NN$, is tight if for any $\epsilon>0$, there exists a constant $c_{\epsilon}$ such that $\PP(\abs{Z_n}>c_{\epsilon}) \leq \epsilon$ for all $n$ \cite{AVDV-book}.
\begin{prop}[Gaussian-smoothed TV distance limit distribution] \label{Prop:GS-TV-limdist}
Let $\rho=\lambda$,  $\cQ=\{s \in \RR^d: p_{\mu*\gamma_{\sigma}}(s)=p_{\nu*\gamma_{\sigma}}(s)\}$ and  $\mathrm{sgn}(x)=x/|x|$ for $x \neq 0$. Then, the following hold:
\begin{enumerate}[(i)]
    \item  (One-sample null and alternative) If 
    \begin{equation}
       \int_{\RR^d}\sqrt{\mathsf{Var}_{\mu}\big(\varphi_{\sigma}(x-\cdot)\big)}\, dx <\infty,\label{ASSUM:TV_null}
    \end{equation}
then  there exists a version of $G_{\mu,\sigma}$ which is $L^1(\RR^d)$ valued, and 
\begin{align}
  &  n^{\frac 12}\big(\tv{\hat \mu_n*\gamma_\sigma}{\nu*\gamma_{\sigma}}-\tv{\mu*\gamma_{\sigma}}{\nu*\gamma_{\sigma}}\big)   
    \notag \\
    &\qquad \qquad \qquad \qquad\qquad \qquad \trightarrow{d}\frac 12 \int_{\cQ}\abs{G_{\mu,\sigma}} dx +\frac 12 \int_{\mathfrak{S}\setminus \cQ} \mathrm{sgn}\big(p_{\mu*\gamma_{\sigma}}-p_{\nu*\gamma_{\sigma}}\big)G_{\mu,\sigma}dx. \label{eq:TV-GS-onesamp-alt}
\end{align}
In particular, \eqref{ASSUM:TV_null} and \eqref{eq:TV-GS-onesamp-alt} holds if $\mu$ has finite $(2d + \epsilon)$ moment for some $\epsilon > 0$, i.e., $\E[\| X \|^{2d+\epsilon}] < \infty$ for $X \sim \mu$.
Conversely, if \eqref{ASSUM:TV_null} is violated, then the sequence $n^{1/2} \tv{\hat \mu_n*\gamma_\sigma}{\mu*\gamma_{\sigma}}$ is not tight.  
\medskip
 \item (Two-sample null and alternative)
 If 
  \begin{equation}
           \int_{\RR^d}\sqrt{\mathsf{Var}_{\mu}\big(\varphi_{\sigma}(x-\cdot)\big)}\, dx+\int_{\RR^d}\sqrt{\mathsf{Var}_{\nu}\big(\varphi_{\sigma}(x-\cdot)\big)}\, dx <\infty,\label{ASSUM:TV_alt-2samp}
    \end{equation}
  then   there exists a version of $G_{\mu,\sigma}$ and $G_{\nu,\sigma}$ which are $L^1(\RR^d)$ valued  such that
\begin{align}
   & n^{\frac 12}\big(\tv{\hat \mu_n*\gamma_\sigma}{\hat \nu_n*\gamma_{\sigma}}-\tv{\mu*\gamma_{\sigma}}{\nu*\gamma_{\sigma}}\big) \notag \\
   & \qquad \qquad \qquad\qquad \trightarrow{d}  \frac 12 \int_{\cQ}\abs{G_{\mu,\sigma}-G_{\nu,\sigma}} dx +\frac 12 \int_{\mathfrak{S}\setminus \cQ} \mathrm{sgn}\big(p_{\mu*\gamma_\sigma}-p_{\nu*\gamma_\sigma}\big)\big(G_{\mu,\sigma}-G_{\nu,\sigma}\big)dx. \label{eq:TV-GS-twosamp-alt}
\end{align}
In particular, \eqref{ASSUM:TV_alt-2samp} and \eqref{eq:TV-GS-twosamp-alt}  holds if $\mu$ and $\nu$ have finite $(2d + \epsilon)$ moments for some $\epsilon > 0$. Conversely, if $\mu=\nu$ and \eqref{ASSUM:TV_alt-2samp} does not hold, then  $n^{1/2} \tv{\hat \mu_n*\gamma_\sigma}{\hat \nu_n*\gamma_{\sigma}}$ is not tight.
\end{enumerate}
\end{prop}

We note here that in contrast to Proposition  \ref{Prop:GS-KL-limdist}, \ref{Prop:GS-chisq-limdist}, and \ref{Prop:GS-helsq-limdist}, the limit distribution for smoothed TV distance in the one- and two-sample alternative is not Gaussian. This is a consequence of the non-linearity of the first-order Hadamard derivative (see \eqref{eq:HadderTVdist}). 
\begin{remark}[Upper bound for \eqref{ASSUM:TV_null}]
To establish that  \eqref{ASSUM:TV_null}  holds if $\mu$ has finite $(2d + \epsilon)$ moments for some $\epsilon > 0$, we  show in the proof that for $X \sim \mu$,
\begin{align}
\int_{\RR^{d}} \sqrt{\mathsf{Var}_{\mu}(\varphi_{\sigma}(x-\cdot))} dx & \le 8^{d/2} + \frac{2^{d/2+1}}{\sigma^{d}\Gamma(d/2)} \int_{0}^{\infty}t^{d-1}\sqrt{\PP(\|X\| > t)} dt,
\label{eq: TV moment bound}
\end{align}
where $\Gamma$ denotes the Gamma function. The RHS is then finite by Markov's inequality provided $\E[\| X \|^{2d+\epsilon}] < \infty$.
\end{remark}
\begin{proof}
    We apply Theorem \ref{Thm:TVdiv-limdist} with $\rho=\lambda$.

\medskip
\noindent\textbf{Part $\bm{(i)}$:}
Observe that
\eqref{eq:TV-GS-onesamp-alt} is a direct consequence of \eqref{eq:TV-onesample-alt}, provided that
\begin{align}
 n^{1/2} \left(\hat\mu_n*\gamma_{\sigma}-\mu*\gamma_{\sigma}\right) \trightarrow{w} G_{\mu,\sigma}\quad \mbox{ in }L^1(\RR^d).\notag
\end{align}
Let $Z_i(x)=\varphi_{\sigma}(x-X_{i}) - \mu*\varphi_{\sigma}(x)$ and 
 $Z(x)=\varphi_{\sigma}(x-X)-\mu*\varphi_{\sigma}(x)$, where $X \sim \mu$. Note that $\norm{Z}_1 \leq 2$ a.s. and $\sqrt{n}\big(\hat\mu_n*\gamma_{\sigma}-\mu*\gamma_{\sigma}\big)=(1/\sqrt{n})\sum_{i=1}^n Z_i$. Since $\PP(\norm{Z}_{1} \geq t)=0$ for $ t>2$, Theorem \ref{Thm:CLT-in-Lp}(ii) implies that the weak convergence above holds if
\begin{align}
    \int_{\RR^d} \Big(\EE \big[\abs{Z(x)}^2  \big]\Big)^{\frac 12}dx= \int_{\RR^d} \Big(\mathsf{Var}_{\mu}\big(\varphi_{\sigma}(x-\cdot)\big)\Big)^{\frac 12}dx <\infty.\notag
\end{align}
The penultimate and final claim in Proposition \ref{Prop:GS-TV-limdist}$(i)$ come from \cite[Lemma 1]{Goldfeld-Kato-2020} and \cite[Proposition 1]{Goldfeld-Kato-2020}, respectively, whose proofs we repeat here for completeness. To prove the former, note that for $X \sim \mu$, one has
\begin{equation}
\mathsf{Var}_{\mu}\big(\varphi_{\sigma}(x-\cdot)\big) 
\le \E[\varphi_{\sigma}^{2}(x - X)] =\frac{1}{(2\pi \sigma^{2})^{d}} \int_{\RR^{d}} e^{-\|x-y\|^{2}/\sigma^{2}} d \mu(y).
\label{eq: TV moment 0}
\end{equation}
Splitting the integral over $\RR^{d}$ into $\|y\| \le \|x\|/2$ and $\|y\| > \|x\|/2$, we further obtain 
\begin{equation}
\int_{\RR^{d}} e^{-\|x-y\|^{2}/\sigma^{2}} d \mu(y)  \le \int_{\|y\| \le \|x\|/2} e^{-\|x-y\|^{2}/\sigma^{2}} d \mu(y)  + \PP \big(\|X\| > \|x\|/2\big). 
\label{eq: TV moment 1}
\end{equation}
Changing to polar coordinates leads to 
\begin{align}
&\int_{\RR^{d}} \sqrt{\PP \big(\|X\| > \|x\|/2\big)} dx =\frac{2^{d+1}\pi^{d/2}}{\Gamma(d/2)} \int_{0}^{\infty} t^{d-1}\sqrt{\PP\big(\|X\| > t\big)} d t. 
\label{eq: TV moment 2}
\end{align}
Next, using $\|x-y\|^{2} \ge \|x\|^{2}/2-\|y\|^{2}$,  we have 
\begin{align}
\int_{\|y\| \le \|x\|/2} e^{-\|x-y\|^{2}/\sigma^{2}} d \mu(y) \le e^{-\|x\|^{2}/(4\sigma^{2})} \int_{\|y\| \le \|x\|/2} d \mu(y)\le e^{-\|x\|^{2}/(4\sigma^{2})},
\label{eq: TV moment 3}
\end{align}
and the square root of the RHS integrates to $(16 \pi \sigma^{2})^{d/2}$. 
Combining (\ref{eq: TV moment 0})--(\ref{eq: TV moment 3}), we obtain inequality (\ref{eq: TV moment bound}). Finally, if $\mu$ has finite $(2d+\epsilon)$ moments, then by Markov's inequality
\[
t^{d-1} \sqrt{\PP\big(\|X\| > t\big)} \le t^{d-1} \wedge \sqrt{\EE[\|X\|^{2d+\epsilon}]}t^{-1-(\epsilon/2)} .
\]
The RHS is integrable on $[0,\infty)$, thus showing that \eqref{ASSUM:TV_null} (and consequently \eqref{eq:TV-GS-onesamp-alt})  holds under the assumption that $\mu$ has a finite $(2d + \epsilon)$-th moment, for some $\epsilon > 0$.

\medskip

The proof of the final claim is divided into two steps. We first show that if $\big(\sqrt{n}\tv{\hat\mu_n*\varphi_{\sigma}}{\mu*\varphi_{\sigma}}\big)_{n \in \NN}$ is tight, then its first moment is uniformly bounded for all $n$. Then we prove that under this uniform boundedness, Condition \eqref{ASSUM:TV_null} holds. 

For the first step, define $S_{i}\mspace{-3 mu} =\mspace{-3 mu} \sum_{j=1}^{i} \mspace{-3 mu}\big(\varphi_{\sigma}(x-X_{j}) - \mu*\varphi_{\sigma}(x)\mspace{-1 mu}\big)$ for $1\mspace{-2 mu} \leq i \mspace{-2 mu} \leq n$, and note that $\sqrt{n}\tv{\hat\mu_n*\varphi_{\sigma}}{\mu*\varphi_{\sigma}} $ $= \| S_{n}/\sqrt{n} \|_{1}/2$. We want to show that if $\| S_{n}/\sqrt{n} \|_{1}/2$ is tight, then 
\begin{align}
\sup_n \E\big[\|S_{n}/\sqrt{n} \|_1\big]=\sup_n\sqrt{n}\EE\big[\tv{\hat\mu_n*\varphi_{\sigma}}{\mu*\varphi_{\sigma}}\big]<\infty.    \notag
\end{align}
 By Hoffmann-J{\o}rgensen's inequality (see \cite[Proposition 6.8]{LT-1991}), we have 
\[
\E\big[\| S_{n} \|_{1}\big] \lesssim \E\left [\max_{1 \le i \le n} \big\| Z_i\big\|_{1}\right] + t_{n,0},
\]
where $t_{n,0} = \inf \big\{ t > 0 : \PP\big( \max_{1 \leq i \leq n}\| S_i \|_{1} > t\big) \leq 1/8\big\}$. The first term on the RHS is bounded by $2$. In addition, by Montgomery-Smith's inequality \cite[Corollary 4]{M-Smith-93}, there exists a universal constant $c$ such that 
\[
t_{n,0} \le \inf \big\{ t > 0 : \PP\big( \| S_{n} \|_{1} > c\,t\big) \le  c \big\}.
\] 
Thus, if $\| S_{n}/\sqrt{n} \|_{1}/2$ is tight (uniformly for all $n$), then $\sup_{n}t_{n,0}/\sqrt{n}$ $ < \infty$, which implies $\sup_{n} \E\big[\|S_{n}/\sqrt{n} \|_1\big] < \infty$, as desired. 

Next, we prove that the uniform boundedness of  $\sqrt{n}\EE\big[\tv{\hat\mu_n*\varphi_{\sigma}}{\mu*\varphi_{\sigma}}\big]$ implies Condition \eqref{ASSUM:TV_null} holds.  Let $k$ be any positive integer. With $Z_i(x)=\varphi_{\sigma}(x-X_{i}) - \mu*\varphi_{\sigma}(x)$ and $\overline{Z}_n=(1/n)\sum_{i=1}^n Z_i$,  Fubini's theorem yields
\begin{align}
\sqrt{n}\,\EE\Big[\tv{\hat\mu_n*\varphi_{\sigma}}{\mu*\varphi_{\sigma}}\Big] &\geq \frac{1}{2} \int_{\RR^{d}} \EE\big[\big(\sqrt{n}\big| \overline{Z}_{n}(x) - \mu*\varphi_{\sigma}(x)\big|\wedge\, k\,\big] dx. \notag
\end{align}
Since $|Z_i(x)| \leq (2\pi \sigma^{2})^{-d/2}$, the CLT implies that for any $x \in \RR^{d}$,
\begin{equation}
 \lim_{n\rightarrow \infty} \E\Big[\big(\sqrt{n}\big| \overline{Z}_{n}(x)- \mu*\varphi_{\sigma}(x)\big|\big) \wedge k\Big] = \E\Big[\big|G_{\mu,\sigma}(x)\big| \wedge k\Big]. \notag
\end{equation}
 Indeed, this follows from the CLT, i.e., $\sqrt{n}\big(\overline{Z}_{n}(x)- \mu*\varphi_{\sigma}(x)\big|\big) \trightarrow{d} N(0,\sigma_x^2)$ with $\sigma_x^2=\mathsf{Var}_{\mu}\big(\varphi_{\sigma}(x-\cdot)\big)$, 
and  the definition of weak convergence 
since $y \mapsto |y| \wedge k$ is bounded (by $k$) and (Lipschitz) continuous.  
Together with Fatou's lemma, we have
\[
\liminf_{n\rightarrow \infty} \sqrt{n}\,\EE\big[\tv{\hat\mu_n*\varphi_{\sigma}}{\mu*\varphi_{\sigma}}\big] \ge
\frac{1}{2} \int_{\RR^{d}}\E[|G_{\mu,\sigma}(x)| \wedge k] \, dx.
\]
Taking $k \to \infty$, we conclude by monotone convergence theorem that 
\begin{align}
 \liminf_{n\rightarrow \infty}\sqrt{n}\E\Big[\tv{\hat\mu_n*\varphi_{\sigma}}{\mu*\varphi_{\sigma}}\Big] &\geq \frac{1}{2} \int_{\RR^{d}} \E\big[\big|G_{\mu,\sigma}(x)\big|\big] \, dx  = \frac{1}{\sqrt{2\pi}}\int_{\RR^{d}} \Big(\mathsf{Var}_{\mu}\big(\varphi_{\sigma}(x-\cdot)\big)\Big)^{\frac 12}\, dx, \notag
\end{align}
where the second equality is because $\E[|W|] = \sqrt{2\E[W^{2}]/\pi}$ for a centered Gaussian variable $W$. This completes the proof of Part $(i)$.

\medskip
\noindent\textbf{Part $\bm{(ii)}$:}
The claim follows from \eqref{eq:TV-twosample-alt}, provided that
\begin{align}
 n^{1/2} \left(\hat\mu_n*\varphi_{\sigma}-\hat\nu_n*\varphi_{\sigma}-\mu*\gamma_{\sigma}+\nu*\gamma_{\sigma}\right) \trightarrow{w} G_{\mu,\sigma}-G_{\nu,\sigma}\quad \mbox{ in }L^1(\RR^d)\notag,
\end{align}
where $G_{\mu,\sigma}$ and $G_{\nu,\sigma}$ are $L^1(\RR^d)$-valued Gaussian random variables. Since $L^1(\RR^d)$ is Polish, arguments similar to those in the proof of Proposition \ref{Prop:GS-KL-limdist}$(iii)$, imply that the above weak convergence holds if
\begin{align}
 n^{1/2} \left(\hat\mu_n*\varphi_{\sigma}-\mu*\gamma_{\sigma}\right) \trightarrow{w} G_{\mu,\sigma}\quad \mbox{ and }\quad \left(\hat\nu_n*\varphi_{\sigma}-\nu*\gamma_{\sigma}\right) \trightarrow{w} G_{\nu,\sigma}\quad \mbox{ in }L^1(\RR^d).\notag
\end{align}
The latter holds under \eqref{ASSUM:TV_alt-2samp} by Theorem \ref{Thm:CLT-in-Lp}(ii) via similar arguments to those presented in the proof of Part $(i)$ above. The proof of the converse claim is again analogous to that in Part $(i)$ with $S_n=\sum_{i=1}^n\varphi_{\sigma}(\cdot-X_{i}) - \varphi_{\sigma}(\cdot-Y_{i})$; details are omitted. This concludes the proof.

\end{proof}

\end{appendices}
\bibliographystyle{IEEEtran}
\bibliography{ref}

\begin{thebibliography}{10}
\providecommand{\url}[1]{#1}
\csname url@samestyle\endcsname
\providecommand{\newblock}{\relax}
\providecommand{\bibinfo}[2]{#2}
\providecommand{\BIBentrySTDinterwordspacing}{\spaceskip=0pt\relax}
\providecommand{\BIBentryALTinterwordstretchfactor}{4}
\providecommand{\BIBentryALTinterwordspacing}{\spaceskip=\fontdimen2\font plus
\BIBentryALTinterwordstretchfactor\fontdimen3\font minus
  \fontdimen4\font\relax}
\providecommand{\BIBforeignlanguage}[2]{{%
\expandafter\ifx\csname l@#1\endcsname\relax
\typeout{** WARNING: IEEEtran.bst: No hyphenation pattern has been}%
\typeout{** loaded for the language `#1'. Using the pattern for}%
\typeout{** the default language instead.}%
\else
\language=\csname l@#1\endcsname
\fi
#2}}
\providecommand{\BIBdecl}{\relax}
\BIBdecl

\bibitem{Goldfeld-Kato-2020}
Z.~Goldfeld and K.~Kato, ``Limit distributions for smooth total variation and
  $\chi^2$-divergence in high dimensions,'' in \emph{Proceedings of the 2020
  IEEE International Symposium on Information Theory (ISIT)}, 2020, pp.
  2640--2645.

\bibitem{SGK-ISIT-2023}
S.~Sreekumar, Z.~Goldfeld, and K.~Kato, ``Limit distribution theory for {KL}
  divergence and applications to auditing differential privacy,'' in \emph{Proceedings of the 2023
  IEEE International Symposium on Information Theory (ISIT)}, 2023, pp.
  2607--2612.

\bibitem{csiszar1967information}
I.~Csisz{\'a}r, ``Information-type measures of difference of probability
  distributions and indirect observation,'' \emph{Studia Scientiarum
  Mathematicarum Hungarica}, vol.~2, pp. 229--318, Jan. 1967.

\bibitem{Ali-Silvey-1966}
S.~M. Ali and S.~D. Silvey, ``A general class of coefficients of divergence of
  one distribution from another,'' \emph{Journal of the Royal Statistical
  Society: Series B (Methodological)}, vol.~28, no.~1, pp. 131--142, 1966.

\bibitem{Renyi-60}
A.~R{\'e}nyi, ``On measures of entropy and information,'' in \emph{Proceedings
  of the Fourth Berkeley Symposium on Mathematical Statistics and Probability},
  vol.~1.\hskip 1em plus 0.5em minus 0.4em\relax Berkely: University of
  California Press, 1961, pp. 547--561.

\bibitem{vanEvren_Reyni_Div2014}
T.~van Erven and P.~Harremo{\"e}s, ``{R}{\'e}nyi divergence and
  {Kullback-Leibler} divergence,'' \emph{IEEE Transactions on Information
  Theory}, vol.~60, no.~7, pp. 3797--3820, Jul. 2014.

\bibitem{zolotarev1983probability}
V.~M. Zolotarev, ``Probability metrics,'' \emph{Teoriya Veroyatnostei i ee
  Primeneniya}, vol.~28, no.~2, pp. 264--287, 1983.

\bibitem{muller1997integral}
A.~M{\"u}ller, ``Integral probability metrics and their generating classes of
  functions,'' \emph{Advances in Applied Probability}, vol.~29, no.~2, pp.
  429--443, Jun. 1997.

\bibitem{villani2008optimal}
C.~Villani, \emph{Optimal {T}ransport: Old and {N}ew}.\hskip 1em plus 0.5em
  minus 0.4em\relax Springer Science \& Business Media, 2008, vol. 338.

\bibitem{santambrogio2015}
F.~Santambrogio, \emph{Optimal Transport for Applied Mathematicians}.\hskip 1em
  plus 0.5em minus 0.4em\relax Birkh\"{a}user, 2015.

\bibitem{CoverThomas}
T.~M. Cover and J.~A. Thomas, \emph{Elements of Information Theory}.\hskip 1em
  plus 0.5em minus 0.4em\relax NewYork: Wiley, 1991.

\bibitem{Csiszar-Shields-2004}
I.~Csiszár and P.~C. Shields, ``Information theory and statistics: A
  tutorial,'' \emph{Foundations and Trends in Communications and Information
  Theory}, vol.~1, no.~4, pp. 417--528, 2004.

\bibitem{kingma2013auto}
D.~P. Kingma and M.~Welling, ``Auto-encoding variational {B}ayes,'' in
  \emph{Proceedings of the International Conference on Learning
  Representations}, Banff, Canada, Apr. 2014.

\bibitem{nowozin2016f}
S.~Nowozin, B.~Cseke, and R.~Tomioka, ``{$f$-GAN: T}raining generative neural
  samplers using variational divergence minimization,'' in \emph{Proceedings of
  Advances in Neural Information Processing Systems}, vol.~29, Barcelona,
  Spain, Dec. 2016, pp. 271--279.

\bibitem{arjovsky2017wasserstein}
M.~Arjovsky, S.~Chintala, and L.~Bottou, ``Wasserstein generative adversarial
  networks,'' in \emph{Proceedings of the 34th International Conference on
  Machine Learning}, Sydney, Australia, Aug. 2017, pp. 214--223.

\bibitem{tolstikhin2018wasserstein}
I.~Tolstikhin, O.~Bousquet, S.~Gelly, and B.~Sch{\"o}lkopf, ``Wasserstein
  auto-encoders,'' in \emph{Proceedings of the 6th International Conference on
  Learning Representations}, Vancouver, Canada, Apr.-May 2018.

\bibitem{Goldfeld2020limit_wass}
Z.~Goldfeld, K.~Greenewald, and K.~Kato, ``Asymptotic guarantees for generative
  modeling based on the smooth {W}asserstein distance,'' in \emph{Proceedings
  of Advances in Neural Information Processing Systems}, vol.~33, Dec. 2020.

\bibitem{Kac-1955}
M.~Kac, J.~Kiefer, and J.~Wolfowitz, ``{On tests of normality and other tests
  of goodness of fit based on distance methods},'' \emph{The Annals of
  Mathematical Statistics}, vol.~26, no.~2, pp. 189--211, Jun. 1955.

\bibitem{QZhang-2018}
Q.~Zhang, S.~Filippi, A.~Gretton, and D.~Sejdinovic, ``Large-scale kernel
  methods for independence testing,'' \emph{Statistics and Computing}, vol.~28,
  pp. 113--130, Jan. 2018.

\bibitem{Hallin-2021}
M.~Hallin, G.~Mordant, and J.~Segers, ``{Multivariate goodness-of-fit tests
  based on Wasserstein distance},'' \emph{Electronic Journal of Statistics},
  vol.~15, no.~1, pp. 1328--1371, Mar. 2021.

\bibitem{Afgani-2008}
M.~Afgani, S.~Sinanovic, and H.~Haas, ``Anomaly detection using the
  {Kullback-Leibler} divergence metric,'' in \emph{Proceedings of the 2008
  First International Symposium on Applied Sciences on Biomedical and
  Communication Technologies}, Oct. 2008, pp. 1--5.

\bibitem{Tajer-2011}
J.~Tajer, A.~Makke, O.~Salem, and A.~Mehaoua, ``A comparison between divergence
  measures for network anomaly detection,'' in \emph{Proceedings of the 7th
  International Conference on Network and Services Management}, Laxenburg, AUT,
  2011, p. 515–519.

\bibitem{Shapiro-1990}
A.~Shapiro, ``On concepts of directional differentiability,'' \emph{Journal of
  Optimization Theory and Applications}, vol.~66, no.~3, p. 477–487, Sep.
  1990.

\bibitem{Romisch-2004}
W.~R\"{o}misch, \emph{Delta Method, Infinite Dimensional}.\hskip 1em plus 0.5em
  minus 0.4em\relax John Wiley \& Sons, 2004.

\bibitem{AVDV-book}
A.~W. van~der Vaart and J.~A. Wellner, \emph{Weak Convergence and Empirical
  Processes}.\hskip 1em plus 0.5em minus 0.4em\relax Springer, New York, 1996.

\bibitem{ding2018detecting}
Z.~Ding, Y.~Wang, G.~Wang, D.~Zhang, and D.~Kifer, ``Detecting violations of
  differential privacy,'' in \emph{Proceedings of ACM SIGSAC Conference on
  Computer and Communications Security}, Oct. 2018, pp. 475--489.

\bibitem{jagielski2020auditing}
M.~Jagielski, J.~Ullman, and A.~Oprea, ``Auditing differentially private
  machine learning: How private is private {SGD}?'' in \emph{Proceedings of
  Advances in Neural Information Processing Systems}, vol.~33, 2020, pp.
  22\,205--22\,216.

\bibitem{domingo2022auditing}
C.~Domingo-Enrich and Y.~Mroueh, ``Auditing differential privacy in high
  dimensions with the kernel quantum {R}\'{e}nyi divergence,''
  \emph{arXiv:2205.13941}, 2022.

\bibitem{polyanskiy2022codingbook}
Y.~Polyanskiy and Y.~Wu, \emph{Information Theory: From Coding to
  Learning}.\hskip 1em plus 0.5em minus 0.4em\relax Cambridge University Press,
  2012.

\bibitem{Wang-2005}
Q.~Wang, S.~R. {Kulkarni}, and S.~{Verdu}, ``Divergence estimation of
  continuous distributions based on data-dependent partitions,'' \emph{IEEE
  Transactions on Information Theory}, vol.~51, no.~9, pp. 3064--3074, Sep.
  2005.

\bibitem{Perez-2008}
F.~{Perez-Cruz}, ``Kullback-{L}eibler divergence estimation of continuous
  distributions,'' in \emph{Proceedings of the 2008 IEEE International
  Symposium on Information Theory}, Toronto, ON, Canada, Jul. 2008, pp.
  1666--1670.

\bibitem{kandasamy2015nonparametric}
K.~Kandasamy, A.~Krishnamurthy, B.~Poczos, L.~Wasserman, and J.~M. Robins,
  ``Nonparametric {von Mises} estimators for entropies, divergences and mutual
  informations,'' in \emph{Proceedings of Advances in Neural Information
  Processing Systems}, vol.~28, Montr{\'e}al, Canada, Dec. 2015, pp. 397--405.

\bibitem{Singh-Poczos-2016}
S.~Singh and B.~P\'{o}czos, ``Finite-sample analysis of fixed-k nearest
  neighbor density functional estimators,'' in \emph{Proceedings of Advances in
  Neural Information Processing Systems}, vol.~29, Barcelona, Spain, Dec. 2016,
  pp. 1225--1233.

\bibitem{Noshad-2017}
M.~Noshad, K.~R. Moon, S.~Y. Sekeh, and A.~O. Hero, ``Direct estimation of
  information divergence using nearest neighbor ratios,'' in \emph{Proceedings
  of the 2017 IEEE International Symposium on Information Theory}, Jun. 2017,
  pp. 903--907.

\bibitem{belghazi2018}
M.~I. Belghazi, A.~Baratin, S.~Rajeshwar, S.~Ozair, Y.~Bengio, A.~Courville,
  and D.~Hjelm, ``Mutual information neural estimation,'' in \emph{Proceedings
  of the 35th International Conference on Machine Learning}, vol.~80, Stockholm
  Sweden, Jul. 2018, pp. 531--540.

\bibitem{Moon-2018}
K.~R. Moon, K.~Sricharan, K.~Greenewald, and A.~O. Hero, ``Ensemble estimation
  of information divergence,'' \emph{Entropy}, vol.~20, no.~8, Aug. 2018.

\bibitem{Nguyen-2010}
X.~{Nguyen}, M.~J. {Wainwright}, and M.~I. {Jordan}, ``Estimating divergence
  functionals and the likelihood ratio by convex risk minimization,''
  \emph{IEEE Transactions on Information Theory}, vol.~56, no.~11, pp.
  5847--5861, Oct. 2010.

\bibitem{berrett2019efficient}
T.~B. Berrett, R.~J. Samworth, and M.~Yuan, ``Efficient multivariate entropy
  estimation via $k$-nearest neighbour distances,'' \emph{The Annals of
  Statistics}, vol.~47, no.~1, pp. 288--318, Feb. 2019.

\bibitem{berrett2019-efficientfunctional}
T.~B. Berrett and R.~J. Samworth, ``Efficient two-sample functional estimation
  and the super-oracle phenomenon,'' \emph{arXiv:1904.09347}, 2019.

\bibitem{Yanjun-2020}
Y.~Han, J.~Jiao, T.~Weissman, and Y.~Wu, ``{Optimal rates of entropy estimation
  over Lipschitz balls},'' \emph{The Annals of Statistics}, vol.~48, no.~6, pp.
  3228--3250, Dec. 2020.

\bibitem{SS-2021-aistats}
S.~Sreekumar, Z.~Zhang, and Z.~Goldfeld, ``Non-asymptotic performance
  guarantees for neural estimation of f-divergences,'' in \emph{AISTATS}, 2021.

\bibitem{sreekumar2021neural}
S.~Sreekumar and Z.~Goldfeld, ``Neural estimation of statistical divergences,''
  \emph{Journal of Machine Learning Research}, vol.~23, no. 126, pp. 1--75,
  2022.

\bibitem{SALICRU-1994}
M.~Salicru, D.~Morales, M.~L. Menendez, and L.~Pardo, ``On the applications of
  divergence type measures in testing statistical hypotheses,'' \emph{Journal
  of Multivariate Analysis}, vol.~51, no.~2, pp. 372--391, 1994.

\bibitem{Moon-2014}
K.~Moon and A.~Hero, ``Multivariate $f$-divergence estimation with
  confidence,'' in \emph{Proceedings of Advances in Neural Information
  Processing Systems}, vol.~27, Montr{\'e}al, Canada, 2014, pp. 2420--2428.

\bibitem{Antos-Ioannis-2001}
A.~Antos and I.~Kontoyiannis, ``Convergence properties of functional estimates
  for discrete distributions,'' \emph{Random Structures \& Algorithms},
  vol.~19, no. 3-4, pp. 163--193, 2001.

\bibitem{Ioannis-Skoularidou-2016}
I.~Kontoyiannis and M.~Skoularidou, ``Estimating the directed information and
  testing for causality,'' \emph{IEEE Transactions on Information Theory},
  vol.~62, no.~11, pp. 6053--6067, 2016.

\bibitem{Barrio-1999}
E.~del Barrio, E.~Gin{\'e}, and C.~Matr{\'a}n, ``Central limit theorems for the
  {W}asserstein distance between the empirical and the true distributions,''
  \emph{The Annals of Probability}, vol.~27, no.~2, pp. 1009--1071, 1999.

\bibitem{DBGU-2005}
E.~del Barrio, E.~Gin{\'e}, and F.~Utzet, ``{Asymptotics for $L_2$ functionals
  of the empirical quantile process, with applications to tests of fit based on
  weighted Wasserstein distances},'' \emph{Bernoulli}, vol.~11, no.~1, pp.
  131--189, Jan. 2005.

\bibitem{Sommerfeld2016Inference}
M.~Sommerfeld and A.~Munk, ``Inference for empirical {W}asserstein distances on
  finite spaces,'' \emph{Journal of the Royal Statistical Society: Series B
  (Statistical Methodology)}, vol.~80, 2016.

\bibitem{del-barrio-loubes-2019}
E.~del Barrio and J.~Loubes, ``{Central limit theorems for empirical
  transportation cost in general dimension},'' \emph{The Annals of
  Probability}, vol.~47, no.~2, pp. 926--951, 2019.

\bibitem{Tameling-2019}
C.~Tameling, M.~Sommerfeld, and A.~Munk, ``{Empirical optimal transport on
  countable metric spaces: Distributional limits and statistical
  applications},'' \emph{The Annals of Applied Probability}, vol.~29, no.~5,
  pp. 2744 -- 2781, 2019.

\bibitem{DBSL-21}
E.~del Barrio, A.~González-Sanz, and J.~Loubes, ``Central limit theorems for
  general transportation costs,'' \emph{arXiv:2102.06379}, 2021.

\bibitem{delBarrio-2022}
E.~del Barrio, A.~G. Sanz, J.-M. Loubes, and J.~Niles-Weed, ``An improved
  central limit theorem and fast convergence rates for entropic transportation
  costs,'' \emph{SIAM Journal on Mathematics of Data Science}, vol.~5, no.~3,
  pp. 639--669, 2023.

\bibitem{MBWW-2022}
T.~Manole, S.~Balakrishnan, J.~Niles-Weed, and L.~Wasserman, ``Plugin
  estimation of smooth optimal transport maps,'' \emph{arXiv:2107.12364}, 2021.

\bibitem{HKSM-22}
S.~Hundrieser, M.~Klatt, T.~Staudt, and A.~Munk, ``A unifying approach to
  distributional limits for empirical optimal transport,''
  \emph{arXiv:2202.12790}, 2022.

\bibitem{Bigot-2019}
J.~Bigot, E.~Cazelles, and N.~Papadakis, ``{Central limit theorems for
  entropy-regularized optimal transport on finite spaces and statistical
  applications},'' \emph{Electronic Journal of Statistics}, vol.~13, no.~2, pp.
  5120 -- 5150, 2019.

\bibitem{Mena-2019}
G.~Mena and J.~Niles-Weed, ``Statistical bounds for entropic optimal transport:
  sample complexity and the central limit theorem,'' in \emph{Proceedings of
  Advances in Neural Information Processing Systems}, vol.~32.\hskip 1em plus
  0.5em minus 0.4em\relax Curran Associates, Inc., 2019.

\bibitem{Klatt-2020}
M.~Klatt, C.~Tameling, and A.~Munk, ``Empirical regularized optimal transport:
  Statistical theory and applications,'' \emph{SIAM Journal on Mathematics of
  Data Science}, vol.~2, no.~2, pp. 419--443, 2020.

\bibitem{HLP-20}
Z.~Harchaoui, L.~Liu, and S.~Pal, ``Asymptotics of entropy-regularized optimal
  transport via chaos decomposition,'' \emph{arXiv: 2011.08963}, 2020.

\bibitem{GX-21}
F.~Gunsilius and Y.~Xu, ``Matching for causal effects via multimarginal
  unbalanced optimal transport,'' \emph{arXiv:2112.04398}, 2021.

\bibitem{Gonzalez-2022}
A.~Gonzalez-Sanz, J.~Loubes, and J.~Niles-Weed, ``Weak limits of entropy
  regularized optimal transport; potentials, plans and divergences,''
  \emph{arXiv:2207.07427}, 2022.

\bibitem{Sadhu-2022}
R.~Sadhu, Z.~Goldfeld, and K.~Kato, ``Limit distribution theory for the smooth
  1-{W}asserstein distance with applications,'' \emph{arXiv:2107.13494}, 2021.

\bibitem{GKNR-smooth-p-2022}
Z.~Goldfeld, K.~Kato, S.~Nietert, and G.~Rioux, ``Limit distribution theory for
  smooth $p$-{W}asserstein distances,'' \emph{arXiv:2203.00159}, 2022.

\bibitem{GKRS-2022-entropicOT}
Z.~Goldfeld, K.~Kato, G.~Rioux, and R.~Sadhu, ``Limit theorems for entropic
  optimal transport maps and the {S}inkhorn divergence,''
  \emph{arXiv:2207.08683}, 2022.

\bibitem{GKRS-22}
------, ``Statistical inference with regularized optimal transport,''
  \emph{arXiv:2205.04283}, 2022.

\bibitem{VanHandel-book}
R.~{van Handel}, \emph{Probability in High Dimension: Lecture Notes-Princeton
  University}.\hskip 1em plus 0.5em minus 0.4em\relax [Online]. Available:
  \url{https://web.math.princeton.edu/~rvan/APC550.pdf}, 2016.

\bibitem{Goldfeld-2020}
Z.~Goldfeld, K.~Greenewald, J.~Niles-Weed, and Y.~Polyanskiy, ``Convergence of
  smoothed empirical measures with applications to entropy estimation,''
  \emph{IEEE Transactions on Information Theory}, vol.~66, no.~7, pp.
  4368--4391, Jul. 2020.

\bibitem{VanderVaart-2008}
A.~van~der Vaart and J.~Zanten, ``Reproducing kernel {H}ilbert spaces of
  {G}aussian priors,'' \emph{Pushing the limits of contemporary statistics:
  contributions in honor of Jayanta K. Ghosh, Institute of Mathematical
  Statistics}, vol.~3, pp. 200--222, 2008.

\bibitem{Va1998}
A.~van~der Vaart, \emph{Asymptotic Statistics}.\hskip 1em plus 0.5em minus
  0.4em\relax Cambridge University Press, 1998.

\bibitem{DMKS-2006}
C.~Dwork, F.~McSherry, K.~Nissim, and A.~Smith, ``Calibrating noise to
  sensitivity in private data analysis,'' in \emph{Proceedings of the Third
  Theory of Cryptography Conference}.\hskip 1em plus 0.5em minus 0.4em\relax
  Springer Berlin Heidelberg, 2006, pp. 265--284.

\bibitem{DKMMN-2006}
C.~Dwork, K.~Kenthapadi, F.~McSherry, I.~Mironov, and M.~Naor, ``Our data,
  ourselves: Privacy via distributed noise generation,'' in \emph{Advances in
  Cryptology - EUROCRYPT 2006}.\hskip 1em plus 0.5em minus 0.4em\relax Springer
  Berlin Heidelberg, 2006, pp. 486--503.

\bibitem{WLF-2016}
Y.~Wang, J.~Lei, and S.~E. Fienberg, ``On-average {KL}-privacy and its
  equivalence to generalization for max-entropy mechanisms,'' in \emph{Privacy
  in Statistical Databases}.\hskip 1em plus 0.5em minus 0.4em\relax Springer
  International Publishing, 2016, pp. 121--134.

\bibitem{Dwork-Roth-2014}
C.~Dwork and A.~Roth, \emph{The Algorithmic Foundations of Differential
  Privacy}.\hskip 1em plus 0.5em minus 0.4em\relax Hanover, MA, USA: Now
  Publishers Inc., Aug. 2014, vol.~9.

\bibitem{Dagan-Kur-2022}
Y.~Dagan and G.~Kur, ``A bounded-noise mechanism for differential privacy,'' in
  \emph{Proceedings of Thirty Fifth Conference on Learning Theory}, vol.
  178.\hskip 1em plus 0.5em minus 0.4em\relax Proceedings of Machine Learning
  Research, Jul. 2022, pp. 625--661.

\bibitem{DL-1993}
R.~A. DeVore and G.~G. Lorentz, \emph{Constructive Approximation}.\hskip 1em
  plus 0.5em minus 0.4em\relax Springer Berlin, Heidelberg, 1993.

\bibitem{GIL-2013}
M.~Gil, F.~Alajaji, and T.~Linder, ``Rényi divergence measures for commonly
  used univariate continuous distributions,'' \emph{Information Sciences}, vol.
  249, pp. 124--131, 2013.

\bibitem{jin-2019-subgaussnorm}
C.~Jin, P.~Netrapalli, R.~Ge, S.~M. Kakade, and M.~I. Jordan, ``A short note on
  concentration inequalities for random vectors with subgaussian norm,''
  \emph{arXiv:1902.03736}, Feb. 2019.

\bibitem{AZYA-2022}
A.~Block, Z.~Jia, Y.~Polyanskiy, and A.~Rakhlin, ``Rate of convergence of the
  smoothed empirical {W}asserstein distance,'' \emph{arXiv:2205.02128}, 2022.

\bibitem{LT-1991}
M.~Ledoux and M.~Talagrand, \emph{Probability in Banach Spaces}.\hskip 1em plus
  0.5em minus 0.4em\relax Springer-Verlag Berlin Heidelberg, 1991.

\bibitem{Gine-Zinn-1990}
E.~Gin\'{e} and J.~Zinn, ``Bootstrapping general empirical measures,''
  \emph{The Annals of Probability}, vol.~18, no.~2, pp. 851--869, 1990.

\bibitem{Bickel-98}
P.~J. Bickel, C.~A.~J. Klaassen, Y.~Ritov, and J.~A. Wellner, \emph{Efficient
  and Adaptive Estimation for Semiparametric Models}.\hskip 1em plus 0.5em
  minus 0.4em\relax Springer New York, NY, 1998.

\bibitem{CCR-2020}
J.~C{\'a}rcamo, A.~Cuevas, and L.-A. Rodr{\'i}guez, ``{Directional
  differentiability for supremum-type functionals: Statistical applications},''
  \emph{Bernoulli}, vol.~26, no.~3, pp. 2143 -- 2175, 2020.

\bibitem{Boyd}
S.~Boyd and L.~Vandenberghe, \emph{Convex Optimization}.\hskip 1em plus 0.5em
  minus 0.4em\relax New-York: Cambridge University Press, 2004.

\bibitem{Nielsen-2013}
F.~Nielsen and R.~Nock, ``On the chi square and higher-order chi distances for
  approximating f-divergences,'' \emph{Signal Processing Letters, IEEE},
  vol.~21, Sep. 2013.

\bibitem{M-Smith-93}
S.~Montgomery-Smith, ``Comparison of sums of independent identically
  distributed random variables,'' \emph{Probability and Mathematical
  Statistics}, vol.~14, 11 1993.

\end{thebibliography}

\end{document}